\documentclass[12pt, reqno, twoside]{amsart}
\usepackage{amsfonts,amsthm,amsmath,amssymb,amscd}
\usepackage{mathabx,mathrsfs} 
\usepackage{mathtools}
\usepackage{graphicx}
\usepackage{indentfirst}
\usepackage[T1]{fontenc}
\usepackage{microtype}
\usepackage{libertine} %
\usepackage{enumitem}
\usepackage{cases}

\usepackage[colorlinks=true, citecolor=blue, linkcolor=blue, urlcolor=blue]{hyperref}
\date{}
\allowdisplaybreaks[2]
\numberwithin{equation}{section}

\usepackage[a4paper, margin=1.25in, bottom=1.1in]{geometry}
\usepackage{fancyhdr}

\fancypagestyle{firstpage}{
  \fancyhf{}
}

\makeatletter
\renewenvironment{abstract}{%
  \global\setbox\abstractbox=\vbox\bgroup
    \normalfont\small
    \list{}{\labelwidth\z@
      \leftmargin\z@ \rightmargin\z@
      \listparindent\normalparindent \itemindent\z@
      \parsep\z@ \@plus\p@
    }%
    \item[\hskip\labelsep\bfseries Abstract.] 
}{%
  \endlist\egroup
}

\renewcommand{\section}{\@startsection{section}{1}{\z@}%
  {1.5\linespacing\@plus\linespacing}{0.8\linespacing}%
  {\normalfont\large\bfseries}}

\renewcommand{\subsection}{\@startsection{subsection}{2}{\z@}%
  {.7\linespacing\@plus\linespacing}{.4\linespacing}%
  {\normalfont\normalsize\bfseries}}

\newcommand{\norm}[2]{\Arrowvert #1 \Arrowvert_{#2}}
\newcommand{\Lnorm}[1]{L^{#1}}
\newcommand{\abs}[1]{\left\lvert #1 \right\rvert}

\newcommand{\propref}[1]{ \hyperref[#1]{Proposition \ref*{#1}}}
\newcommand{\lemref}[1]{ \hyperref[#1]{Lemma \ref*{#1}}}
\newcommand{\defref}[1]{ \hyperref[#1]{Definition \ref*{#1}}}
\newcommand{\cororef}[1]{ \hyperref[#1]{Corollary \ref*{#1}}}
\newcommand{\rmkref}[1]{ \hyperref[#1]{Remark \ref*{#1}}}
\newcommand{\secref}[1]{ \hyperref[#1]{Section \ref*{#1}}}

\newcommand{\rhobar}{\bar{\rho}}

\newcommand{\mubar}{\bar{\mu}}

\newcommand{\Qfrak}{\mathfrak{Q}}

\newtheorem{theorem}{Theorem}[section]

\newtheorem{lemma}[theorem]{Lemma}
\newtheorem{proposition}[theorem]{Proposition}
\newtheorem{remark}{Remark}[section] 
\newtheorem{definition}{Definition}[section]

\begin{document}

\title[NS expanding solutions]{Expanding solutions to the compressible Navier--Stokes equations with degenerate viscosities in spherical symmetry: global existence and inviscid limit}
	
	\author{Shuying~Hu}
	\address[S. Hu]{The Institute of Mathematical Sciences, The Chinese University of Hong Kong, Hong Kong, China.}
	\email{syhu@link.cuhk.edu.hk}
	
	\author{Zhouping Xin}
	\address[Z. Xin]{The Institute of Mathematical Sciences, The Chinese University of Hong Kong, Hong Kong, China.}
	\email{zpxin@ims.cuhk.edu.hk}
	
	\author{Yuan~Yuan
    }
	\address[Y. Yuan]{School of Mathematical Sciences, South China Normal University, Guangzhou 510631, China. }
	\email{{yyuan2102@m.scnu.edu.cn}}

\begin{abstract}
	This paper is devoted to studying the strong solutions to the vacuum free boundary problem for the isentropic compressible Navier--Stokes equations with density-dependent viscosities under spherical symmetry, which models the motions of isentropic compressible viscous flows surrounded by vacuum. We construct a class of expanding global solutions when the initial data is a small perturbation of the expanding affine solutions for the adiabatic exponent $\gamma>1$ and the viscosity coefficients proportional to $\rho^{\gamma}$, with $\rho$ being the fluid density. In addition, when the viscosity coefficients tend to zero, the perturbed solutions are proved to converge to the solutions to the compressible Euler equations with an explicit converging rate of viscosity coefficients.
    
    This extends the results of Had\v{z}i\'c--Jang [Invent. Math. 214 (2018), 1205--1266] and Shkoller--Sideris [Arch. Ration. Mech. Anal. 234 (2019), 115--180] from the three-dimensional inviscid flows to the three-dimensional spherically symmetric viscous flows with a class of viscosities degenerate at vacuum states. 
    Due to the strong degeneracy of the viscosities, the ellipticity due to the viscosity is weak. Therefore, we apply normal derivative estimates to the equations, as in the works of Had\v{z}i\'c--Jang and Shkoller--Sideris. Nevertheless, for spherically symmetric flows, singularities arise at the origin, which lead to additional difficulties. Taking standard normal derivatives will produce some uncontrollable terms, thereby preventing the energy from closing. To handle these obstacles, we introduce a non-trivial differential operator and a proper test function and then establish uniform a priori estimates independent of viscosities, from which the global-in-time existence and inviscid limit are proved. This yields not only the global existence of the solutions to the isentropic Navier--Stokes equations with initial data which are generic small perturbations of the expanding affine solution, but also an alternative proof of the corresponding result for the inviscid Euler equations.
\end{abstract}



\maketitle

\tableofcontents
\section{Introduction}
	In this work, we investigate the global existence and the inviscid limit of an isentropic viscous gas surrounded by the vacuum via free boundary. The domain occupied by the gas is assumed to be open and simply connected, and is denoted by $\Omega(t) \subset \mathbb{R}^3 $. 	
	In $\Omega(t)$, the isentropic viscous gas flow is modeled by the following vacuum free boundary problem for the isentropic compressible Navier--Stokes system (CNS):
	\begin{equation} \label{CNS}
		\begin{cases}
			\rho_t+\text{div}(\rho\mathbf{u})=0 & \text{in}\quad \Omega(t),\\
			(\rho u)_t+\text{div}(\rho \mathbf{u}\otimes \mathbf{u})+\text{div}\mathfrak{S}=0 &  \text{in}\quad \Omega(t),\\
			\rho>0 &  \text{in}\quad \Omega(t),\\
			\rho=0 \quad \text{and}\quad \mathfrak{S}\cdot \mathbf{n}=\mathbf{0} &  \text{on}\quad \Gamma(t):=\partial \Omega(t),\\
			\mathcal{V}(\Gamma(t))=\mathbf{u}\cdot\mathbf{n},\\
			(\rho,\mathbf{u})=(\rho_0,\mathbf{u}_0) &  \text{on}\quad \Omega:=\Omega(0).\\
		\end{cases}
	\end{equation}
	Here $(\mathbf{x},t)\in \mathbb{R}^3\times [0,\infty),\rho,\mathbf{u},\mathfrak{S}$ denote respectively the space and time variable, density, velocity and stress tensor. $\Gamma(t), \mathcal{V}(\Gamma(t))$ and $\mathbf{n}$ represent respectively moving interface of fluids and vacuum states, normal velocity of $\Gamma(t)$ and exterior unit normal vector to $\Gamma(t)$. The stress tensor is given by
	$$\mathfrak{S}=p\mathbb{I}_3-\mu\left(\nabla \mathbf{u}+(\nabla\mathbf{u})^T\right)-\lambda\text{div}\mathbf{u}\mathbb{I}_3,$$ 
	where $\mathbb{I}_3$ is the $3\times 3$ identity matrix, $p$ is the pressure of the gas,
	$\mu$ and $2\mu+3\lambda$ are the shear viscosity, the bulk viscosity, respectively.
    We consider the polytropic gases for which $$p=A\rho^{\gamma},$$
	where $A>0$ is a constant, $\gamma>1$ is the adiabatic exponent. 
    In this paper the viscosities are assumed to be \emph{density-dependent} functions which \emph{vanish} on the vacuum boundary:
    \begin{equation} \label{vis-coeff0}
        \mu=\bar{\mu}\rho^{\delta}, ~ \lambda=\bar{\lambda}\rho^{\delta},
    \end{equation}
    where $\delta ,\bar{\mu},$ and $ \bar{\lambda} $ are all constants satisfying 
    \begin{equation} \label{vis-coeff}
     \delta>0,  \quad \text{and} \quad \bar{\mu} \ge 0, ~ 2\bar{\mu}+3\bar{\lambda} \ge 0.    
    \end{equation}
    Such a setting of the viscosities is motivated by the kinetic theory of gas dynamic: when the CNS is derived from the Boltzmann equations through the Chapman--Enskog expansion (\cite{Chapm.C1970}), $\mu, \lambda,$ and the heat conductivity coefficient  are power functions of the absolute temperature $\theta$; while for isentropic flow, this dependence is translated into the dependence of the viscosity on the density (\cite{Liu.X.Y1998}); see \cite{Xin.Z2021a} for a detailed explanation.

    It should be noted that due to the degeneracy of the viscosity \eqref{vis-coeff0}, the boundary condition $\mathfrak{S}\cdot \mathbf{n}=\mathbf{0}$, which characterizes the stress balance at the free boundary, is automatically satisfied on the boundary for smooth flows, since the stress tensor vanish there.

\subsection{Previous work}
    \subsubsection{Cauchy and initial-boundary value problems in presence of vacuum}
	The study of the compressible Navier--Stokes equations (CNS) has a vast history. For Cauchy and initial-boundary value problems where the initial density is strictly away from vacuum (i.e., $\rho \geq \underline{\rho} > 0$), the local well-posedness of classical solutions was established by \cite{Serri.1959, Itaya.1971, Tani.1977}. The work of \cite{Matsu.N1980, Matsu.N1983} demonstrated the global stability of equilibria for heat-conductive flows under small perturbations.

	The situation becomes markedly different when vacuum is present. For the Cauchy problems associated with the compressible Navier--Stokes equations with constant viscosity coefficients, the appearance of vacuum may trigger certain singular behaviors in solutions. These include the loss of continuous dependence on initial data for weak solutions (\cite{Hoff.S1991}), as well as the finite-time or even instantaneous blowup of classical solutions in inhomogeneous Sobolev spaces (\cite{Xin.1998,Li.W.X2019,Liu.Y2023,Li.X2024}). 
    
    The primary reason for the failure of continuous dependence on initial data for weak solutions is that the kinematic viscosity coefficient is independent of the density.  
    In light of this observation, the modified version of the compressible Navier--Stokes equations incorporating density-dependent viscosity coefficients, which is also motivated by the physical consideration as mentioned above, was proposed in \cite{Liu.X.Y1998}. They showed that for this CNS system with density-dependent viscosities, weak solutions can exist locally and be matched locally with the initial data in the presence of vacuum.
    This degenerate system has recently attracted considerable attention, leading to advances in both weak and strong solutions to its Cauchy problem with vacuum.  When $\mu$ and $\lambda$ satisfy the so-called BD relation, i.e., 
    $\lambda(\rho)=2(\rho \mu'(\rho)-\mu(\rho))$,
    Bresch et al. \cite{Bresc.D.L2003, Bresc.D2003} proposed a new mathematical entropy, now known as the BD entropy, and obtained the global existence of weak solutions with some additional drag or friction terms. Li and Xin \cite{Li.X2015}, and Vasseur and Yu \cite{Vasse.Y2016a}  
    proved the global existence of weak solutions without additional terms. When $\mu,\lambda$ take the form of \eqref{vis-coeff0}, which does not satisfy the above BD relation, the global well-posedness of regular solutions to the Cauchy problem for $\delta>1$ was established in \cite{Xin.Z2021a}; see also \cite{Li.P.Z2019, Duan.X.Z2023a} for more related results on local well-posedness of regular solutions for isentropic or non-isentropic cases. 
    We also refer interested readers to \cite{Kazhi.V1995,  Jiu.W.X2013,Huang.L2022} and the references therein for the global well-posedness of strong solutions for arbitrarily large initial data in the case that $\mu$ is constant and $\lambda(\rho)=a\rho^\beta$.
    
	\subsubsection{Vacuum free boundary problem and the physical vacuum condition}
    Over the past two decades, the vacuum free boundary problem (VFBP) of \eqref{CNS} has drawn considerable attention.
    Compared with the Cauchy problem, the VFBP involves equations for moving interfaces, enabling a better characterization of vacuum interface motion. Moreover, the fluid equations are considered only within a moving vacuum boundary, and the variables (density or velocity) do not need to connect smoothly to the vacuum.  
    In this paper, we consider the initial data that satisfy the \emph{physical vacuum condition}: 
	\begin{equation} \label{eq:pvac}
		-\infty < \nabla_{\mathbf{n}}(c_0^2) < 0 \quad \text{on } \Gamma(0),
	\end{equation}
	where $c_0 = \sqrt{P'(\rho_0)}$ is the initial sound speed, i.e., the initial sound speed $c_0$ is only $C^{1/2}$-H\"older continuous across the vacuum interface. 
    
	The physical vacuum condition \eqref{eq:pvac} was first proposed by \cite{Liu.1996} when studying inviscid flows with damping, and it arises naturally in many physical systems, such as self-similar solutions to the porous media equation(\cite{Liu.1996}), Lane--Emden solutions to gravitational gaseous stars (\cite{Jang.2010,Luo.X.Z2014}), and affine solutions to the compressible Euler equations (\cite{Sider.2017}).
	As pointed out by Liu in \cite{Liu.1996}, such a singularity of the density makes the standard hyperbolic method fail in establishing the local well-posedness theory for the inviscid flow. 
	The local well-posedness of compressible Euler equations under \eqref{eq:pvac} were first established by Jang and Masmoudi \cite{Jang.M2009, Jang.M2015} and Coutand, Lindblad, and Shkoller \cite{Couta.L.S2010, Couta.S2011, Couta.S2012} by different methods (hyperbolic-type energy estimates and degenerate parabolic regularizations, respectively) handling the degeneracy near the vacuum interface in weighted energy spaces. Luo, Xin and Zeng \cite{Luo.X.Z2014} proved the existence for 3D spherically symmetric motions with less regularities of the initial data, and without requiring the compatibility conditions of the derivatives vanishing at the center of symmetry. Additionally, a general uniqueness theorem of classical solutions was obtained by using the relative entropy method in \cite{Luo.X.Z2014}.
    Ifrim and Tataru \cite{Ifrim.T2024} demonstrated Hadamard-style local well-posedness and continuation criteria in the Eulerian setting for multi-dimensional spaces, applying to low-regularity solutions and general domains; see also \cite{Disco.I.T2022, Liu.L2024} for local theory in the Eulerian setting of the relativistic Euler equations and non-isentropic Euler equations, respectively.

	The breakthroughs on \emph{global} classical solutions for inviscid flows were made by Luo and Zeng \cite{Luo.Z2016} for the damped Euler equations in 1D, and by Had{\v z}i{\'c} and Jang \cite{Hadzi.J2018} for the Euler equations in 3D for $1<\gamma\leq 5/3$ without spherical symmetry. The initial data were imposed near the Barenblatt self-similar and affine solutions, respectively. Later, Shkoller and Sideris \cite{Shkol.S2019} addressed the $\gamma>5/3$ threshold left in \cite{Hadzi.J2018}, and thus establishing global existence for the shallow water equations when $\gamma=2$.  Zeng, in a series of works culminating in \cite{Zeng.2026}, finally settled the nonlinear stability of Barenblatt-type solutions for the 3D damped Euler equations without any symmetries. 
    Recently, Disconzi, Hu, and Luo \cite{Disco.H.L2026} obtained a class of global expanding solutions to 3D relativistic Euler equations with a physical vacuum boundary.
    One may refer to the local, global and the continued gravitational collapse theories for gaseous stars \cite{Gu.L2016,Hadzi.J2018a,Guo.H.J2021} and for non-isentropic flows \cite{Geng.L.W.X2019,Liu.L2024,Ricka.2021}; see \cite{Jang.M2012,Liu.L2024} for the discussions on the local well and ill-posedness of the VFBP for the compressible Euler equations without the physical vacuum singularity.

    Returning to the VFBP for viscous flows with constant viscosities, the global regularity and behavior of weak solutions near the interface, with the initial density connecting to vacuum very smoothly, were studied by Luo, Xin and Yang in \cite{Luo.X.Y2000}. 
    Zeng \cite{Zeng.2015} proved the global existence of smooth solutions to the one-dimensional CNS by just requiring the density decays algebraically to the vacuum boundary, where the physical vacuum singularity is included when $1<\gamma<3$. 
    The local well-posedness of strong solutions under the physical vacuum condition, and the nonlinear asymptotic sability of Lane--Emden solutions to the 3D compressible Navier--Stokes--Poisson (CNSP) equations in spherically motions were obtained by Jang \cite{Jang.2010} and Luo, Xin, and Zeng \cite{Luo.X.Z2016a}, respectively.
    Liu and Yuan \cite{Liu.Y2019} established the local-in-time existence and uniqueness of strong solutions to the full CNS in three-dimensional domains under the boundary condition that the density generally decays algebraically while the normal derivative of temperature does not decay near the vacuum boundary. The globally defined solutions around the self-similar solutions to CNS with zero bulk viscosity can be found in \cite{Liu.Y2019a,Liu.2019a}.
    In the works \cite{Gui.W.W2019,Chen.H.W.W2021}, the local well-posedness of VFBP for the 3D CNS under the physical vacuum condition by using conormal derivatives and no strong compatibility condition on the initial data was imposed.

	In contrast, the VFBP for the \textit{degenerate} compressible Navier--Stokes equations, where viscosities depend on density or temperature, presents additional challenges due to a double degeneracy in both the time evolution and the spatial dissipation within the momentum equation. For establishing the normal estimates, the appearance of degenerate dissipation terms makes div-curl type arguments in the inviscid case, as well as standard elliptic estimates in the constant viscosity case, not directly applicable here. As a result, the well-posedness theory in this setting is less developed. Most existing contributions have focused on the global existence of weak solutions; we refer the reader to \cite{Yang.Z2002,Duan.2011,Guo.L.X2012} and the references therein. 
    Recently, for the viscous Saint-Venant system (a specific degenerate CNS with $\gamma=2,~\delta=1$), \cite{Li.W.X2026, Li.W.X2025} proved local well-posedness of classical solutions under the physical vacuum condition. \cite{Xin.Z.Z2025} established global well-posedness of classical solutions to the viscous Saint-Venant system in 1D case with large initial data satisfying the physical vacuum or BD entropy condition, and then \cite{Chen.Z.Z2026} extended the theory to degenerate CNS with $\gamma>\frac{4}{3},~\mu=\mubar \rho,~\lambda=0$ in 2D or 3D spheically symmetry. We also refer readers to \cite{Ou.Z2015,Luo.X.Z2016} for global strong solutions to the VFBP for the degenerate CNS with gravity under the physical vacuum condition.

   Results on the vanishing viscosity limit for free boundary problems are relatively scarce. To the best of our knowledge, the first result of the vanishing viscosity limit for viscous surface waves is due to Masmoudi and Rousset \cite{Masmo.R2017}, where the inviscid limit for the incompressible Navier--Stokes equations with a free surface was justified on a uniform time interval. Mei, Wang, and Xin \cite{Mei.W.X2017} proved that on a time interval, solutions of the CNS with or without surface tension converge to those of the compressible Euler system as both viscosity and surface tension tend to zero. In above works, the density of the solution has a positive lower bound. For the VFBP, when studying the local existence of the compressible Euler or Euler--Poisson equations under the physical vacuum condition in \cite{Couta.S2011,Couta.S2012,Gu.L2016} which have been mentioned above, solutions are found in the limit as the \emph{artificial} viscosity 
   tends to zero. To date, the inviscid limit of the VFBP under physically motivated viscosities has not been studied.

    It is worth point out that one of the inherent challenges in establishing global existence of solutions to the VFBP for flows is controlling the motion of the fluid interface.
    In particular, in multi-dimensional cases, the interface may develop splash-type singularities (\cite{Castr.C.F.G.G2013, Couta.S2019}). It is therefore challenging to obtain global solutions in multi-dimensional settings. 
    Therefore, to establish global existence in multi-dimensional settings, it is natural to impose smallness assumptions on the initial data around prescribed expanding or steady motions.

\subsection{Affine solutions and motivations}

   This work is devoted to studying expanding global solutions of the VFBP for the CNS in the spherically symmetric setting,  
   where $\Omega(t)$ is a ball with the changing radius $R(t)$, 
   $$\rho(\mathbf{x},t)=\rho(r,t)\quad \text{and} \quad \mathbf{u}(\mathbf{x},t)=u(r,t)\mathbf{x}/r \quad \text{with}~~r=|\mathbf{x}|\in(0,R(t)).$$
   Then system \eqref{CNS} can be reduced to 
   \begin{equation} \label{CNS-sym}
   	\begin{cases}
   		(r^2 \rho)_t + ( r^2 \rho u)_r = 0 & r \in (0,R(t)),\\
   		\rho(u_t+uu_r)+p_r=(2\bar{\mu}+\bar{\lambda})\left[\rho^{\delta}\frac{(r^2u)_r}{r^2}\right]_r-4\bar{\mu}(\rho^{\delta})_r\frac{u}{r} & r \in (0,R(t)),\\
   		\rho>0 & r \in (0,R(t)),\\
   		\rho=0 \quad \text{and} \quad A\rho^\gamma-2\bar\mu\rho^\delta u_r-\bar\lambda\rho^\delta\frac{(r^2u)_r}{r^2}=0& r=R(t)),\\
   		\dot{R}(t)=u(R(t),t) \quad \text{with} \quad R(0)=R_0,~ u(0,t)=0,\\
   		(\rho,u)=(\rho_0,u_0) & t=0, r \in (0,R_0).\\
   	\end{cases}
   \end{equation}
   
   For the VFBP, the equations are usually reformulated as a fixed‑domain problem in Lagrangian coordinates, in which particle trajectories and boundary motions are clearly described by the flow map.
   Without abusing notations and for convenience, we use $x$ as the reference variable and define the flow map $r(x,t)$ by
   $$
   \partial_t r(x,t)=u(r(x,t),t)\quad \text{for } x\in (0,\bar{R}),~t>0.
   $$
   For an affine solution, the flow map $r_\alpha(x,t)$ takes the form $$r_\alpha(x,t)=\alpha(t)x  \quad \text{for } \alpha(t)>0.$$
   The affine motions of ideal fluids in the presence of physical vacuum were studied by Sideris \cite{Liu.1996} for the compressible Euler equations, whereas the earlier work by Liu \cite{Liu.1996} treated the damped case.
   The existence  of globally expanding solutions near such affine solutions have been established in \cite{Hadzi.J2018, Shkol.S2019}; notable results within the same framework can also be found for flows with damping in \cite{Luo.Z2016, Zeng.2026}, for flows with self-gravitation in \cite{Hadzi.J2018a}, for nonisentropic flows in \cite{Ricka.2021}, and for relativistic flows in \cite{Disco.H.L2026} and the references therein. In the viscous case, globally expanding solutions near the affine solutions are only obtained for spherically symmetric flows without bulk viscosity (\cite{Liu.2019a,Liu.Y2019a}).   
   In these studies, it is observed that \emph{the expanding nature of these affine solutions introduces a new stabilizing mechanism for the perturbed equation}, provides strong control over the Jacobian of the flow map, and helps prevent imploding singularity (\cite{Merle.R.R.S2022,Merle.R.R.S2022a}) and self-intersection of the free interface (\cite{Castr.C.F.G.G2013, Couta.S2019}).
   
   For the degenerate CNS, the affine solutions to \eqref{CNS} exist for the case $\delta=\gamma$, which has been shown in \cite{Li.2025}. 
   More specifically, denote $(\rho_{\alpha},u_{\alpha})$ as the affine solution with \emph{initial radius} $\bar{R}$, 
   ans set the density and velocity in Lagrangian coordinates as $f_{\alpha}(x,t)=\rho_\alpha(r(x,t),t)$, $v_\alpha(x,t)=u_\alpha(r(x,t),t)$. Then $\eqref{CNS-sym}_1$ implies $\partial_t (f_\alpha r_\alpha^2\partial_x r_\alpha)=0$. Now we can further consider the ansatz
   \begin{equation} \label{def-aff}
   	v_\alpha(x,t)=\alpha^{\prime}(t)x \quad \text{and}\quad f_\alpha(x,t)=\alpha^{-3}(t)\bar{\rho}(x).
   \end{equation}  
   Plugging the ansatz \eqref{def-aff} into $\eqref{CNS-sym}_2$, it yields that 
   \begin{equation} 
   	\frac{\bar{\rho} x \alpha''}{\alpha^2}+\frac{A\left(\bar{\rho}^{\gamma}\right)'}{\alpha^{3\gamma}}=(2\bar{\mu}+3\bar{\lambda})\frac{\left(\bar{\rho}^\delta\right)'}{\alpha^{3\delta}}\frac{\alpha'}{\alpha}.
   \end{equation}
   Therefore, by setting 
   $	\delta=\gamma   $
   	, the above equation can be reduced to two ordinary differential equations, namely 
   \begin{equation}
   	\label{eq-alpha1}
   	\begin{cases}
   		A\left(\bar{\rho}^\gamma \right)'=-c_1\bar{\rho}x,\quad \rhobar_{|_{x=\bar{R}}}=0,\\
   		\alpha''=c_1\alpha^{2-3\gamma}-(2\bar{\mu}+3\bar{\lambda})\frac{c_1}{A}\alpha^{1-3\gamma}\alpha',\\
   	\end{cases}
   \end{equation}
   for some constant $c_1>0$.  
   Hence, the affine solution to $\eqref{CNS-sym}$ is determined by the expanding coefficient $\alpha(t)$ and the profile $\bar{\rho}$ which solve \eqref{eq-alpha1}. In particular, given any fixed positive constant $c_1 \in (0,\infty) $, together with initial data
   \begin{equation} \label{def-aff-data1}
   	(\alpha, \alpha') \big|_{t=0} = ( \alpha_0, \alpha_1) \in (0,\infty)\times (-\infty,\infty),
   \end{equation}
   and initial radius $\bar{R}$, 
   system \eqref{eq-alpha1} is globally well-posed and there are constants $ c_2$ depending on $A, \gamma, c_1, \alpha_0, \alpha_1, $ and $2\bar\mu+3\bar\lambda$, such that 
   \begin{equation} \label{sup-alpha}
   	\lim_{t\to \infty}\alpha'(t)=c_2>0;
   \end{equation}
   see the details in \lemref{lem-alpha00}. With \eqref{sup-alpha}, it holds that $\alpha''(t)>0$ for large $t$. 
   Moreover, $\eqref{eq-alpha1}_1$ implies that the affine solution satisfies the vacuum free boundary condition for $c_1>0$.

   However, the analysis of solutions to the degenerate CNS,  \eqref{CNS-sym}, around these affine solutions remains open. In fact, $\eqref{eq-alpha1}_2$ clearly shows that the viscous terms will slow down the expansion of the fluid and affect asymptotic behaviors of the affine motions, and thus weaken the stabilizing effects of the affine solutions. Moreover, the analysis of the global existence and inviscid limit of the solutions to \eqref{CNS-sym} (for $\delta =\gamma$) presents the following \emph{important  challenges and values}.

Due to the strong degeneracy of the viscosity coefficients, \emph{the elliptic structure is weak} in two senses and leading to consequences for the estimates below.
  	\vspace{-1mm}
\begin{enumerate}
	\item The viscous operator loses uniform ellipticity on $\Omega$(t) (its minimum eigenvalue does not possess a positive lower bound). Note that in the Lagrangian coordinates, the coefficients of the equations degenerate only in the normal direction; consequently, taking tangential derivatives does not alter the structure of the equations and the boundary conditions. This makes such operations  the conventional first step toward improving the regularity of the solution, and the elliptic structure are applied to recover normal regularities.
	However, in our setting, if we follow this conventional approach, 
	we would only obtain strongly degenerate weighted energy estimates. This approach does not improve effectively the regularity of the variables via the Hardy inequality, since this would require very high-order tangential estimates, such as those established in the local well-posedness theory for the compressible Euler or Euler--Poisson equations  \cite{Jang.M2015,Couta.S2012,Luo.X.Z2014}.	
	\item  By the virtue of the expanding motions, in the reformulated equation (see \eqref{eq-Lag} below) the viscous term decays fastest among all terms in time, while the time evolution term exhibits the fastest growth. 
	For global-in-time well-posedness, we need to establish estimates that are uniform or integrable in time, which is fundamentally different from the local-in-time setting.
	Accordingly, if we were to adopt the elliptic structure of the equation to recover normal regularity for several times, we would obtain some lower-order terms non-integrable in time, ultimately preventing the closure of the energy estimates.
\end{enumerate}
  	\vspace{-1mm}
  Therefore, we instead carry out normal derivative estimates on the equations directly as those for compressible Euler equations in \cite{Hadzi.J2018,Shkol.S2019}. This requires a delicate analysis of how the equation structure, especially the degenerate viscous terms, behaves under normal differentiation. 

  Meanwhile, in addition to the difficulty caused by density degeneracy on the boundary, a central challenge in studying spherically symmetric motions lies in dealing with the coordinate singularity at the origin (the center of symmetry). 
  In \cite{Luo.X.Z2014}, the authors applied higher-order tangential energy estimates and elliptic structures to improve regularities, and constructed suitable cutoff functions in elliptic estimates to handle the difficulty of singularity at the origin. Nevertheless, we have to employ higher-order normal derivative estimates as stated above, and directly applying standard derivatives $\partial^N_x$ to the equations produces terms such as $\partial^{N+1}\eta$ and $x \partial^N \eta$, which are of comparable order when viewed through the Hardy inequality.
  Applying a normal derivative estimate with an improper differential operator not only introduces uncontrollable boundary terms, but also typically yields estimates on some linear combinations of $\partial^{N+1}\eta$ and $x \partial^N \eta$, and such estimates cannot control $\partial^{N+1}\eta$ itself.
  To overcome this difficulty, we introduce a \emph{non-trivial differential operator} $x\partial_x^{N} + (N+1) \partial_x^{N-1}$ and a proper test function when applying estimates to the system; {see the detailed analysis and explanations in \secref{sect-ideas} below}.
  In contrast, the key feature of our constructed operator $x\partial_x^{N} + (N+1)\partial_x^{N-1}$ lies in its ability to avoid boundary terms and ensure that estimates of the specific linear combination (i.e. $G_N$ in \eqref{def-GN} below) also provide individual control over $\partial^{N+1}\eta$ and $x \partial^N \eta$.

  In this paper, we address the above challenges and establish the following main results:
  we obtain the existence of global solutions around the affine solutions to the compressible Navier--Stokes equations with degenerate viscosities in the case of $\delta=\gamma$. This existence shows the structure stability of the affine solutions. 
  Moreover, we demonstrate the inviscid limit toward the solutions around the affine solutions of the compressible Euler equations with an explicit converging rate of viscosity coefficients, which provides an alternative proof of the corresponding result for the inviscid Euler equations. 
  
\section{Main results}
\subsection{Lagrangian reformulation and the perturbation}
Before stating the main results, we first again adopt the Lagrangian formulation and reduce the original free boundary problem \eqref{CNS-sym} to the problem on fixed domain $(0,\bar R)$. Note that here the Lagrangian variable $r(x,t)$ is re-defined with respect to the reference affine solution in \eqref{eq-alpha1}: 
\begin{equation} \label{CNS-Lag}
\int_0^{r(x,t)}s^2\rho(s,t)\,ds=\int_0^xs^2\bar\rho(s)\,ds\quad \text{for } x\in (0,\bar{R}),~t>0.     
\end{equation}

Set $f(x,t)=\rho(r(x,t),t)$, $v(x,t)=u(r(x,t),t)$. Then $\eqref{CNS-sym}$ is reformulated as
$$f(x,t)=\frac{x^2\bar{\rho}(x)}{r^2(x,t)r_x(x,t)}\quad  \text{for} ~~x \in (0,\bar{R}),$$
and
\begin{equation} \label{eq-Lag-original}
\begin{aligned}
&\bar{\rho}\left(\frac{x}{r}\right)^2r_{tt}+\left[A\left(\frac{x^2}{r^2}\frac{\bar{\rho}}{r_x}\right)^{\gamma}\right]_x\\
=&(2\bar{\mu}+\bar{\lambda})\left[\left(\frac{x^2}{r^2}\frac{\bar{\rho}}{ r_x}\right)^{\gamma}\frac{(r^2r_t)_x}{r^2r_x}\right]_x 
-4\bar{\mu}\left[\left(\frac{x^2}{r^2}\frac{\bar{\rho}}{r_x}\right)^{\gamma}\right]_x\frac{r_t}{r}.
\end{aligned}
\end{equation}

Next, we define the perturbation variables with respect to a given affine solution. 
Given that the affine solution expands rapidly---as implied by \eqref{sup-alpha}, it grows linearly as time tends to infinity---any small initial perturbation might grow significantly over time. Therefore, it is natural to study the perturbation in the following form:
\begin{equation} \label{def-ptb}
	\eta(x,t):=\dfrac{r(x,t)-r_\alpha(x,t)}{r_\alpha(x,t)}.
\end{equation}
In this work, we show that $\eta$ is Lyapunov stable, which means if it is initially small, it remains small for all time; see $\eqref{ineq-energy-final}_2$ below. Such stability is important. First, it guarantees that the background solution persists under small perturbations, meaning that the solution is physically observable and does not blow up under tiny disturbances. Second, it ensures that the solution maintains its nearly linear expansion when perturbed, indicating that this kind of expansion is robust and can be sustained in nature.

\eqref{def-ptb} implies that 
\begin{equation*} 
r(x,t) = (1+\eta(x,t)) \alpha(t) x,
\end{equation*}
and thus by substituting it into \eqref{eq-Lag-original}, system \eqref{CNS-sym} is transformed to:
\begin{equation} \label{eq-Lag0}
\begin{aligned}
&\dfrac{\bar{\rho}x}{(1+\eta)^2}  \left[ \alpha  \eta_{tt} + 2 \alpha' \eta_t+\left(c_1\alpha^{2-3\gamma}-(2\bar\mu+3\bar\lambda)\frac{c_1}{A}\alpha^{1-3\gamma}\alpha'\right)(1+\eta) \right] \\&
 + \big(A\alpha^{2-3\gamma} - (2\bar\mu+3\bar\lambda)\alpha^{1-3\gamma}\alpha' \big) \left[\left(\frac{\bar\rho}{(1+\eta)^2(1+\eta+x\eta_x)}\right)^\gamma\right]_x\\
 = & (2\bar\mu+\bar\lambda)\alpha^{2-3\gamma}\left[\left(\frac{\bar\rho}{(1+\eta)^2(1+\eta+x\eta_x)}\right)^\gamma\left(\dfrac{\eta_t + x \eta_{xt}}{1+\eta + x \eta_x} + \dfrac{2\eta_t}{1+\eta}\right)\right]_x\\&
 - 4\bar\mu \alpha^{2-3\gamma}\left[\left(\frac{\bar\rho}{(1+\eta)^2(1+\eta+x\eta_x)}\right)^\gamma\right]_x \dfrac{\eta_t}{1+\eta},
\end{aligned}
\end{equation}
for $ x \in (0,\bar{R}), t \in (0,\infty) $.

In addition, notice that \eqref{sup-alpha} implies that $\alpha$ grows linearly as $t$ tends to $\infty$, it is more convenient to introduce the following change of variables in time as \cite{Hadzi.J2018}:
\begin{equation}
\tau = \tau(t) := \int_0^t \dfrac{1}{\alpha(s)} \,ds.
\end{equation}
Then in the new coordinates $(x,\tau)$, the affine solution is given by
\begin{equation} \label{def-aff2}
r_\alpha(x,t)=\alpha(\tau)x, \quad v_\alpha(x,\tau)=\alpha_\tau(\tau)\big[\alpha(\tau)\big]^{-1}x, \quad f_\alpha(x,\tau)=\alpha^{-3}(\tau)\bar{\rho}(x)
\end{equation}
where $(\bar\rho, \alpha) $ satisfies
\begin{equation}
\label{eq-alpha2}
\begin{cases}
A\left(\bar{\rho}^\gamma \right)_x=-c_1\bar{\rho}x,\quad \rhobar_{|_{x=\bar{R}}}=0,\\
\alpha_{\tau\tau}=c_1 \alpha^{4-3\gamma}-(2\bar\mu+3\bar\lambda)\frac{c_1}{A}\alpha^{2-3\gamma}\alpha_\tau +\frac{(\alpha_\tau)^2}{\alpha},
\end{cases}
\end{equation}
and the initial condition \eqref{def-aff-data1} is equivalent to
\begin{equation} \label{def-aff-data2}
(\alpha, \alpha_\tau)|_{\tau=0}=(\alpha_0,\alpha_0\alpha_1).
\end{equation}
We further impose the assumption that for $2\bar\mu+3\bar\lambda\in[0,\varpi]$,
\begin{equation}
    \label{def-aff-data3}
    \alpha_1>0, \quad  \alpha_0- \dfrac{\varpi}{A}\alpha_1>0,~~ \text{and} ~~\alpha_0^{(3\gamma-1)/2}-\dfrac{\varpi}{A}c_1^{1/2} >0. 
\end{equation}
Given \eqref{def-aff-data3},
as shown in \lemref{lem-alpha0} below,  there exist positive constants $\beta_1=\alpha_1$, and $\beta_2$ depending on $ \gamma, c_1, \alpha_0, \alpha_1, $ such that
\begin{equation} \label{ineq-alpha1}
\beta_1 \leq  \dfrac{\alpha_\tau}{\alpha} \le \beta_2,\quad \text{and}
\quad \alpha_0 e^{\beta_1 \tau} \leq \alpha \leq \alpha_0 e^{\beta_2 \tau}.
\end{equation}
Note that by the monotonicity, the latter two inequalities of \eqref{def-aff-data3} is easily satisfied for $\alpha_0$ large enough with respect to $\alpha_1$. When $\varpi$ tends to $0$, since the latter two inequalities of \eqref{def-aff-data3} reduces to $\alpha_0>0$, the inviscid limit of \eqref{def-aff-data3} is consistent with that in the inviscid case  \cite{Hadzi.J2018,Shkol.S2019}.

For simplicity, we use the same notations for the functions $ \alpha, \eta$ in the new coordinates $ (x,\tau) $ as in the coordinates $ (x,t) $. Finally, the original free boundary problem \eqref{CNS-sym} can be written as the following fixed boundary problem, in the new coordinates:
\begin{equation} \label{eq-Lag}
\begin{aligned}
&\dfrac{\bar{\rho}}{(1+\eta)^2} x \left[ \alpha  \eta_{\tau\tau} +  \alpha_\tau \eta_\tau + c_1\alpha^{3-3\gamma}\tilde{\alpha}(1+\eta)\right] \\
& +  A\alpha^{3-3\gamma}\tilde{\alpha}\left[\left(\frac{\bar\rho}{(1+\eta)^2(1+\eta+x\eta_x)}\right)^\gamma\right]_x\\
 = & (2\bar\mu+\bar\lambda)\alpha^{3-3\gamma}\left[\left(\frac{\bar\rho}{(1+\eta)^2(1+\eta+x\eta_x)}\right)^\gamma\left(\dfrac{\eta_\tau + x \eta_{x\tau}}{1+\eta + x \eta_x} + \dfrac{2\eta_\tau}{1+\eta}\right)\right]_x\\&
 - 4\bar\mu \alpha^{3-3\gamma} \left[\left(\frac{\bar\rho}{(1+\eta)^2(1+\eta+x\eta_x)}\right)^\gamma\right]_x \dfrac{\eta_\tau}{1+\eta}.
\end{aligned}
\end{equation}
for $ x \in (0,\bar{R}), \tau \in (0,\infty) $ and with the initial data:
\begin{equation}  \label{data}
\eta(x,0)=\eta_0,~\eta_\tau(x,0)=\eta_1,
\end{equation}
where 
\begin{equation}
    \label{def-alpha-tilde}
    \tilde\alpha:= \alpha - \frac{2\bar\mu+3\bar\lambda}{A}\frac{\alpha_\tau}{\alpha}.
\end{equation}
Note that when \eqref{def-aff-data3} holds, then $\tilde{\alpha}$ is equivalent to $\alpha$. Namely, there exists a constant $\beta_3>0$  depending on $ A,  \gamma, c_1, \varpi, \alpha_0, \alpha_1, $ such that 
\begin{equation} \label{equiv-tilde-alpha1}
    \beta_3 \alpha \le \tilde\alpha \le \alpha;
\end{equation}
see the details in \lemref{lem-alpha0}.

\begin{remark}
The boundary conditions $\eqref{CNS-sym}_4$ take the form
\begin{equation*} \label{bd1}
\dfrac{\bar\rho}{(1+\eta)^2(1+\eta+x\eta_x)}=0    
\end{equation*}
and
\begin{equation*} \label{bd2}
\Big(\dfrac{\bar\rho}{(1+\eta)^2(1+\eta+x\eta_x)}\Big)^\gamma\big\{ A-(2\bar\mu+3\bar\lambda)\frac{\alpha_\tau}{\alpha} - (2\bar\mu+\bar\lambda)\frac{\eta_\tau+x\eta_{x\tau}}{1+\eta+x\eta_x} -2\bar\lambda \dfrac{\eta_\tau}{1+\eta}\big\}=0    
\end{equation*}
at the vacuum boundary $x=\bar R$ under the Lagrangian coordinates. Due to the regularity of $\eta$ \eqref{def-reg1}--\eqref{def-reg2} and the smallness of $\eta,x\eta_x$ (see \autoref{thm-stab} below), together with the fact that $\bar\rho$ vanishes at the vacuum boundary $x=\bar R$, the above two boundary conditions are automatically satisfied by the standard trace theorem for both viscous and inviscid case.
\end{remark}

\subsection{Main theorems}

Denote $L^p, H^m$ as the classical Lebesgue and Sobolev space defined on $(0,\bar{R})$. Without further mention, $\int \cdot \, dx$ denotes $\int_0^{\bar R} \cdot \, dx$ for simplicity.

To show the global existence of strong solutions to problem \eqref{eq-Lag}, we introduce the following energy and dissipation functionals:
\begin{equation} \label{def-energy}
	\begin{aligned}
		\mathcal{E}_N(\tau):=\sum_{j=0}^N & \Bigl( \alpha^{1+r_1} \int \bar\rho^{(\gamma-1)j+1} x^4(\partial_x^j\eta_{\tau})^2\,dx\\
		&+ \alpha^{3-3\gamma+r_1}\tilde\alpha\int\bar\rho^{(\gamma-1)(j+1)+1} \bigl( x^4 (\partial_x^{j+1}\eta)^2 + x^2(\partial_x^j\eta)^2 \bigr)\,dx \Bigr),\\
        \mathcal{E}_N^{(1)}(\tau):= \sum_{j=0}^N & \int_0^\tau\alpha^{r_1}\alpha_\tau\int \bar\rho^{(\gamma-1)j+1} x^4(\partial_x^j\eta_{\tau})^2  \,dx\,d\tau',\\
        \mathcal{E}_N^{(2)}(\tau):= \sum_{j=0}^N & \int_0^\tau \alpha^{2-3\gamma+r_1}\alpha_\tau\tilde{\alpha} \int\bar\rho^{(\gamma-1)(j+1)+1}\bigl( x^4(\partial_x^{j+1}\eta)^2+x^2(\partial_x^j\eta)^2 \bigr)\,dxd\tau',\\
		\mathcal{D}_N(\tau):=\sum_{j=0}^N &\int_0^\tau\alpha^{3-3\gamma+r_1}\int\bar\rho^{(\gamma-1)(j+1)+1}\bigl( x^4(\partial_x^{j+1}\eta_\tau)^2+x^2(\partial_x^j\eta_\tau)^2 \bigr)\,dx\,d\tau',\\
	\end{aligned}
\end{equation}
where $r_1$ is a constant satisfying $ r_1 \le \min \{ 3\gamma-4,1\}$. Given the initial data in \eqref{data}, the initial energy is defined accordingly as
\begin{equation} \label{def-ini-energy}
\begin{aligned}
\mathcal{E}_{in} = &\sum_{j=0}^{N_0} \Bigl( \int \bar\rho^{(\gamma-1)j+1} x^4(\partial_x^j\eta_1)^2\,dx \\
& \qquad + \int\bar\rho^{(\gamma-1)(j+1)+1} \bigl( x^4 (\partial_x^{j+1}\eta_0)^2 + x^2(\partial_x^j\eta_0)^2 \bigr)\, dx \Bigr),    
\end{aligned}  
\end{equation}
where $N_0:= \max \{4 + \lceil \frac{1}{\gamma-1} \rceil, 6 \} $.

Correspondingly, a strong solution to equation \eqref{eq-Lag} is defined as follows.
\begin{definition}[Strong solution to NS] \label{def-stong-sol-NS}
For $\bar\mu>0, 2\bar\mu+3\bar\lambda\ge0$, given any time $T>0$, $\eta$ is called a strong solution to \eqref{eq-Lag} with initial data \eqref{data} in $[0,T]$ if it solves \eqref{eq-Lag} for a.e. $(x,\tau) \in (0,\bar R) \times (0,T)$ with the following regularity:
\begin{equation} \label{def-reg1}
\begin{aligned}
     \eta \in  H^1 ([0,T];H^2) \cap H^2([0,T];L^2).
    \end{aligned}
\end{equation}
\end{definition}

\begin{definition}[Strong solution to Euler] \label{def-stong-sol-EL}
For $\bar\mu = 2\bar\mu+3\bar\lambda = 0$, given any time $T>0$, $\eta$ is called a strong solution to problem \eqref{eq-Lag} with initial data \eqref{data} in $[0,T]$ if it solves \eqref{eq-Lag} for a.e. $(x,\tau) \in (0,\bar R) \times (0,T)$ with the following regularity:
\begin{equation} \label{def-reg2}
\begin{aligned}
     \eta \in  L^2 ([0,T];H^2) \cap H^2([0,T];L^2).
    \end{aligned}
\end{equation}
\end{definition}

The main results of this paper are the following: 
\begin{theorem}[Global existence] \label{thm-stab}
Let $\gamma >1$, $2\bar\mu+3\bar\lambda\in[0,\varpi]$, and $(\bar\rho, \alpha)$ be the affine solution given by \eqref{def-aff2}, \eqref{eq-alpha2} with initial data \eqref{def-aff-data2} satisfying \eqref{def-aff-data3}. 
Assume that the initial data $\eta_0,\eta_1$ in \eqref{data} admit even extensions to $(-\bar R, \bar R)$.  

Then there exists a constant $\bar\epsilon>0$ depending only on $A,  \gamma, c_1, \varpi, \bar R, \alpha_0, \alpha_1,$ such that if the extended initial data \eqref{data} satisfy the following regularity condition
\begin{equation}\label{regul-data}
\begin{aligned}
  &\bar\rho^{(\gamma-1)\frac{j}{2}+\frac{1}{2}}x^2 \partial_x^j \eta_1 \in L^2  (-\bar R, \bar R), \\
  &\bar\rho^{(\gamma-1)\frac{j+1}{2}+\frac{1}{2}}x^2 \partial_x^{j+1} \eta_0,~\bar\rho^{(\gamma-1)\frac{j+1}{2}+\frac{1}{2}}x \partial_x^{j} \eta_0 \in L^2 (-\bar R, \bar R),  
\end{aligned}    
\end{equation}
for any $0 \le j \le N_0$ and
\begin{equation} \label{ineq-ini-energy}
    \mathcal{E}_{in} \le \bar\epsilon,
\end{equation}
then problem \eqref{eq-Lag} admits a unique global strong solution $\eta$ with the following estimates:
\begin{equation} \label{ineq-energy-final}
\begin{aligned}
&\sup_{0 \le \tau  < \infty }  \mathcal{E}_{N_0}(\tau) + (2\bar\mu + \bar\lambda)           \mathcal{D}_{N_0}(\infty)   \le C\mathcal{E}_{in},\\
&\sup_{0 \le \tau  < \infty } \Big\{  \norm{\eta}{H^2}^2 +  \alpha^{1+r_1} \norm{\eta_\tau}{H^2}^2   \Big\}  \le C\mathcal{E}_{in},\\
& \sup_{0 \le \tau  < \infty } \alpha^{1+r_1} \norm{\eta_{\tau\tau}}{H^1}^2 \le  C  [1+ (2\bar\mu+\bar\lambda)^2] \mathcal{E}_{in},
\end{aligned}
\end{equation}
holds for any $r_1 \le \min \{3\gamma-4,1 \}$. Moreover, the solution $\eta$ also satisfies the following estimates:
\begin{align}
&\sup_{0 \le \tau  < \infty }  \mathcal{E}_{N_0}^{(1)}(\tau) + \norm{\alpha^{\frac{1+r_1}{2}}\eta_\tau}{L^2([0,\infty),H^2)}^2  \le C_{r_1}\mathcal{E}_{in}, &&\text{for }r_1 \le 3\gamma-4, r_1<1; \\
&\sup_{0 \le \tau  < \infty }  \mathcal{E}_{N_0}^{(2)}(\tau) \le C_{r_1}\mathcal{E}_{in} , &&\text{for $ r_1 < 3\gamma-4, r_1 \le 1$;} \\
&\norm{\alpha^{\frac{1+r_1}{2}}\eta_{\tau\tau}}{L^2([0,\infty);H^1)}^2 \le C_{r_1} [1+ (2\bar\mu+\bar\lambda)^2] \mathcal{E}_{in},  &&\text{for }r_1 \le 3\gamma-4, r_1<1; \\
&\norm{ \alpha^{r_2}\eta}{L^2([0,\infty),H^2)}^2 \le C\mathcal{E}_{in}, && \text{for $ r_2 <0 $;} 
\end{align}
where $C$ is a constant independent of the viscosity coefficients $\bar\mu, \bar\lambda,$ time $T$, and $r_1, r_2$, but depending on 
$A,  \gamma, c_1, \varpi, \bar R, \alpha_0, \alpha_1,$ and $C_{r_1}$ is a constant depending on $A,  \gamma, c_1, \varpi, \bar R, \alpha_0, $ $ \alpha_1,  r_1$.
\end{theorem}

\begin{remark}
    Let $\eta(x,\tau)$ be the strong solution to problem \eqref{eq-Lag} obtained in \autoref{thm-stab} and $(\bar\rho(x), \alpha(\tau))$ be the affine solution given by \eqref{def-aff2}, \eqref{eq-alpha2}. Due to the regularity of $\eta$ \eqref{def-reg1} \eqref{def-reg2}, the flow map $r(x,t)$ and the new variable $\tau(t)$ are well-defined, so do their inverse maps. Denote
    \begin{align*}
         \rho(r,t)&:=\dfrac{\bar\rho}{\alpha^3(1+\eta)^2(1+\eta+x\eta_x)}\Big(x(r,t),\tau(t) \Big),\\
         u(r,t)&:=\big[\eta_\tau  + \frac{\alpha_\tau}{\alpha} x(1+\eta) \big]\Big(x(r,t),\tau(t) \Big).
    \end{align*}
    Then $(\rho,u)$ is the strong solution to problem \eqref{CNS-sym} with $\delta=\gamma$. 
    We remark here that the perturbed solution is required to have the same total mass as the background solution; see \eqref{CNS-Lag}.
\end{remark}

\begin{remark} \label{rmk-trace}
The assumption that the initial data $\eta_0,\eta_1$ are even functions is natural in this spherical symmetric setting. Indeed, for a spherically symmetric solution, the physical meaning of the Lagrangian coordinate naturally induces an odd extension of the solution to the equation \eqref{eq-Lag-original} with respect to the center, and therefore the perturbation variable \eqref{def-ptb} is even. Moreover, by the reflection property and the uniqueness of solutions to the equation \eqref{eq-Lag0}, if the initial data are even, the solution remains even for all later times. This structural assumption is also technically convenient: when performing integration by parts in the a priori estimates, boundary terms at the center would normally appear. By approximating the even initial data with smooth even functions, the solution remains smooth and even, and hence all boundary terms at the origin vanish. 
\end{remark}

Moreover, since the a priori estimates we obtained are uniform with respect to the viscosity coefficients, we further have the following theorem by the standard compactness method and we also establish the convergence rate: 
\begin{theorem}[Inviscid limit] \label{thm-vanish}
Let $\gamma>1$ and $2\bar\mu+3\bar\lambda\in[0,\varpi]$.  
Denote $\alpha^{\bar\mu,\bar\lambda}$ the solution to equation \eqref{eq-alpha2}, and $\alpha^0$ the solution to the same problem with viscosity coefficients $\bar\mu=\bar \lambda=0$ and with the initial data
\begin{equation} \label{data-vanish-alpha}
(\alpha^{\bar\mu,\bar\lambda},\alpha^{\bar\mu,\bar\lambda}_\tau)|_{\tau=0}=(\alpha^0,\alpha^0_\tau)|_{\tau=0}=(\alpha_0,\alpha_0\alpha_1)
\end{equation}
satisfying \eqref{def-aff-data3}. 
Let $\eta^{\bar\mu,\bar\lambda}$ be the strong solution to problem \eqref{eq-Lag} obtained in \autoref{thm-stab}, and $\eta^0$ be the strong solution to the same problem with viscosity coefficients $\bar\mu=\bar\lambda=0$ and with the initial data
\begin{equation} \label{data-vanish}
    \eta^{\bar\mu,\bar\lambda}(x,0)=\eta^0(x,0)=\eta_0,\quad\text{and }~\eta^{\bar\mu,\bar\lambda}_\tau(x,0)=\eta_\tau^0(x,0)=\eta_1.
\end{equation}
Assume that the initial data $(\eta_0,\eta_1)$ satisfies \eqref{regul-data} and \eqref{ineq-ini-energy}.

Then for any time $T>0$, 
\begin{equation} \label{conv}
\begin{aligned}
\eta^{\bar\mu,\bar\lambda} \to \eta^0 & ~~~\text{in}~ W^{1,\infty}([0,T]; H^s),  
\end{aligned}
\quad\text{as}~ \bar\mu, 2\bar\mu+3\bar\lambda \to 0,
\end{equation}
where $s \in [0,2)$ is a constant. Moreover, we also have
\begin{equation} \label{ineq-conv1}
\begin{aligned}
&\sup_{0 \le \tau  \le T} \Big\{  (\alpha^0)^{\frac{4-3\gamma+r_1}{2}}\norm{\eta^{\bar\mu,\bar\lambda} - \eta^0}{H^1} +  (\alpha^0)^{\frac{1+r_1}{2}} \norm{\eta^{\bar\mu,\bar\lambda}_\tau - \eta_\tau^0}{H^1}   \Big\} \le C_T \mathcal{E}_{in}^{\frac{1}{2}}\big(  2\bar\mu+\bar\lambda \big),
\end{aligned}
\end{equation}
holds for any $ r_1 \le \min \{3\gamma-4,1 \}$, where $C_T$ is a constant independent of the viscosity coefficients $\bar\mu, \bar\lambda$ but depend on 
$A, \gamma, c_1, \varpi, \bar R, \alpha_0, \alpha_1, T$.
\end{theorem}

\begin{remark}\label{rem-boundary layer}
 Our results show that in this inviscid limit, no boundary layer appears. This is partially due to the consistency that there are no boundary conditions for flows with highly degenerate viscosities and for inviscid flow. 
While for the viscous Saint-Venant system ($\gamma=2,~\mu=\mubar \rho,~\lambda=0$), 
Li, Wang, and Xin \cite{Li.W.X2026,Li.W.X2025} found that any classical solution admits an additional Neumann boundary condition on the velocity, which plays an important role in the global well-posedness result of the viscous Saint-Venant system in \cite{Xin.Z.Z2025,Chen.Z.Z2026}.

\end{remark}

\begin{remark}\label{rem-artificial viscosity}
    In \cite{Couta.S2012}, the local-in-time strong solution to the VFBP for the compressible Euler equations was obtained in the limit as the artificial viscosity tends to zero. Following their way of setting the artificial viscosity $\kappa$, the momentum equation in Lagrangian coordinates would be
    $$
    \bar{\rho}\left(\frac{x}{r}\right)^2r_{tt}+\left[A\left(\frac{x^2}{r^2}\frac{\bar{\rho}}{r_x}\right)^{\gamma}\right]_x=-\kappa\left[A\left(\frac{x^2}{r^2}\frac{\bar{\rho}}{r_x}\right)^{\gamma}\right]_{xt} \, .
    $$
    which is quite different from \eqref{eq-Lag-original}. The physically motivated viscosities, \eqref{vis-coeff0}, introduce more complicated viscous terms in \eqref{eq-Lag-original}.
\end{remark} 

In the rest of this paper, $C$ will be used to denote generic positive constants which are independent of $T, \bar\mu, \bar\lambda, r_1$ and $r_2$ but may depend on $A,  \gamma, c_1, \varpi, \bar R, \alpha_0, \alpha_1$; $C_T$ and $C_{r_1}$ additionally depend on $T$ and $r_1$, respectively.

\subsection{Strategy and main ideas}
\label{sect-ideas}
We now proceed to outline the main steps and ideas of the proof. We begin by establishing the standard basic energy estimates under suitable a priori assumptions. However, due to the strong nonlinearity of the equation, these basic energy estimates alone are insufficient to justify the validity of the a priori assumptions. As a result, it becomes necessary to carry out higher-order energy estimates under a different set of a priori assumptions.

To reveal these difficulties and the corresponding methods of the higher-order estimates, we first rewrite equation \eqref{eq-Lag} via \eqref{eq-alpha2} as follows:
\begin{equation} \label{eq-Lag-high}
\begin{aligned}
&  \alpha x \eta_{\tau\tau} +  \alpha_\tau x\eta_\tau-c_1\alpha^{3-3\gamma}\tilde{\alpha}(1+\eta) x\eta\\
&+  A\alpha^{3-3\gamma}\tilde{\alpha}\frac{(1+\eta)^2}{\bar\rho}\Big\{\bar\rho^\gamma\big[(1+\eta)^{-2\gamma}(1+\eta+x\eta_x)^{-\gamma}-1\big]\Big\}_x \\
 = & -\frac{2\bar\mu+\bar\lambda}{\gamma} \alpha^{3-3\gamma}\frac{(1+\eta)^2}{\bar\rho}\Big\{\bar\rho^\gamma \big[(1+\eta)^{-2\gamma}(1+\eta+x\eta_x)^{-\gamma} \big]_\tau\Big\}_x \\
 & - 4\bar\mu \alpha^{3-3\gamma} \dfrac{(1+\eta)\eta_\tau}{\bar\rho}\Big\{\bar\rho^\gamma \big[(1+\eta)^{-2\gamma}(1+\eta+x\eta_x)^{-\gamma} \big]\Big\}_x . 
\end{aligned}
\end{equation}
The equation exhibits a complex structure. The strong degeneracy near the boundary, together with the geometric singularity at the origin makes the standard energy functionals ineffective. Moreover, different terms carry different expanding rates in time, and for global existence, rapidly growing terms cannot be absorbed by slowly expanding ones.
In the following, we present three major aspects to illustrate the ideas for carrying out the higher-order estimates on \eqref{eq-Lag-high} and the inviscid limit.

\subsubsection{Degenerate viscosities.}
The degeneracy of elliptic structure near the vacuum boundary can be overcome by the framework of weighted energy estimates and the Hardy inequality developed by \cite{Jang.M2009,Couta.L.S2010,Couta.S2011,Luo.X.Z2014} and the references therein on local well-posedness theory for inviscid flows.

However, as mentioned in the introduction, for global-in-time well-posedness, we need to establish estimates that are uniform or integrable in time, which is fundamentally different from the local-in-time setting. In \eqref{eq-Lag-high}, the time evolution term exhibits the fastest growth (in the order of $\alpha$) than all the other terms (in the orders of $\alpha^{4-3 \gamma}$ or $\alpha^{3-3 \gamma} $). Accordingly, if we were to adopt the conventional approach---first performing tangential energy estimates to improve regularities, and using the elliptic structure to recover normal regularity for several times---we would obtain some lower-order terms non-integrable in time, ultimately preventing the closure of the energy estimates.

As a result, we carry out normal derivative estimates on the equations directly as those for compressible Euler equations in \cite{Hadzi.J2018,Shkol.S2019}, instead of the conventional approach for viscous flows with constant shear viscosity in \cite{Liu.Y2019a,Liu.2019a}. This requires a delicate analysis of how the equation structure, especially the degenerate viscous terms, behaves under normal differentiation. 

\subsubsection{Singularities at the origin.}
Different from the case for the basic energy estimates, in deriving higher-order estimates, the term from the derivative of the perturbed pressure does not admit a representation as a total time derivative. Before proceeding to the energy estimates, it is instructive to consider a simplified model. 

Observe that the principal term of 
$$(1+\eta)^{-2\gamma}(1+\eta+x\eta_x)^{-\gamma}-1$$
is $-\gamma(x\eta_x+3\eta)$. Hence, a simplified model which captures the essential structure of equation \eqref{eq-Lag-high} is:
\begin{equation} \label{eq-toy}
\alpha x\eta_{\tau\tau} - A\gamma \alpha^{3-3\gamma}\tilde{\alpha} \frac{[\bar\rho^\gamma (x\eta_x+3\eta)]_x}{\bar\rho } = (2\bar\mu+\bar\lambda)\alpha^{3-3\gamma} \frac{[\bar\rho^\gamma (x\eta_{x\tau}+3\eta_\tau)]_x}{\bar\rho }.
\end{equation} 
\emph{Due to the singularities at the origin, the normal estimate with the standard derivative $\partial_x^N$ does not work.} For instance, if we apply the standard derivative $\partial_x^N$ to \eqref{eq-toy}, multiply the resultant by $\bar\rho^{(\gamma-1)N+1} x^2  \partial_x^N(x\eta_\tau)$ (such that $\bar\rho^{(\gamma-1)N+1} x^2  \partial_x^N(x\eta_\tau) \cdot \partial_x^N(x\eta_{\tau\tau })$ is a total time derivative), and integrate over $[0,\bar R]$, it yields that
\begin{align*}
&\frac{1}{2} \alpha \frac{d}{d\tau}\int  \bar\rho^{(\gamma-1)N+1} (x^2\partial_x^N\eta_\tau + N x\partial_x^{N-1} \eta_\tau)^2\,dx \\
&+\frac{A\gamma}{2} \alpha^{3-3\gamma}\tilde{\alpha} \frac{d}{d\tau}\int \bar\rho^{(\gamma-1)(N+1)+1}  (x^2\partial_x^{N+1}\eta +(N+3)x\partial_x^{N}\eta  )^2\,dx \\
&+  (2\bar\mu+\bar\lambda)\alpha^{3-3\gamma} \int \bar\rho^{(\gamma-1)(N+1)+1}  (x^2\partial_x^{N+1}\eta_\tau +(N+3)x\partial_x^{N}\eta_\tau  )^2\,dx \\
=&  -2N \alpha^{3-3\gamma}\tilde{\alpha} \int \bar\rho^{(\gamma-1)(N+1)+1}(x^2\partial_x^{N+1}\eta +(N+3)x\partial_x^{N}\eta  ) \cdot\partial_x^{N-1}\eta_\tau\,dx +\cdots.
\end{align*}
The presence of the first term on the right-hand side prevents one from closing this estimate. To be specific, this term cannot be expressed as a total time derivative. At the same time, it follows from the Hardy inequality that $\partial_x^{N-1}\eta_\tau$ can be treated as of the same order as $x\partial_x^{N}\eta_\tau$. Thus, if one considers only the spatial weights, it can only be controlled by a combination of the terms from the pressure and the viscosity. Nevertheless, it cannot be absorbed uniformly in time due to the rapid growth of $\alpha$. 

To tackle these challenges, we must design carefully an appropriate operator for the higher-order energy estimates. 
Accordingly, we introduce the \emph{non-trivial degenerate differential operator}
$$x\partial_x^{N} + (N+1) \partial_x^{N-1}$$ and the proper test function $$\bar\rho^{(\gamma-1)N+1} [x\partial_x^{N} + (N+1) \partial_x^{N-1}](x\eta_\tau).$$ Observe that by denoting
        \begin{equation}  \label{def-GN}
         G_N :=
         \begin{cases}
          x^2\partial_x^{N+1}\eta+(2N+3)x\partial_x^N\eta+N(N+2) \partial_x^{N-1}\eta,   & \text{for }N \ge 1,\\
          x^2\partial_x\eta+3x\eta, &  \text{for }N=0,
         \end{cases}
        \end{equation}
it holds that for $N\ge 1$,
\begin{align*}
  &(x\partial_x^{N} + (N+1) \partial_x^{N-1})[(x\eta_x+3\eta)_x]= G_{N+1}=\partial_x(G_N), \\
  &(x\partial_x^{N} + (N+1) \partial_x^{N-1})(x\eta)=G_{N-1}.
\end{align*}
Therefore, we obtain the following energy estimates by the above non-trivial treatment on normal derivatives:
\begin{align}
  &\frac{1}{2} \alpha \frac{d}{d\tau}\int  \bar\rho^{(\gamma-1)N+1} G_{N-1,\tau}^2\,dx + \frac{A\gamma}{2} \alpha^{3-3\gamma}\tilde{\alpha} \frac{d}{d\tau} \int \bar\rho^{(\gamma-1)(N+1)+1}  G_N^2\,dx  \notag \\ 
+ & (2\bar\mu+\bar\lambda)\alpha^{3-3\gamma} \int \bar\rho^{(\gamma-1)(N+1)+1}  G_{N,\tau}^2\,dx 
= \text{lower~order~terms}. \label{eq-energy-simple}
\end{align}
Moreover, $G_N$ also provides individual control over $\partial^{N+1}\eta$ and $x \partial^N \eta$, which is precisely one of the purposes to construct the non-trivial differential operator; see \lemref{lem-GN} below for details.

Remarkably, raising one spatial derivative in \eqref{eq-energy-simple} requires matching the weight $\bar\rho^{\gamma-1}$. Yet this still gains 1/2-order regularity in the view of the Hardy inequality, which allows us to close the estimate by carrying out higher-order normal derivative estimates up to order $N_0$. 
Thus, after some delicate estimates of the Jacobians and other lower-order terms, we successfully derive closed higher-order energy estimates for $1\le N \le N_0$ as well as verifying the a priori assumptions. Then \autoref{thm-stab} is proved. 

It is worth noting that $x\partial_x^{N} + (N+1) \partial_x^{N-1}=\partial_x^{N-1} (x(\partial_x +\frac{2}{x} \textrm{I}))$,  where $(\partial_x +\frac{2}{x} )f=\textrm{div}(f \mathbf{e}_x)$ is exactly the divergence operator under the spherical symmetry. Therefore, in some sense,  \emph{$G_N$ is related to some higher-order divergence derivatives}, i.e. 
$$ \int \bar\rho^{(\gamma-1)(N+1)+1}  G_N^2\,dx = \int \bar\rho^{(\gamma-1)(N+1)+1}   (\partial_x^{N} x \textrm{div} (x\eta \mathbf{e}_x) )^2\,dx .$$
Here $x\eta=(r-r_\alpha)/\alpha(t)$ is the perturbation of the flow map over the expanding coefficient $\alpha(t)$. Interested readers may refer to \cite{Guo.H.J2021,Disco.H.L2026} for related \emph{singular} vector field method.

\subsubsection{Inviscid limit.}
The key to establishing the inviscid limit is to derive estimates that are uniform with respect to the viscosity coefficients. Our basic and higher-order energy estimates obtained in above regime could satisfy precisely this property.
In fact, in our case, the degeneracy weight of the viscous term near vacuum is the same as that of the pressure term. At the same time, the expanding rate of the viscous term is lower than that of the pressure term. As a consequence, the lower‑order terms arising from the viscous term can be controlled by the energy and dissipation functionals coming from the leading parts of the pressure term and the viscous term. 
Accordingly, by a standard compactness method, we can also show the convergence of a viscous solution to an inviscid one. 
 
 To establish the convergence rates, denote by $\eta^{\bar\mu,\bar\lambda}$ the viscous solution and $\eta^0$ the inviscid solution. Define the difference $\xi:=\eta^{\bar\mu,\bar\lambda}-\eta^0$. Then we perform energy estimates on $\xi$ using the same higher-order operator introduced in the proof of existence result, which gives the desired convergence rates and finishes the proof of \autoref{thm-vanish}.

\subsection{Structure of the paper}
This work is organized as follows:  in Section 3 we introduce some inequalities which will be frequently used. The a priori estimates, including basic and higher-order energy estimates, are presented in the Sections 4 and 5. 
The basic energy estimate shows clearly the elliptic structures of the pressure and the viscous term, as well as some intrinsic damping effects introduced by the expanding nature of the affine solutions. The higher-order energy estimates are the core of the a priori estimates, relying crucially on the strategies and ideas mentioned above.
Based upon the a priori estimates, in the final section, we establish the global existence of strong solutions around the affine solutions, and then verify the inviscid limit, which finishes the proof of \autoref{thm-stab} and \autoref{thm-vanish}. 

\section{Preliminaries}

In this section, we state some inequalities and lemmas which we will frequently use in this work, and introduce the properties of the affine solutions.  

\subsection{Inequalities}
\begin{lemma}[Hardy inequality]
	\label{lem-Hardy}
For $k\in \mathbb{R}$ and function $g$ defined on $[0,1]$ satisfying $\int_{0}^{1} s^k[g^2+(g')^2]\ ds<\infty$, there holds that
\begin{enumerate}[label = \rm (\roman*),ref = \rm(\roman*)]
\item if $k>1$, then
\begin{equation}
\int_{0}^{1}s^{k-2}g^2 \ ds\leq C  \int_{0}^{1}s^k[g^2+(g')^2]\ ds.
\end{equation}
\item if $k<1$, then  $g$ has a trace at $x=0$ and
\begin{equation}
\int_{0}^{1} s^{k-2} (g-g(0))^2 ds \leq C \int_{0}^{1} s^k(g')^2 \ ds.
\end{equation}
\end{enumerate}
\end{lemma}

\begin{proof}
	See \cite[Lemma~C.1]{Hadzi.J2018}.
\end{proof}

\begin{lemma}[Weighted space embedding]
	\label{lem-embedding}
	For $1<k\in \mathbb{R}$ and function $g$ defined on $[0,1]$ satisfying $\int_{0}^{1} [s^{k}g^2+s^k(g')^2]\ ds<\infty$, there holds that:
	\begin{equation}
		\sup_{s\in[0,1]} s^{k-1} g^2(s) \leq C \int_{0}^{1} [s^{k}g^2+s^k(g')^2]\ ds.
	\end{equation}
\end{lemma}
The proof can be finished by the Hardy-type embedding in the theory of weighted Sobolev spaces; see \cite{Li.W.X2025} for details. Here we provide a direct proof without introducing weighted Sobolev spaces.
\begin{proof} For smooth functions $g$, since $k>1$
	\begin{align*}
	    \sup_{s\in[0,1]} s^{k-1} g^2(s) & \leq \int_0^1 \abs{ ( s^{k-1} g^2)' } \, ds  \leq  \int_0^1  [(k-1) s^{k-2} g^2 +2s^{k-1}\abs{gg'}] \,ds \\
        &\le C \int_0^1 [s^{k-2} g^2 +s^{k}(g')^2] \,ds \le C \int_0^1 [s^{k} g^2 +s^{k}(g')^2] \,ds,
	\end{align*}
    where in the last inequality, the Hardy inequality is applied. Therefore, the rest of the proof can be followed by a standard density argument.
\end{proof}

\subsection{Properties of the affine solutions} \label{sect-affine}
\begin{lemma} \label{lem-alpha00}
Given any fixed positive constant $c_1 \in (0,\infty) $, together with initial data \eqref{def-aff-data1}
   and initial radius $\bar{R}$, 
   system \eqref{eq-alpha1} is globally well-posed. And there exist a constant $ c_2 $ depending on $ A,  \gamma, c_1, \alpha_0, \alpha_1,$ and $2\bar\mu+3\bar\lambda$, such that \eqref{sup-alpha} holds.
\end{lemma}
\begin{proof}
In virtue of the classical ODE theory, for the global well-posedness of system \eqref{eq-alpha1},  it suffices to prove the global existence of $\alpha$.
It follows from $\eqref{eq-alpha1}_2$ that
\begin{equation*}
	\dfrac{1}{2} \bigl( (\alpha')^2 \bigr)' = \dfrac{c_1}{3-3\gamma} \bigl( \alpha^{3-3\gamma} \bigr)'-(2\bar{\mu}+3\bar{\lambda})\frac{c_1}{A}\alpha^{1-3\gamma}(\alpha ')^2.
\end{equation*}
Therefore,
\begin{equation} \label{eq-alpha-int}
	\begin{aligned}
		&(\alpha')^2(t) + \dfrac{2c_1}{3\gamma - 3} \alpha^{3-3\gamma}(t)+2(2\bar{\mu}+3\bar{\lambda})\frac{c_1}{A}\int_0^t \alpha^{1-3\gamma}(\alpha')^2  \,ds \\
		= & \alpha_1^2 + \dfrac{2c_1}{3\gamma - 3} \alpha_0^{3-3\gamma} \in (0,\infty),  
	\end{aligned}
\end{equation}
which implies that $ \alpha $ is globally defined and strictly positive.

Suppose that $\alpha'(t)\leq 0$ for all $t$. Then $\eqref{eq-alpha1}_2$ shows that $\alpha''>0$ for all $t$, which implies that $\alpha'$ tends to some constant $\tilde{C} \le 0$ and that $\alpha''$ tends to 0 as $t \to \infty$. If  $\tilde{C}<0$, then $\alpha \to -\infty$, which is in contradiction with the fact that $\alpha$ is  always positive. If $\tilde{C}=0$, then $\alpha'',\alpha'$ both have to tend to $0$ as $t \to \infty$ and thus $\lim_{t\to\infty}\alpha \le \alpha(0)< \infty$, which is in contradiction with $\eqref{eq-alpha1}_2$. Therefore, there exists some $t_0 \ge 0$ such that $\alpha'(t_0)>0$. 

Furthermore, it holds that $\alpha'(t)>0$ for all $t \ge t_0$.
Otherwise, by contradiction arguments, one can suppose that there exists $t_1>t_0$ such that $\alpha'(t)>0$ for $t\in [t_0,t_1]$ and $\alpha'(t_1)=0$, and thus $\alpha''(t_1)\leq 0$, which conflicts with $\eqref{eq-alpha1}_2$.  

This implies $\lim_{t \to \infty}\alpha(t)$ exists. If $\lim_{t \to \infty}\alpha(t)< \infty$, then $\lim_{t\to\infty}\alpha'(t)=0$. This contradicts with $\eqref{eq-alpha1}_2$ which then implies $\lim_{t\to \infty} \alpha''(t)>0$. Therefore, one has 
\begin{equation} \label{lmt-alpha}
	\lim_{t \to \infty}\alpha(t)=+\infty.  
\end{equation}
Moreover, it follows from \eqref{eq-alpha-int} that $ $. 
\begin{equation}
    \label{def-beta2}
    |\alpha'|\le  \big(\alpha_1^2 + \dfrac{2c_1}{3\gamma - 3} \alpha_0^{3-3\gamma}\big)^{1/2} =:\beta_2 
\end{equation}
This, together with \eqref{lmt-alpha} and $\eqref{eq-alpha1}_2$, shows that there exists some $t_1\ge t_0$ such that
\begin{equation}
	\alpha''(t)>0 \quad \text{for all} \quad t\ge t_1, \quad \lim_{t \to \infty}\alpha''(t)=0,
\end{equation} 
and then the monotone bounded convergence theorem implies that
\begin{equation}
	  0<\lim_{t\to \infty}\alpha'(t)<\infty,
\end{equation}
which yields \eqref{sup-alpha}. 
\end{proof}

\begin{lemma} \label{lem-alpha0}
If $\alpha_1>0$, then the solution to \eqref{eq-alpha1} satisfies that $\alpha'(t)>0$ for all $t>0$. 
Moreover, assume that \eqref{def-aff-data3} holds, then there exist positive constants $\beta_1=\alpha_1$, 
$\beta_2=\beta_2( \gamma, c_1,  \alpha_0, \alpha_1 )$, and $\beta_3=\beta_3(A,  \gamma, c_1, \varpi, \alpha_0, \alpha_1)$
 such that 
$$ \beta_1 \le \alpha'(t) \le \beta_2,
$$ 
and 
$$\beta_3 \alpha\leq \alpha - \frac{2\bar\mu+3\bar\lambda}{A}\alpha' \leq \alpha$$
for all $t>0$, i.e. \eqref{ineq-alpha1} and \eqref{equiv-tilde-alpha1} in the new coordinate of $\tau$.
\end{lemma}

\begin{proof}
	Given $\alpha_1>0$, the fact $\alpha'(t)\in(0,\beta_2]$ has been proved in \lemref{lem-alpha00}, with $\beta_2$ being defined in \eqref{def-beta2}. 
    It follows from $\eqref{eq-alpha1}_2$ that
    \begin{equation*}
\Big( \dfrac{\alpha'}{\alpha} \Big)'=\dfrac{\alpha''}{\alpha} - \dfrac{(\alpha')^2}{\alpha^2}=c_1 \alpha^{-3\gamma} \Big(\alpha - \dfrac{2\bar\mu+3\bar\lambda}{A} \alpha' \Big) - \dfrac{(\alpha')^2}{\alpha^2}.
    \end{equation*}
Then \eqref{sup-alpha} implies that
$( \alpha' / \alpha )'<0$ for sufficiently large t. Therefore, $\alpha'/ \alpha$ attains its maximum at $t=0$ or its critical points (if any). If $t_2$ is a critical point of $\alpha'/ \alpha$, then
$$  \dfrac{\alpha'}{\alpha}(t_2)=  c_1^{1/2} \alpha^{-3\gamma/2} \Big(\alpha - \dfrac{2\bar\mu+3\bar\lambda}{A} \alpha' \Big)^{1/2} (t_2)\le c_1^{1/2} \alpha_0^{(1-3\gamma)/2} .$$
Consequently,
$$\alpha \ge \alpha - \dfrac{2\bar\mu+3\bar\lambda}{A}\alpha' \ge \min\{ 1- \dfrac{\varpi}{A}\dfrac{\alpha_1}{\alpha_0}, 1-\dfrac{\varpi}{A}c_1^{1/2} \alpha_0^{(1-3\gamma)/2} \} \alpha=: \beta_3 \alpha.$$
Moreover, together with the equation of $\alpha$, $\eqref{eq-alpha1}_2$, it follows from that $\alpha''(t)>0$ for $t\geq 0$, and thus $\alpha'(t)\geq \alpha'(0)=\alpha_1$. 
By setting $\beta_1=\alpha_1$, the proof of the lemma is finished.
\end{proof}

Denote by $\alpha^{\bar\mu,\bar\lambda}$ the solution to the function $\eqref{eq-alpha2}_2$ with initial data \eqref{data-vanish-alpha} and $\alpha^0$ the solution to the same problem with $\bar\mu=2\bar\mu + 3\bar\lambda=0$ and with the same initial data. Then we have the following estimates of $\alpha^{\bar\mu,\bar\lambda}/\alpha^0$.

\begin{lemma} \label{lem-check-alpha}
Assume that \eqref{def-aff-data3} holds. Then for any $\tau \in [0,T)$, the following inequalities hold:
\begin{equation} \label{ineq-check-alpha1}\\ 
\begin{alignedat}{3}
    &0\le  1-\frac{\alpha^{\bar\mu,\bar\lambda}}{\alpha^0}  &&\le  C_T(2\bar\mu+3\bar\lambda), \\
    &0 \le \frac{\alpha^0_\tau}{\alpha^0} - \frac{\alpha^{\bar\mu,\bar\lambda}_\tau}{\alpha^{\bar\mu,\bar\lambda}} && \le  C_T(2\bar\mu+3\bar\lambda).    
\end{alignedat}
\end{equation}
\end{lemma}
\begin{proof}
It suffices to consider the case of $2\bar\mu+3\bar\lambda>0$. 
Set
$$\check \alpha:=\frac{\alpha^{\bar\mu,\bar\lambda}}{\alpha^0},$$ 
and then by \eqref{eq-alpha2}, $\check \alpha$ satisfies the following equation
\begin{equation} \label{eq-check-alpha}
(\ln\check \alpha)_{\tau\tau}  = c_1 (\alpha^0)^{3-3\gamma}\big[(\check\alpha)^{3-3\gamma}-1 \big] -(2\bar\mu+3\bar\lambda)\frac{c_1}{A} \alpha^{\bar\mu,\bar\lambda}_\tau (\alpha^{\bar\mu,\bar\lambda})^{1-3\gamma}
\end{equation}
with initial data
\begin{equation} \label{data-check-alpha}
    (\check\alpha, \check\alpha_\tau)|_{\tau=0}=(1,0).
\end{equation}
At the same time, \eqref{eq-alpha-int} in $\tau$ implies that
\begin{equation} \label{eq-check-alpha-int}
\begin{aligned}
&\Big(\frac{\alpha^{\bar\mu,\bar\lambda}_\tau}{\alpha^{\bar\mu,\bar\lambda}} + \frac{\alpha^0_\tau}{\alpha^0} \Big)(\ln \check\alpha)_\tau \\
= &\frac{2c_1}{3\gamma-3} (\alpha^0)^{3-3\gamma}\big[1- (\check\alpha)^{3-3\gamma} \big] - 2(2\bar\mu+3\bar\lambda)\frac{c_1}{A} \int_0^\tau (\alpha^{\bar\mu,\bar\lambda}_\tau)^2 (\alpha^{\bar\mu,\bar\lambda})^{-3\gamma} \,d\tau'.
\end{aligned}
\end{equation}

Firstly, we claim that for any $\tau>0$,
\begin{equation}
    \label{ineq-check-alpha-ln}
    \ln \check \alpha, ~(\ln \check \alpha)_\tau <0.
\end{equation} 
Indeed, \eqref{eq-check-alpha} and \eqref{data-check-alpha} imply $$\ln \check \alpha(0)= (\ln \check \alpha)_\tau(0)=0,~ (\ln \check \alpha)_{\tau\tau}(0)<0.$$ 
Therefore, there exists $\tau_0>0$ such that for any $\tau \in (0,\tau_0)$, \eqref{ineq-check-alpha-ln} holds. 
If $\ln \check \alpha <0$ fails at some time $\tau>0$, set $\tau_1:= \inf\{\tau>\tau_0 | \ln \check \alpha(\tau)=0 \}$, then 
\begin{equation*} 
\ln \check \alpha(\tau_1)=0,\quad (\ln \check \alpha)_\tau(\tau_1)\ge0,
\end{equation*}
which contradicts with \eqref{eq-check-alpha-int}. As a result, $\ln \check \alpha <0$ for all $\tau>0$. Then it follows from \eqref{eq-check-alpha-int} that $(\ln \check \alpha)_\tau <0$ for all $\tau>0$.

On the other hand, since $\alpha^{\bar\mu,\bar\lambda}_\tau>0$, we have $\alpha^{\bar\mu,\bar\lambda} \ge \alpha_0$ and $\alpha^{0} \le \alpha^0(T)$, from which it yields that
\begin{align*}
  &\alpha^{\bar\mu,\bar\lambda}_\tau (\alpha^{\bar\mu,\bar\lambda})^{1-3\gamma} \le \beta_2 (\alpha^{\bar\mu,\bar\lambda})^{2-3\gamma} \le \beta_2 (\alpha_0)^{2-3\gamma} \,,\\
  &C_T=  \frac{\alpha_0}{\alpha^0(T)} \le \check \alpha \le 1, \quad (\check\alpha)^{3-3\gamma}-1\ge 0\,.
\end{align*}
Therefore, it follows from \eqref{eq-check-alpha} 
that 
\begin{equation}
     (\ln \check \alpha)_{\tau\tau} \ge -C_T(2\bar\mu+3\bar\lambda) .
\end{equation}
Together with \eqref{ineq-check-alpha-ln}, it follows that for 
$$-C_T( 2\bar\mu+3\bar\lambda)\le (\ln \check \alpha)_\tau <0 ,$$
as well as for $\check \alpha -1$, 
which proves \eqref{ineq-check-alpha1}.
\end{proof}


\section{Basic energy}
To derive the basic energy estimates, we assume that for some small enough $\epsilon \in (0,1)$, for any $\tau \in (0,T)$,
\begin{equation} \label{apriori-basic}
\max \lbrace \norm{\eta}{\Lnorm{\infty}}, \norm{x\eta_x}{\Lnorm{\infty}}\}<\epsilon.
\end{equation}

It directly follows from \eqref{apriori-basic} that for small $\epsilon$, 
\begin{equation}
	\label{ineq-eta}
	1+\eta, 1+\eta+x\eta_x \in (\frac{1}{2}, \frac{3}{2}).
\end{equation}

Now we present a proposition for the basic tangential estimates.
\begin{proposition} \label{prop-basic} 
Let $\gamma >1$, $2\bar\mu+3\bar\lambda\in[0,\varpi]$, and assume that \eqref{def-aff-data3} holds. Under the a priori assumption \eqref{apriori-basic} with sufficiently small $\epsilon \in (0,1)$, the following basic energy estimate holds for any $r_1 \le \min \{3\gamma-4,1\}$, $\tau \in(0,T)$:
\begin{equation} \label{ineq-basic-energy}
\begin{aligned}
& \quad\alpha^{1+r_1} \int  \bar{\rho}x^4\eta_\tau^2\,dx+(1-r_1)\int_0^\tau \alpha^{r_1} \alpha_\tau\int \bar{\rho}x^4\eta_\tau^2\,dxd\tau'\\& 
+\alpha^{3-3\gamma+r_1}\tilde{\alpha} \int \bar\rho^\gamma (x^4 \eta_x^2 + x^2\eta^2) \,dx\\ &
+ ( 3\gamma- 4 - r_1 )\int_0^\tau  \alpha^{2-3\gamma+r_1}\alpha_\tau\tilde{\alpha}  \int  \bar\rho^\gamma (x^4 \eta_x^2 + x^2\eta^2) \,dxd\tau'\\ &
+\int_0^\tau \alpha^{3-3\gamma+r_1} \int \bar\rho^\gamma x^2 \Big\{ \frac{4}{3}\bar\mu \big[ (1+\eta)x\eta_{x\tau} - x \eta_x\eta_\tau \big]^2\\ 
&\qquad
+\frac{1}{3} (2\bar\mu+3\bar\lambda) \big[ (1+\eta)(\eta_\tau+x\eta_{x\tau})+2 (1+\eta+x\eta_x) \eta_\tau\big]^2 \Big\}\,dxd\tau'\\
 \le  & C (2\bar\mu+3\bar\lambda) \int_0^\tau \alpha^{3-3\gamma+r_1}\int \bar\rho^\gamma \big[x^4\eta_x^2  + x^2\eta^2\big] \,dx d\tau' +C \mathcal{E}_{in}.
\end{aligned}
\end{equation}

\end{proposition}	

\begin{proof}
Multiplying \eqref{eq-Lag} by $\alpha^{r_1}(1+\eta)^2x^3\eta_\tau$ and then integrating with respect to $x$ yields
\begin{align*}
&\frac{1}{2} \frac{d}{d\tau} \alpha^{1+r_1} \int\bar{\rho}x^4\eta_\tau^2\,dx+\frac{1-r_1}{2}\int \alpha^{r_1} \alpha_\tau \bar{\rho}x^4\eta_\tau^2\,dx\\ 
 =& - c_1\alpha^{3-3\gamma+r_1}\tilde\alpha \int  \bar\rho x^4(1+\eta)\eta_\tau\,dx\\ &
+ A\alpha^{3-3\gamma+r_1}\tilde\alpha \int \left(\frac{\bar\rho}{(1+\eta)^2(1+\eta+x\eta_x)}\right)^\gamma\left[(1+\eta)^2x^3\eta_\tau\right]_x\,dx\\ &
- (2\bar\mu+\bar\lambda)\alpha^{3-3\gamma+r_1} \int \left(\frac{\bar\rho}{(1+\eta)^2(1+\eta+x\eta_x)}\right)^\gamma \\ 
& \qquad \qquad \qquad \qquad \qquad \cdot \left(\frac{\eta_\tau+x\eta_{x\tau}}{1+\eta+x\eta_x}+\frac{2\eta_\tau}{1+\eta}\right)\left[(1+\eta)^2x^3\eta_\tau\right]_x\,dx\\ &
+4\bar\mu\alpha^{3-3\gamma+r_1} \int \left(\frac{\bar\rho}{(1+\eta)^2(1+\eta+x\eta_x)}\right)^\gamma\left[(1+\eta)x^3\eta_\tau^2\right]_x\,dx\\ 
=:&L_1+L_2+L_3+L_4.
\end{align*}

By $\eqref{eq-alpha1}_1$, $L_1$ and $L_2$ can be rewritten, respectively, as
\begin{align*}
L_1 =& - A\alpha^{3-3\gamma+r_1}\tilde\alpha \int \bar\rho^\gamma x^2\frac{d}{d\tau}\left[\frac{3}{2}(1+\eta)^2+(1+\eta)x\eta_x\right]\,dx,\\
L_2=&\frac{A}{1-\gamma}\alpha^{3-3\gamma+r_1} \tilde\alpha \int\bar\rho^\gamma x^2\frac{d}{d\tau}\left[(1+\eta)^2(1+\eta+x\eta_x)\right]^{1-\gamma}\,dx.   
\end{align*}
Denote 
$$F_1:=\frac{3}{2}(1+\eta)^2+(1+\eta)x\eta_x+\frac{1}{\gamma-1}\left[(1+\eta)^2(1+\eta+x\eta_x)\right]^{1-\gamma}-\frac{3}{2}-\frac{1}{\gamma-1},$$
then by Taylor's expansion, it then holds
\begin{equation*}
\begin{aligned}
F_1= & \frac{9\gamma-3}{2}\eta^2+(3\gamma-1)\eta x\eta_x+\frac{\gamma}{2}x^2\eta_x^2-\frac{1}{2}(3\gamma-2)(3\gamma-1)\eta^3\\ &
-\frac{1}{2}\gamma(3\gamma-1)\eta x^2\eta_x^2-\frac{1}{2}(3\gamma-2)(3\gamma-1)\eta^2x\eta_x-\frac{1}{6}\gamma(\gamma-1)x^3\eta_x^3+O(\epsilon^4).
\end{aligned}
\end{equation*}
Then there exist positive constants $\tilde C_1,\tilde C_2,\tilde C_3,\tilde C_4$ depending only on $\gamma$ such that for $\epsilon$ sufficiently small,
\begin{equation*}
\tilde C_1x^2\eta_x^2+ \tilde C_2\eta^2 \le F_1 \le \tilde C_3x^2\eta_x^2+ \tilde C_4\eta^2.
\end{equation*}
Therefore,
\begin{align*}
 L_1+L_2=&-\frac{d}{d\tau}A \alpha^{3-3\gamma+r_1} \tilde\alpha \int  \bar\rho^\gamma x^2F_1\,dx + (A \alpha^{3-3\gamma+r_1} \tilde\alpha)_\tau \int \bar\rho^\gamma x^2F_1\,dx    \\   
 =&-\frac{d}{d\tau}A \alpha^{3-3\gamma+r_1} \tilde\alpha \int  \bar\rho^\gamma x^2F_1\,dx \\
 &- A( 3\gamma- 4 - r_1 ) \alpha^{2-3\gamma+r_1}\alpha_\tau\tilde{\alpha}  \int \bar\rho^\gamma x^2F_1\,dx    \\  
 &+ (2\bar\mu+3\bar\lambda)[\alpha^{1-3\gamma+r_1}\alpha_\tau^2-c_1\alpha^{5-6\gamma+r_1}\tilde\alpha]\int \bar\rho^\gamma x^2F_1\,dx \\
 \le & -\frac{d}{d\tau}A \alpha^{3-3\gamma+r_1} \tilde\alpha \int  \bar\rho^\gamma x^2F_1\,dx \\
 &- A( 3\gamma- 4 - r_1 ) \alpha^{2-3\gamma+r_1}\alpha_\tau\tilde{\alpha}  \int \bar\rho^\gamma x^2F_1\,dx    \\  
 &+ C (2\bar\mu+3\bar\lambda) \alpha^{3-3\gamma+r_1}\int \bar\rho^\gamma x^2F_1\,dx. 
\end{align*}

A direct computation shows that
\begin{equation*}
\begin{aligned}
&L_3+L_4\\=& -\alpha^{3-3\gamma+r_1} \int \bar\rho^\gamma x^2 (1+\eta)^{-2\gamma}(1+\eta+x\eta_x)^{-\gamma-1} \Big\{\frac{4}{3}\bar\mu\big[(1+\eta)x\eta_{x\tau}-x\eta_x\eta_\tau\big]^2 \\ 
&\qquad\qquad\qquad
+\frac{1}{3}(2\bar\mu+3\bar\lambda)\big[(1+\eta)(\eta_\tau+x\eta_{x\tau})+2(1+\eta+x\eta_x)\eta_\tau\big]^2 \Big\}\,dx.
\end{aligned}
\end{equation*}

Combining the above estimates, we have, for $4-3\gamma+r_1<0$,
\begin{align*}
&\quad  \frac{d}{d\tau} \frac{1}{2} \alpha^{1+r_1} \int \bar{\rho}x^4 \eta_\tau^2\,dx +\frac{1-r_1}{2} \alpha^{r_1} \alpha_\tau  \int\bar{\rho}x^4\eta_\tau^2\,dx\\ &
+\frac{d}{d\tau}A\alpha^{3-3\gamma+r_1} \tilde\alpha \int  \bar\rho^\gamma \big[\tilde C_1x^4\eta_x^2 + \tilde C_2x^2\eta^2\big]\,dx\\&
+A( 3\gamma- 4 - r_1 ) \alpha^{2-3\gamma+r_1}\alpha_\tau\tilde{\alpha}  \int  \bar\rho^\gamma \big[\tilde C_1x^4\eta_x^2 + \tilde C_2x^2\eta^2\big]\,dx\\ & 
+ \alpha^{3-3\gamma+r_1} \int \bar\rho^\gamma x^2 \Big\{ \frac{4}{3}\bar\mu \big[ (1+\eta)x\eta_{x\tau} - x \eta_x\eta_\tau \big]^2\\ 
&\qquad\quad
+\frac{1}{3} (2\bar\mu+3\bar\lambda) \big[ (1+\eta)(\eta_\tau+x\eta_{x\tau})+2 (1+\eta+x\eta_x) \eta_\tau\big]^2 \Big\}\,dxd\tau'\\
\le& C (2\bar\mu+3\bar\lambda) \alpha^{3-3\gamma+r_1}\int \bar\rho^\gamma \big[\tilde C_3x^4\eta_x^2 + \tilde C_4x^2\eta^2\big] \,dx .
\end{align*}

Then the basic energy estimate \eqref{ineq-basic-energy} is proved.
\end{proof}

\section{Higher-order energy estimates}

In this section, we derive higher-order energy estimates. To this end, we set the following a prior assumption: for some positive constant $ \epsilon \in (0,1) $, which is small enough, such that for any $ \tau \in (0,T) $,
\begin{align}
  &\begin{aligned}	\label{def-apriori-assumx}
  \sup_{0 \leq \tau \leq T}  \norm{ (\eta,\eta_x) }{\Lnorm{\infty}} &+ \sup_{0 \leq \tau \leq T}\sum_{k=2}^{N_0-3} \norm{  \bar\rho^{( \gamma - 1 )( k - \frac{3}{2})} \partial_x^{k} \eta }{\Lnorm{\infty}} \\
    &+ \sup_{0 \leq \tau \leq T} \norm{  \bar\rho^{( \gamma - 1 )(N_0-\frac{7}{2}) } x\partial_x^{N_0-2} \eta }{\Lnorm{\infty}}  < \epsilon,\\
    \end{aligned}\\
&\begin{aligned} \label{def-apriori-assumt}
    \sup_{0 \leq \tau \leq T}  \norm{ (\eta_\tau,\eta_{x\tau}) }{\Lnorm{\infty}} &+ \sup_{0 \leq \tau \leq T}\sum_{k=2}^{N_0-3} \norm{  \bar\rho^{( \gamma - 1 )( k - \frac{3}{2})} \partial_x^{k} \eta_\tau }{\Lnorm{\infty}} \\
    &+ \sup_{0 \leq \tau \leq T} \norm{  \bar\rho^{( \gamma - 1 )(N_0-\frac{7}{2})} x\partial_x^{N_0-2} \eta_\tau }{\Lnorm{\infty}}  < \epsilon,\\
    \end{aligned}\\
\intertext{and}
&\begin{aligned} \label{def-apriori-assum1}
    \sup_{0 \leq \tau \leq T} \alpha^{\min\{ \frac{3\gamma-3}{2}, 1\}} \norm{ (\eta_\tau,\eta_{x\tau}) }{\Lnorm{\infty}}   < \epsilon,
    \end{aligned}
\end{align}
This a priori assumption will dramatically simplify the presentation of our proof, and it is closed in the sense that it can be bounded by the energy functionals. Thus, with a continuity argument, one can conclude that both the energy estimates and the a priori assumptions hold.

To prove the theorems, it suffices to prove the following a priori estimates. Depending on the parameter $\gamma$, we derive different a priori estimates by employing time-dependent weights.
\begin{proposition}\label{prop-high}
Let $\gamma >1$, $2\bar\mu+3\bar\lambda\in[0,\varpi]$, and assume that \eqref{def-aff-data3} holds. Under the a priori assumption \eqref{def-apriori-assumx}--\eqref{def-apriori-assum1} with sufficiently small $\epsilon$, the following energy estimates hold for any $r_1 \le \min\{3\gamma-4,1 \}$, $\tau \in (0,T)$ and $1 \le N \le N_0$:
    \begin{equation} \label{ineq-energy1}
    \begin{aligned}
    &\mathcal{E}_{N} + (1-r_1)\mathcal{E}_{N}^{(1)} + (4-3\gamma+r_1) \mathcal{E}_{N}^{(2)} + (2\bar\mu+\bar\lambda)\mathcal{D}_N \\
    \le & C(2\bar\mu+\bar\lambda) \int_0^\tau \alpha^{-1} \mathcal{E}_{N} \,d\tau' 
    +C \int_0^\tau \alpha^{-\min\{ \frac{3\gamma-3}{2}, 1\} } \mathcal{E}_N\,d\tau' 
    + C \mathcal{E}_{in} \, .
    \end{aligned}  
    \end{equation}

\end{proposition}

\begin{remark}
It should be noted that $\mathcal{E}_N^{(1)}(\tau)$ and $\mathcal{E}_N^{(2)}(\tau)$ are not used to control any lower-order terms in the derivation of the above energy inequality \eqref{ineq-energy1} nor to verify the a priori assumptions \eqref{def-apriori-assumx}--\eqref{def-apriori-assum1} in the subsequent proof. Therefore, uniform estimates with respect to $r_1$ can be derived. The estimates of $\mathcal{E}_N^{(1)}(\tau)$ or $\mathcal{E}_N^{(2)}(\tau)$ are merely byproducts in the special cases $r_1<1$ or $r_1<3\gamma-4$, respectively.

\end{remark}

\begin{remark} \label{rmk-visc2}
Observe that $$2\bar\mu+\bar\lambda=\frac{2\bar\mu + 3\bar\lambda}{3} +\frac{4}{3}\bar\mu,$$
this together with the positive definiteness of the viscosity coefficients $\bar\mu, 2\bar\mu + 3\bar\lambda$ in \eqref{vis-coeff} implies that as both $\bar\mu$ and $2\bar\mu+3\bar\lambda$ tend to 0,
$$2\bar\mu+\bar\lambda \to 0.$$ 
Consequently, the above energy estimates are all uniform with respect to bounded $\bar\mu$ and $2\bar\mu+3\bar\lambda$.
\end{remark}
\subsection{Estimates on lower order terms}
Before performing higher-order estimates for the perturbed equation, we first establish several preliminary estimates on lower-order terms—particularly those involving Jacobians—under the assumption \eqref{def-apriori-assumx} and \eqref{def-apriori-assumt}.

\begin{lemma} \label{lem-Jacobx}
	Under the a priori assumption \eqref{def-apriori-assumx}, it holds that 
	\begin{enumerate}[label = \rmfamily(\roman*),ref = \rmfamily(\roman*)]
	\item For $1 \le s \le N_0+1 $,
	\begin{equation}
		\label{ineq-Jacobx0}
    \begin{aligned}
		&\abs{\partial_x^{s} [ (1+\eta)^{\nu_1} (1+\eta+x\eta_x)^{\nu_2} ] }\\
        =& \mathcal{O}(1) \Big[ x \abs{ \partial_x^{s+1} \eta} + \abs{ \partial_x^{s} \eta} + \epsilon \sum_{j=3}^{s} \bar\rho^{(\gamma-1)(j-s-\frac{1}{2} )}(x\abs{\partial_x^{j} \eta} + \abs{ \partial_x^{j-1} \eta } ) \Big]\\
            &+\mathcal{O}(1)\epsilon  \bar\rho^{(\gamma-1)(\frac{-s+1}{2} )}(x\abs{\partial_x^{2} \eta} + \abs{ \partial_x \eta } ) .     
    \end{aligned}
	\end{equation}	
		
	\item For $1 \le s \le N_0$,
	\begin{align}
    &\begin{aligned} \label{ineq-Jacobx1}
		&\partial_x^{s} [ (1+\eta)^{\nu_1} (1+\eta+x\eta_x)^{\nu_2} x \eta_{xx} ] \\
       =&(1+\eta)^{\nu_1} (1+\eta+x\eta_x)^{\nu_2}(x\partial_x^{s+2} \eta + s\partial_x^{s+1} \eta) \\
        &+\mathcal{O}(1) \epsilon \Big[ \sum_{j=3}^{s+1} \bar\rho^{(\gamma-1)(j-s-\frac{3}{2})}(x\abs{\partial_x^{j}  \eta} + \abs{\partial_x^{j-1}  \eta}) +  \bar\rho^{(\gamma-1)\frac{-s}{2}} x\abs{\partial_x^{2} \eta}   \Big],
        \end{aligned} 
        \\
\intertext{and}		
    &\begin{aligned} \label{ineq-Jacobx2}
		&\partial_x^{s} [ (1+\eta)^{\nu_1} (1+\eta+x\eta_x)^{\nu_2} \eta_{x} ]  \\ 
        =&(1+\eta)^{\nu_1} (1+\eta+x\eta_x)^{\nu_2}\partial_x^{s+1} \eta \\
        &+\mathcal{O}(1) \epsilon \Big[ \sum_{j=3}^{s+1} \bar\rho^{(\gamma-1)(j-s-\frac{3}{2})} \abs{\partial_x^{j-1}  \eta} +  \bar\rho^{(\gamma-1)\frac{-s}{2} } \abs{ \partial_x \eta } \Big].\\
	\end{aligned}
	\end{align}
	
    \end{enumerate}
Here and thereafter, we use the convention that $\sum_{j=j_1}^{j_2}f=0$ for $j_2<j_1$.
\end{lemma}  
\begin{proof}

We prove \eqref{ineq-Jacobx0} by induction on $s$.

By virtue of \eqref{ineq-eta} and \eqref{def-apriori-assumx} for $\partial_x^2 \eta$,  it holds that
\begin{align*}		
& \abs{\partial_x [ (1+\eta)^{\nu_1} (1+\eta+x\eta_x)^{\nu_2} ] } =\mathcal{O}(1) ( x\abs{\partial_x^2 \eta} + \abs{\partial_x \eta}),\\
& \abs{\partial_x^{2} [ (1+\eta)^{\nu_1} (1+\eta+x\eta_x)^{\nu_2} ] } =\mathcal{O}(1) \Bigl[ x \abs{\partial_x^3 \eta} + \abs{\partial_x^2 \eta} + ( x\abs{\partial_x^2 \eta} + \abs{\partial_x \eta})^2 \Bigr] \\
&\qquad\qquad\qquad=\mathcal{O}(1) \Bigl[ x \abs{\partial_x^3 \eta} + \abs{\partial_x^2 \eta}+\epsilon  \bar\rho^{-\frac{(\gamma-1)}{2}}(x\abs{\partial_x^{2} \eta} + \abs{ \partial_x \eta } )\Bigr]  ,\\
& \abs{\partial_x^{3} [ (1+\eta)^{\nu_1} (1+\eta+x\eta_x)^{\nu_2} ] }\\
=& \mathcal{O}(1) \Bigl[ x \abs{\partial_x^4 \eta} + \abs{\partial_x^3 \eta} + (x \abs{\partial_x^3 \eta} + \abs{\partial_x^2 \eta}) ( x\abs{\partial_x^2 \eta} + \abs{\partial_x\eta} ) + ( x\abs{\partial_x^2 \eta} + \abs{\partial_x\eta} )^3 \Bigr] \\
=& \mathcal{O}(1) \Bigl[ x \abs{\partial_x^4 \eta} + \abs{\partial_x^3 \eta} + \epsilon\bar\rho^{-\frac{(\gamma-1)}{2}}(x \abs{\partial_x^3 \eta} + \abs{\partial_x^2 \eta})  + \epsilon \bar\rho^{-(\gamma-1)} ( x\abs{\partial_x^2 \eta} + \abs{\partial_x\eta} ) \Bigr]  .
\end{align*}
Hence, \eqref{ineq-Jacobx0} holds for $s=1,2,3$. 

Assume now that for all $1 \le s \le s_0,$ where $s_0 \le N_0$, \eqref{ineq-Jacobx0} holds. Then direct computation shows that
\begin{align*}
    &\partial_x^{s_0+1} [ (1+\eta)^{\nu_1} (1+\eta+x\eta_x)^{\nu_2} ]\\
   =&\partial_x^{s_0} [ \nu_1 (1+\eta)^{\nu_1-1} (1+\eta+x\eta_x)^{\nu_2} \eta_x + \nu_2(1+\eta)^{\nu_1} (1+\eta+x\eta_x)^{\nu_2-1}(x\eta_{xx}+2\eta_x)]\\
   =&  \nu_1 (1+\eta)^{\nu_1-1} (1+\eta+x\eta_x)^{\nu_2} \partial_x^{s_0+1}\eta \\
    & + \nu_2(1+\eta)^{\nu_1} (1+\eta+x\eta_x)^{\nu_2-1}\big( x\partial_x^{s_0+2}\eta+(s_0+2)\partial_x^{s_0+1}\eta \big)\\
    & + \nu_1 \sum_{i=1}^{s_0}  \partial_x^{i} [ (1+\eta)^{\nu_1-1} (1+\eta+x\eta_x)^{\nu_2} ] \partial_x^{s_0-i+1}\eta \\
    & + \nu_2 \sum_{i=1}^{s_0}  \partial_x^{i} [(1+\eta)^{\nu_1} (1+\eta+x\eta_x)^{\nu_2-1}] \big( x\partial_x^{s_0-i+2}\eta+(s_0-i+2)\partial_x^{s_0-i+1}\eta \big).
\end{align*} 
Applying \eqref{ineq-Jacobx0} for $1\leq i\leq s_0$ in the above equality, we obtain 
\begin{align*}
    &\abs{ \partial_x^{s_0+1} [ (1+\eta)^{\nu_1} (1+\eta+x\eta_x)^{\nu_2} ] }\\
   =& \mathcal{O}(1) \Big[x \abs{ \partial_x^{s_0+2} \eta} + \abs{ \partial_x^{s_0+1} \eta} \Big] \\
    & +\mathcal{O}(1)\sum_{i=1}^{s_0} \Big[ x \abs{ \partial_x^{i+1} \eta} + \abs{ \partial_x^{i} \eta} + \epsilon \sum_{j=3}^{i} \bar\rho^{(\gamma-1)(j-i-\frac{1}{2} )}(x\abs{\partial_x^{j} \eta} + \abs{ \partial_x^{j-1} \eta } ) \\
   &\qquad\qquad\quad+ \epsilon  \bar\rho^{(\gamma-1)(\frac{-i+1}{2} )}(x\abs{\partial_x^{2} \eta} + \abs{ \partial_x \eta }) \Big] \cdot \big( x\abs{\partial_x^{s_0-i+2}\eta} + \abs{\partial_x^{s_0-i+1}\eta} \big) .
\end{align*}
The summations of the three terms in the above square bracket are treated slightly differently in the following estimates.
Denote $k_0:=\lceil \frac{s_0+1}{2} \rceil$. Since $N_0\geq 6$, then $\frac{s_0+1}{2}>s_0-N_0+3$, and thus $s_0-i+2\leq N_0-2$ for $i\geq k_0$. Hence, by symmetry and \eqref{def-apriori-assumx}, we have 
\begin{align}
&\sum_{i=1}^{s_0} \big( x \abs{ \partial_x^{i+1} \eta} + \abs{ \partial_x^{i} \eta} \big)  \big( x\abs{\partial_x^{s_0-i+2}\eta} + \abs{\partial_x^{s_0-i+1}\eta} \big) \notag\\
\leq &C \epsilon \sum_{i=k_0}^{s_0} \bar\rho^{(\gamma-1)(i-s_0-\frac{1}{2} )} \big( x \abs{ \partial_x^{i+1} \eta} + \abs{ \partial_x^{i} \eta} \big)  \,. \label{ineq-Jacobx01}
\end{align}
For the 2nd term, noting that $\sum_{j=j_1}^{j_2}f=0$ for $j_2<j_1$, and $s_0-i+2 \leq N_0-2$ for $i>3$, we divide it into $i=3$ and $i>3$:  
\begin{align}
    &\sum_{i=1}^{s_0} \sum_{j=3}^{i} \bar\rho^{(\gamma-1)(j-i-\frac{1}{2} )}(x\abs{\partial_x^{j} \eta} + \abs{ \partial_x^{j-1} \eta } )  \big( x\abs{\partial_x^{s_0-i+2}\eta} + \abs{\partial_x^{s_0-i+1}\eta} \big) \notag\\
   \leq & \bar\rho^{(\gamma-1)(-\frac{1}{2} )}(x\abs{\partial_x^{3} \eta} + \abs{ \partial_x^{2} \eta } )  ( x\abs{\partial_x^{s_0-1}\eta} + \abs{\partial_x^{s_0-2}\eta} ) \notag \\
   &\qquad\qquad+ C\epsilon \sum_{i=4}^{s_0} \sum_{j=3}^{i} \bar\rho^{(\gamma-1)(j-i-\frac{1}{2} )}(x\abs{\partial_x^{j} \eta} + \abs{ \partial_x^{j-1} \eta } ) \bar\rho^{(\gamma-1)(i-s_0-\frac{1}{2} )} \notag \\
   \leq &C\epsilon \bar\rho^{(\gamma-1)(-2)} ( x\abs{\partial_x^{s_0-1}\eta} + \abs{\partial_x^{s_0-2}\eta} ) \notag \\
   &\qquad\qquad+ C \epsilon\sum_{j=3}^{s_0} \bar\rho^{(\gamma-1)(j-s_0-1)}(x\abs{\partial_x^{j} \eta} + \abs{ \partial_x^{j-1} \eta } )\notag \\
   \leq & C \epsilon\sum_{j=3}^{s_0} \bar\rho^{(\gamma-1)(j-s_0-1)}(x\abs{\partial_x^{j} \eta} + \abs{ \partial_x^{j-1} \eta } ). \label{ineq-Jacobx02}
\end{align}
For the 3rd term,   
\begin{align}
   & \sum_{i=1}^{s_0}   \epsilon  \bar\rho^{(\gamma-1)(\frac{-i+1}{2} )}(x\abs{\partial_x^{2} \eta} + \abs{ \partial_x \eta }) \cdot \big( x\abs{\partial_x^{s_0-i+2}\eta} + \abs{\partial_x^{s_0-i+1}\eta} \big) \notag \\
   \leq & C\epsilon  \sum_{i=1}^{s_0} \bar\rho^{(\gamma-1)(\frac{-i}{2}) } \big( x\abs{\partial_x^{s_0-i+2}\eta} + \abs{\partial_x^{s_0-i+1}\eta} \big)   \notag\\
   \leq & C\epsilon  \sum_{j=3}^{s_0+1} \bar\rho^{(\gamma-1)(\frac{j-s_0-2}{2}) } \big( x\abs{\partial_x^{j}\eta} + \abs{\partial_x^{j-1}\eta} \big) + C\epsilon  \bar\rho^{(\gamma-1)(\frac{-s_0}{2}) } \big( x\abs{\partial_x^{2}\eta} + \abs{\partial_x \eta} \big). \label{ineq-Jacobx03}
\end{align}
Combining \eqref{ineq-Jacobx01}-\eqref{ineq-Jacobx03}, we have 
\begin{align*}
    &\abs{ \partial_x^{s_0+1} [ (1+\eta)^{\nu_1} (1+\eta+x\eta_x)^{\nu_2} ] }\\
=& \mathcal{O}(1) \Big[ x \abs{ \partial_x^{s_0+2} \eta} + \abs{ \partial_x^{s_0+1} \eta} + \epsilon \sum_{j=3}^{s_0+1} \bar\rho^{(\gamma-1)(j-s_0-\frac{3}{2} )}(x\abs{\partial_x^{j} \eta} + \abs{ \partial_x^{j-1} \eta } ) \Big]\\
            &+\mathcal{O}(1)\epsilon  \bar\rho^{(\gamma-1)(\frac{-s_0}{2} )}(x\abs{\partial_x^{2} \eta} + \abs{ \partial_x \eta } ),
\end{align*}
which proves \eqref{ineq-Jacobx0}.  A similar argument yields \eqref{ineq-Jacobx1} and \eqref{ineq-Jacobx2}.

\end{proof}

\begin{lemma} \label{lem-Jacobt}
	Under the a priori assumption \eqref{def-apriori-assumt}, it holds that, for $1 \le s \le N_0+1$,
    \begin{align}
    &\begin{aligned} &\label{ineq-Jacobt1}
		\partial_x^{s}  [ (1+\eta)^{\nu_1} (1+\eta+x\eta_x)^{\nu_2}  x\eta_{x\tau} ] \\
		  = &(1+\eta)^{\nu_1} (1+\eta+x\eta_x)^{\nu_2} ( x \partial_x^{s+1} \eta_\tau +s \partial_x^{s} \eta_\tau )  \\
          & + \nu_2 (1+\eta)^{\nu_1} (1+\eta+x\eta_x)^{\nu_2-1} [ x \partial_x^{s+1}\eta + (s+1)\partial_x^{s}\eta ] \cdot x\eta_{x\tau}\\
        & + \nu_1(1+\eta)^{\nu_1-1} (1+\eta+x\eta_x)^{\nu_2} \partial_x^{s}\eta \cdot x\eta_{x\tau}\\
         &+ \mathcal{O} (1) \epsilon \Big[\sum_{j=3}^{s} \bar\rho^{(\gamma-1)(j-s-\frac{1}{2} )} (   x\abs{ \partial_x^{j} \eta_\tau } + \abs{ \partial_x^{j-1} \eta_\tau } +  x \abs{ \partial_x^{j} \eta} +  \abs{ \partial_x^{j-1} \eta})\Big]\\
         &+\mathcal{O} (1) \epsilon \bar\rho^{(\gamma-1)(\frac{-s+1}{2} )}(x\abs{\partial_x^{2} \eta_\tau} + \abs{ \partial_x \eta_\tau }+ x\abs{\partial_x^{2} \eta} + \abs{ \partial_x \eta } ) ,
    \end{aligned}
\intertext{and}
   & \begin{aligned} & \label{ineq-Jacobt0}
		\partial_x^{s}  [ (1+\eta)^{\nu_1} (1+\eta+x\eta_x)^{\nu_2}  \eta_{\tau} ]  \\
          = & (1+\eta)^{\nu_1} (1+\eta+x\eta_x)^{\nu_2}  \partial_x^{s} \eta_\tau   \\
          & + \nu_2 (1+\eta)^{\nu_1} (1+\eta+x\eta_x)^{\nu_2-1} [ x \partial_x^{s+1}\eta + (s+1)\partial_x^{s}\eta ] \cdot \eta_{\tau}\\
        & + \nu_1(1+\eta)^{\nu_1-1} (1+\eta+x\eta_x)^{\nu_2} \partial_x^{s}\eta \cdot \eta_{\tau}\\
         &+ \mathcal{O} (1) \epsilon \Big[\sum_{j=3}^{s} \bar\rho^{(\gamma-1)(j-s-\frac{1}{2} )} (    \abs{ \partial_x^{j-1} \eta_\tau } +  x \abs{ \partial_x^{j} \eta} +  \abs{ \partial_x^{j-1} \eta}) \Big]\\
         &+ \mathcal{O} (1) \epsilon\bar\rho^{(\gamma-1)(\frac{-s+1}{2} )}( \abs{ \partial_x \eta_\tau }+ x\abs{\partial_x^{2} \eta} + \abs{ \partial_x \eta } ).
    \end{aligned}
	\end{align}
\end{lemma}  

\begin{proof} 
For $1 \le s \le N_0+1$, one can get
\begin{align*}
         &  \partial_x^{s}  [ (1+\eta)^{\nu_1} (1+\eta+x\eta_x)^{\nu_2}  x\eta_{x\tau} ]  \\
        =& \sum_{i=0}^{s} C_s^i \partial_x^{i} [ (1+\eta)^{\nu_1} (1+\eta+x\eta_x)^{\nu_2} ]  \cdot \partial_x^{s-i} (x\eta_{x\tau}) \\
        =& (1+\eta)^{\nu_1} (1+\eta+x\eta_x)^{\nu_2} ( x \partial_x^{s+1} \eta_\tau +s \partial_x^{s} \eta_\tau ) +\partial_x^{s} [ (1+\eta)^{\nu_1} (1+\eta+x\eta_x)^{\nu_2} ] \cdot x\eta_{x\tau} \\
        &+ \sum_{i=1}^{s-1} C_s^i\partial_x^{i} [ (1+\eta)^{\nu_1} (1+\eta+x\eta_x)^{\nu_2} ]  \cdot \partial_x^{s-i} (x\eta_{x\tau}).
\end{align*}
Direct computations show that
    \begin{align*}
         & \partial_x^{s} [ (1+\eta)^{\nu_1} (1+\eta+x\eta_x)^{\nu_2} ]\\
        =& \partial_x^{s-1}[ \nu_2 (1+\eta)^{\nu_1} (1+\eta+x\eta_x)^{\nu_2-1} ( x \eta_{xx} + 2\eta_x) + \nu_1(1+\eta)^{\nu_1-1} (1+\eta+x\eta_x)^{\nu_2} \eta_x  ]\\
        = &  \nu_2 (1+\eta)^{\nu_1} (1+\eta+x\eta_x)^{\nu_2-1} [ x \partial_x^{s+1}\eta + (s+1)\partial_x^{s}\eta ] \\
        & + \nu_1(1+\eta)^{\nu_1-1} (1+\eta+x\eta_x)^{\nu_2} \partial_x^{s}\eta \\
        & + \nu_2 [ \partial_x^{s-1},  (1+\eta)^{\nu_1} (1+\eta+x\eta_x)^{\nu_2-1} ] ( x \eta_{xx} + 2\eta_x) \\
        & + \nu_1 [ \partial_x^{s-1}, (1+\eta)^{\nu_1-1} (1+\eta+x\eta_x)^{\nu_2}] \eta_x,
    \end{align*}
    where  $$[\partial, f]g=\partial (fg)-f\partial g$$
    denotes the standard commutator.
    It then follows from \eqref{ineq-Jacobx0} that
    \begin{align*}
    & \partial_x^{s} [ (1+\eta)^{\nu_1} (1+\eta+x\eta_x)^{\nu_2} ]\\
         = &  \nu_2 (1+\eta)^{\nu_1} (1+\eta+x\eta_x)^{\nu_2-1} [ x \partial_x^{s+1}\eta + (s+1)\partial_x^{s}\eta ] \\
        & + \nu_1(1+\eta)^{\nu_1-1} (1+\eta+x\eta_x)^{\nu_2} \partial_x^{s}\eta \\
        & + \mathcal{O} (1) \epsilon \Big[  \sum_{j=3}^{s} \bar\rho^{(\gamma-1)(j-s-\frac{1}{2} )} (x \abs{ \partial_x^j \eta } +  \abs{ \partial_x^{j-1} \eta }) + \bar\rho^{(\gamma-1)(\frac{-s+1}{2} )}( x\abs{\partial_x^{2} \eta} + \abs{ \partial_x \eta } ) \Big],
\intertext{and}
& \sum_{i=1}^{s-1} C_s^i \partial_x^{i} [ (1+\eta)^{\nu_1} (1+\eta+x\eta_x)^{\nu_2} ]  \cdot \partial_x^{s-i} (x\eta_{x\tau})\\  
= &\mathcal{O}(1) \sum_{i=1}^{s-1} \Big[ x \abs{ \partial_x^{s-i+1} \eta} + \abs{ \partial_x^{s-i} \eta} + \epsilon \sum_{j=3}^{s-i} \bar\rho^{(\gamma-1)(j-s+i-\frac{1}{2} )}(x\abs{\partial_x^{j} \eta} + \abs{ \partial_x^{j-1} \eta } )\\
            & \qquad+\mathcal{O}(1)\epsilon  \bar\rho^{(\gamma-1)(\frac{-s+i+1}{2} )}(x\abs{\partial_x^{2} \eta} + \abs{ \partial_x \eta } ) \Big] (x\abs{\partial_x^{i+1}\eta_\tau} + i \abs{\partial_x^{i}\eta_\tau}).
\intertext{Therefore, note that $N_0 \ge 6$, one has}
         &  \partial_x^{s}  [ (1+\eta)^{\nu_1} (1+\eta+x\eta_x)^{\nu_2}  x\eta_{x\tau} ]  \\
         =&  (1+\eta)^{\nu_1} (1+\eta+x\eta_x)^{\nu_2} ( x \partial_x^{s+1} \eta_\tau +s \partial_x^{s} \eta_\tau )  \\
         & + \nu_2 (1+\eta)^{\nu_1} (1+\eta+x\eta_x)^{\nu_2-1} [ x \partial_x^{s+1}\eta + (s+1)\partial_x^{s}\eta ] \cdot x\eta_{x\tau}\\
        & + \nu_1(1+\eta)^{\nu_1-1} (1+\eta+x\eta_x)^{\nu_2} \partial_x^{s}\eta \cdot x\eta_{x\tau}\\
         &+ \mathcal{O} (1) \epsilon \Big[\sum_{j=3}^{s} \bar\rho^{(\gamma-1)(j-s-\frac{1}{2} )} (   x\abs{ \partial_x^{j} \eta_\tau } + \abs{ \partial_x^{j-1} \eta_\tau } +  x \abs{ \partial_x^{j} \eta} +  \abs{ \partial_x^{j-1} \eta})\Big]\\
         & +\mathcal{O}(1) \epsilon\bar\rho^{(\gamma-1)(\frac{-s+1}{2} )}(x\abs{\partial_x^{2} \eta_\tau} + \abs{ \partial_x \eta_\tau }+ x\abs{\partial_x^{2} \eta} + \abs{ \partial_x \eta } ) \Big], 
	\end{align*}
	which is exactly \eqref{ineq-Jacobt1}. Similar computations also show	\eqref{ineq-Jacobt0}.
\end{proof}

Building upon the above estimates, we now proceed to separate out the higher- and lower-order terms in the derivatives of the perturbed pressure and the viscous stress tensor.

To handle the perturbed pressure, we set
\begin{equation}
	\label{def-Qfrak}
	\mathfrak{Q}:=(1+\eta)^{-2\gamma}(1+\eta+x\eta_x)^{-\gamma}-1.
\end{equation}

\begin{lemma} \label{lem-Px}
	Under the a priori assumption \eqref{def-apriori-assumx}, it holds that 
	\begin{enumerate}[label = \rmfamily(\roman*),ref = \rmfamily(\roman*)]
		\item For $1 \le s \le N_0+1$, it holds that
		\begin{equation}
			\label{ineq-Px1}
          \begin{aligned}
          	\partial_x^s \Qfrak= &- \gamma (1+\eta)^{-2\gamma} (1+\eta+x\eta_x)^{-\gamma-1} [x\partial_x^{s+1}\eta+(s+3)\partial_x^s\eta ]   \\ 
            &	+\mathcal{O}(1) \epsilon \Big[\sum_{j=3}^{s} \bar\rho^{(\gamma-1)(j-s-\frac{1}{2})} (x \abs{ \partial_x^j \eta } + \abs{ \partial_x^{j-1} \eta } )\\
            &\quad+\bar\rho^{(\gamma-1)\frac{-s+1}{2}}(x\abs{\partial_x^2\eta}+\abs{\partial_x\eta}) \Big].
          \end{aligned}
		\end{equation}	
		
		\item    
        For $1 \le N \le N_0$, it holds that
		\begin{equation}
			\label{ineq-Px2}
			\begin{aligned}
				&\partial_x \Big\{ \bigl[\rhobar^{(\gamma-1)(N+1)+1}   (x \partial_x^{N}+N \partial_x^{N-1}) \Qfrak \bigr] (1 + \eta)^2 \Big\} \\
			=	& - \gamma \partial_x \Big\{  \rhobar^{(\gamma-1)(N+1)+1}  (1+\eta)^{-2\gamma+2} (1+\eta+x\eta_x)^{-\gamma-1} G_N   \Big\}   \\ 
			  &+ \mathcal{O}(1) \epsilon  \sum_{j=3}^{N+1} \bar\rho^{(\gamma-1)(j -\frac{1}{2})+1} ( x^2 \abs{ \partial_x^j \eta } + x \abs{ \partial_x^{j - 1} \eta } ) \\
	    & 
	    +\mathcal{O}(1) \epsilon  \sum_{j=3}^{N} \bar\rho^{(\gamma-1)(j +\frac{1}{2})+1} ( x \abs{ \partial_x^j \eta } +  \abs{ \partial_x^{j - 1} \eta } ) \\
        &+\mathcal{O}(1) \epsilon \bar\rho^{(\gamma-1)\frac{N+1}{2}+1}(x\abs{\partial_x^2\eta}+\abs{\partial_x\eta}).
			\end{aligned}
		\end{equation}
    where $G_N$ is given in \eqref{def-GN}.
        \end{enumerate}
\end{lemma}

\begin{proof}
	Note that	
	\begin{equation}
		\label{eq-Qfrak-1}
		\begin{aligned}
			\mathfrak{Q}_x &=-\gamma(1+\eta)^{-2\gamma}(1+\eta+x\eta_x)^{-\gamma} \left(\frac{2\eta_x}{1+\eta}+\frac{2\eta_x+x\eta_{xx}}{1+\eta+x\eta_x}\right)\\ &
			=-\gamma (1+\eta)^{-2\gamma}(1+\eta+x\eta_x)^{-\gamma-1}\left[x\eta_{xx}+4\eta_x+2(1+\eta)^{-1} x \eta_x^2\right].
		\end{aligned}   
	\end{equation}
	So for $s=1$, \eqref{ineq-Px1} is verified with \eqref{ineq-eta}.
	
	For $2\leq s\leq N_0+1$, it follows from \lemref{lem-Jacobx} that 
	\begin{align*}
    \partial_x^s \Qfrak
        = &-\gamma \partial_x^{s-1}[ (1+\eta)^{-2\gamma} (1+\eta+x\eta_x)^{-\gamma-1}x\eta_{xx}]\\
        & -4\gamma \partial_x^{s-1}[ (1+\eta)^{-2\gamma} (1+\eta+x\eta_x)^{-\gamma-1}\eta_x]\\
		 & -2\gamma \partial_x^{s-1}  \big[(1+\eta)^{-2\gamma-1} (1+\eta+x\eta_x)^{-\gamma-1} x\eta_x^2\big] \\
		= & -\gamma (1+\eta)^{-2\gamma} (1+\eta+x\eta_x)^{-\gamma-1} \big[x\partial_x^{s+1}\eta+(s+3)\partial_x^s\eta \big]\\
        &+\mathcal{O}(1) \epsilon \Big[\sum_{j=3}^{s} \bar\rho^{(\gamma-1)(j-s-\frac{1}{2})} (x \abs{ \partial_x^j \eta } + \abs{ \partial_x^{j-1} \eta } )+\bar\rho^{(\gamma-1)\frac{-s+1}{2}}(x\abs{\partial_x^2\eta}+\abs{\partial_x\eta}) \Big].
	\end{align*}
	So \eqref{ineq-Px1} is proved.
	
	By direct calculations, one can get
	\begin{align*}
		&\partial_x \Big\{ \bigl[\rhobar^{(\gamma-1)(N+1)+1}   (x \partial_x^{N}+N\partial_x^{N-1}) \Qfrak \bigr] (1 + \eta)^2 \Big\} \\
		=&  	(\rhobar^{(\gamma-1)(N+1)+1} )_x  \bigl[(x\partial_x^{N}+N\partial_x^{N-1}) \Qfrak \bigr] (1 + \eta)^2  \\
		& + \rhobar^{(\gamma-1)(N+1)+1}  \bigl[ \big( x\partial_x^{N+1}+(N+1)\partial_x^{N} \big) \Qfrak\bigr] (1+\eta)^2\\
		& +  2\bigl[\rhobar^{(\gamma-1)(N+1)+1}   (x\partial_x^{N}+N\partial_x^{N-1}) \Qfrak \bigr] (1 + \eta)\eta_x =: I_1+I_2+I_3, \\
	I_1=& -\gamma (\rhobar^{(\gamma-1)(N+1)+1} )_x   \Bigl\{(1+\eta)^{-2\gamma+2} (1+\eta+x\eta_x)^{-\gamma-1}   G_N\\
		&
		+ \mathcal{O}(1) \epsilon  \sum_{j=3}^{N} \bar\rho^{(\gamma-1)(j-N-\frac{1}{2})} ( x^2\abs{ \partial_x^j \eta } + x\abs{ \partial_x^{j - 1} \eta } ) \\
		& 
		 +\mathcal{O}(1) \epsilon  \sum_{j=3}^{N-1} \bar\rho^{(\gamma-1)(j-N+\frac{1}{2})} ( x\abs{ \partial_x^j \eta } + \abs{ \partial_x^{j - 1} \eta } ) \\
         &+\mathcal{O}(1) \epsilon \bar\rho^{(\gamma-1)\frac{-N+1}{2}}(x\abs{\partial_x^2\eta}+\abs{\partial_x\eta})\Big\}, \\
	I_2+I_3	= & -\gamma \rhobar^{(\gamma-1)(N+1)+1} \Bigl\{(1+\eta)^{-2\gamma+2} (1+\eta+x\eta_x)^{-\gamma-1} \partial_x (G_N)\\
	 &
	 +  \mathcal{O}(1) \epsilon  \sum_{j=3}^{N+1} \bar\rho^{(\gamma-1)(j-N-\frac{3}{2})} ( x^2\abs{ \partial_x^{j} \eta  }   + x\abs{ \partial_x^{j-1} \eta } )\\
	& 
	 +\mathcal{O}(1) \epsilon  \sum_{j=3}^{N} \bar\rho^{(\gamma-1)(j-N-\frac{1}{2})}  ( x\abs{ \partial_x^{j} \eta  }   + \abs{ \partial_x^{j-1} \eta } ) \\
     &+\mathcal{O}(1) \epsilon \bar\rho^{(\gamma-1)\frac{-N}{2}}(x\abs{\partial_x^2\eta}+\abs{\partial_x\eta}) \Big\}.
	\end{align*}
	This, together with $\eqref{eq-alpha1}_1$ yields,
	\begin{align*}
		I_1+I_2+I_3=  & -\gamma \partial_x\bigl\{ \rhobar^{(\gamma-1)(N+1)+1} (1+\eta)^{-2\gamma+2} (1+\eta+x\eta_x)^{-\gamma-1} G_N \big\}\\
	    &+ \mathcal{O}(1) \epsilon  \sum_{j=3}^{N+1} \bar\rho^{(\gamma-1)(j-\frac{1}{2})+1} ( x^2 \abs{ \partial_x^j \eta } + x \abs{ \partial_x^{j - 1} \eta } ) \\
	    & 
	    +\mathcal{O}(1) \epsilon  \sum_{j=3}^{N} \bar\rho^{(\gamma-1)(j+\frac{1}{2})+1} ( x \abs{ \partial_x^j \eta } +  \abs{ \partial_x^{j - 1} \eta } ) \\
        &+\mathcal{O}(1) \epsilon \bar\rho^{(\gamma-1)\frac{N+1}{2}+1}(x\abs{\partial_x^2\eta}+\abs{\partial_x\eta}),
	\end{align*}
which implies \eqref{ineq-Px2}.
\end{proof}

\begin{lemma} \label{lem-Pxt}
	Under the a priori assumption \eqref{def-apriori-assumt}, it holds that 
	\begin{enumerate}[label = \rmfamily(\roman*),ref = \rmfamily(\roman*)]
		\item For $1 \le s \le N_0+1$,
          \begin{align}
          	&\partial_x^s \Qfrak_\tau \nonumber \\
           =  &	-\gamma (1+\eta)^{-2\gamma} (1+\eta+x\eta_x)^{-\gamma-1} [x\partial_x^{s+1}\eta_\tau+(s+3)\partial_x^s\eta_\tau ]  \nonumber \\
         & + \gamma (\gamma+1) (1+\eta)^{-2\gamma} (1+\eta+x\eta_x)^{-\gamma-2} [x\partial_x^{s+1}\eta + (s+1) \partial_x^s \eta] (x\eta_{x\tau}+3\eta_\tau) \nonumber \\
         & + 2\gamma^2 (1+\eta)^{-2\gamma-1} (1+\eta+x\eta_x)^{-\gamma-1} \partial_x^s \eta \cdot (x\eta_{x\tau}+3\eta_\tau) \nonumber \\
         & + 2\gamma (\gamma+1) (1+\eta)^{-2\gamma-1} (1+\eta+x\eta_x)^{-\gamma-2} [x\partial_x^{s+1}\eta + (s+1) \partial_x^s \eta] x\eta_{x}\eta_\tau	\label{ineq-Pxt1}\\
         & + 2\gamma(2\gamma+1) (1+\eta)^{-2\gamma-2} (1+\eta+x\eta_x)^{-\gamma-1} \partial_x^s \eta \cdot x\eta_{x}\eta_\tau \nonumber \\
         &-2\gamma (1+\eta)^{-2\gamma-1} (1+\eta+x\eta_x)^{-\gamma-1}(x\partial_x^{s+1}\eta + s \partial_x^s\eta) \eta_\tau \nonumber \\
         &+ \mathcal{O} (1) \epsilon \sum_{j=3}^{s} \bar\rho^{(\gamma-1)(j-s-\frac{1}{2} )} ( x\abs{ \partial_x^{j} \eta_\tau } +  \abs{ \partial_x^{j-1} \eta_\tau }  +  x \abs{ \partial_x^{j} \eta} +  \abs{ \partial_x^{j-1} \eta}) \nonumber \\
          &+\mathcal{O} (1) \epsilon \bar\rho^{(\gamma-1)(\frac{-s+1}{2} )}(x\abs{\partial_x^{2} \eta_\tau} + \abs{ \partial_x \eta_\tau }+ x\abs{\partial_x^{2} \eta} + \abs{ \partial_x \eta } ). \nonumber
          \end{align}	
		
		\item For $1 \le N \le N_0$,
		\begin{equation}
			\label{ineq-Pxt2}
			\begin{aligned}
				&\partial_x \Big\{ \bigl[\rhobar^{(\gamma-1)(N+1)+1}   (x \partial_x^{N}+N \partial_x^{N-1}) \Qfrak_\tau \bigr] (1 + \eta)^2 \Big\} \\
			=	& \partial_x\Bigl\{ \rhobar^{(\gamma-1)(N+1)+1} \big[  -\gamma (1+\eta)^{-2\gamma+2} (1+\eta+x\eta_x)^{-\gamma-1} G_{N,\tau} +H_{N,1} \big] \Big\}\\
	    &+ \mathcal{O} (1) \epsilon \sum_{j=3}^{N+1} \bar\rho^{(\gamma-1)(j-\frac{1}{2} )+1} (  x^2\abs{ \partial_x^{j} \eta_\tau } +   x\abs{ \partial_x^{j-1} \eta_\tau }  +  x^2 \abs{ \partial_x^{j} \eta} +  x \abs{ \partial_x^{j-1} \eta}) \\
	    &+ \mathcal{O} (1) \epsilon \sum_{j=3}^{N} \bar\rho^{(\gamma-1)(j+\frac{1}{2} )+1} ( x\abs{ \partial_x^{j} \eta_\tau } +  \abs{ \partial_x^{j-1} \eta_\tau }  +  x \abs{ \partial_x^{j} \eta} +  \abs{ \partial_x^{j-1} \eta})  \\
        &+\mathcal{O}(1) \epsilon \bar\rho^{(\gamma-1)\frac{N+1}{2}+1}(x\abs{\partial_x^2\eta_\tau}+\abs{\partial_x\eta_\tau}+ x\abs{\partial_x^2\eta}+\abs{\partial_x\eta}),
			\end{aligned}
		\end{equation}	
    where $G_N$ is defined in \eqref{def-GN}, and $H_{N,1}=\mathcal{O}(1) \epsilon (\abs{x^2 \partial_x^{N+1}\eta} + \abs{ x \partial_x^N \eta} + \abs{\partial_x^{N-1}\eta } )$.
	\end{enumerate}
\end{lemma}

\begin{proof}
	Note that	
	\begin{equation}
		\label{eq-Qfrak-tau}
		\begin{aligned}
			\mathfrak{Q}_\tau &=-\gamma (1 + \eta)^{ - 2 \gamma} ( 1 + \eta + x \eta_x )^{- \gamma} \left( \frac{2 \eta_\tau}{1 + \eta} + \frac{ \eta_\tau + x\eta_{x\tau} }{1 + \eta + x \eta_x}\right)\\ &
			= -\gamma ( 1 + \eta )^{ - 2 \gamma} ( 1 + \eta + x \eta_x )^{ - \gamma - 1 } \left[ x \eta_{x \tau } + 3 \eta_\tau + 2 (1 + \eta )^{-1} x \eta_x \eta_\tau \right].
		\end{aligned}   
	\end{equation}
	So for $1 \leq s \leq N_0+1$, it follows from  \lemref{lem-Jacobt} that 
	\begin{align*}
		\partial_x^s \Qfrak_\tau   
		=& -\gamma \partial_x^{s} [(1+\eta)^{-2\gamma} (1+\eta+x\eta_x)^{-\gamma-1}(x\eta_{x\tau}+3\eta_\tau)]\\
		 &-2\gamma \partial_x^{s}  \big[(1+\eta)^{-2\gamma-1} (1+\eta+x\eta_x)^{-\gamma-1} x\eta_x \eta_\tau \big] \\
		 =  &	-\gamma (1+\eta)^{-2\gamma} (1+\eta+x\eta_x)^{-\gamma-1} [x\partial_x^{s+1}\eta_\tau+(s+3)\partial_x^s\eta_\tau ]   \\
         & + \gamma (\gamma+1) (1+\eta)^{-2\gamma} (1+\eta+x\eta_x)^{-\gamma-2} [x\partial_x^{s+1}\eta + (s+1) \partial_x^s \eta] (x\eta_{x\tau}+3\eta_\tau) \\
         & + 2\gamma^2 (1+\eta)^{-2\gamma-1} (1+\eta+x\eta_x)^{-\gamma-1} \partial_x^s \eta \cdot (x\eta_{x\tau}+3\eta_\tau) \\
         & + 2\gamma (\gamma+1) (1+\eta)^{-2\gamma-1} (1+\eta+x\eta_x)^{-\gamma-2} [x\partial_x^{s+1}\eta + (s+1) \partial_x^s \eta] x\eta_{x}\eta_\tau \\
         & + 2\gamma(2\gamma+1) (1+\eta)^{-2\gamma-2} (1+\eta+x\eta_x)^{-\gamma-1} \partial_x^s \eta \cdot x\eta_{x}\eta_\tau \\
         &-2\gamma (1+\eta)^{-2\gamma-1} (1+\eta+x\eta_x)^{-\gamma-1}(x\partial_x^{s+1}\eta + s \partial_x^s\eta) \eta_\tau\\
         &+ \mathcal{O} (1) \epsilon  \sum_{j=3}^{s} \bar\rho^{(\gamma-1)(j-s-\frac{1}{2} )} (  x\abs{ \partial_x^{j} \eta_\tau } +   \abs{ \partial_x^{j-1} \eta_\tau }  +  x \abs{ \partial_x^{j} \eta} +  \abs{ \partial_x^{j-1} \eta})\\
          &+\mathcal{O} (1) \epsilon \bar\rho^{(\gamma-1)(\frac{-s+1}{2} )}(x\abs{\partial_x^{2} \eta_\tau} + \abs{ \partial_x \eta_\tau }+ x\abs{\partial_x^{2} \eta} + \abs{ \partial_x \eta } ) ,
	\end{align*}
	So \eqref{ineq-Pxt1} is proved.
	
	By direct calculations, one obtains
	\begin{align*}
		&\partial_x \Big\{ \bigl[\rhobar^{(\gamma-1)(N+1)+1}   (x \partial_x^{N}+N\partial_x^{N-1}) \Qfrak_\tau \bigr] (1 + \eta)^2 \Big\} \\
		=& (\rhobar^{(\gamma-1)(N+1)+1} )_x   \bigl[(x\partial_x^{N}+N\partial_x^{N-1}) \Qfrak_\tau \bigr] (1 + \eta)^2  \\
		& + \rhobar^{(\gamma-1)(N+1)+1}  \bigl[ \big( x\partial_x^{N+1}+(N+1)\partial_x^{N} \big) \Qfrak_\tau \bigr] (1+\eta)^2\\
		& +  2\bigl[\rhobar^{(\gamma-1)(N+1)+1}   (x\partial_x^{N}+N\partial_x^{N-1}) \Qfrak_\tau \bigr] (1 + \eta)\eta_x =: I_1+I_2+I_3, \\ 
        \noalign{\vskip 8pt}  
	I_1=&  (\rhobar^{(\gamma-1)(N+1)+1} )_x   \bigl\{-\gamma (1+\eta)^{-2\gamma+2} (1+\eta+x\eta_x)^{-\gamma-1}   G_{N,\tau}\\
         & + \gamma (\gamma+1) (1+\eta)^{-2\gamma+2} (1+\eta+x\eta_x)^{-\gamma-2} \\
         &\quad \cdot [x^2\partial_x^{N+1}\eta + (2N+1) x \partial_x^N \eta + N^2 \partial_x^{N-1}\eta] (x\eta_{x\tau}+3\eta_\tau) \\
         & + 2\gamma^2 (1+\eta)^{-2\gamma+1} (1+\eta+x\eta_x)^{-\gamma-1} (x\partial_x^N \eta + N \partial_x^{N-1}\eta) (x\eta_{x\tau}+3\eta_\tau) \\
         & + 2\gamma (\gamma+1) (1+\eta)^{-2\gamma+1} (1+\eta+x\eta_x)^{-\gamma-2} \\
         &\quad \cdot [x^2\partial_x^{N+1}\eta + (2N+1) x \partial_x^N \eta + N^2 \partial_x^{N-1}\eta]  x\eta_{x}\eta_\tau \\
         & + 2\gamma(2\gamma+1) (1+\eta)^{-2\gamma} (1+\eta+x\eta_x)^{-\gamma-1} (x\partial_x^N \eta + N \partial_x^{N-1}\eta) x\eta_{x}\eta_\tau \\
         &-2\gamma (1+\eta)^{-2\gamma+1} (1+\eta+x\eta_x)^{-\gamma-1} \\
         & \quad \cdot [x^2\partial_x^{N+1}\eta + 2N x \partial_x^N \eta + N(N-1) \partial_x^{N-1}\eta] \eta_\tau\big\}\\
         &
		+ \mathcal{O}(1) \epsilon  \sum_{j=3}^{N} \bar\rho^{(\gamma-1)(j-N-\frac{1}{2})} ( x^2\abs{ \partial_x^j \eta_\tau } + x\abs{ \partial_x^{j - 1} \eta_\tau }+ x^2\abs{ \partial_x^j \eta } + x\abs{ \partial_x^{j - 1} \eta } ) \\
		& 
		 +\mathcal{O}(1) \epsilon  \sum_{j=3}^{N-1} \bar\rho^{(\gamma-1)(j-N+\frac{1}{2})} ( x\abs{ \partial_x^j \eta_\tau } + \abs{ \partial_x^{j - 1} \eta_\tau } + x\abs{ \partial_x^j \eta } + \abs{ \partial_x^{j - 1} \eta } ) \\
         &+\mathcal{O}(1) \epsilon \bar\rho^{(\gamma-1)\frac{-N+1}{2}}(x\abs{\partial_x^2\eta_\tau}+\abs{\partial_x\eta_\tau}+ x\abs{\partial_x^2\eta}+\abs{\partial_x\eta})\big\},\\
     \noalign{\vskip 8pt}  
	I_2+I_3	= & \rhobar^{(\gamma-1)(N+1)+1} \bigl\{  -\gamma(1+\eta)^{-2\gamma+2} (1+\eta+x\eta_x)^{-\gamma-1} \partial_x ( G_{N,\tau} )\\
	       & + \gamma (\gamma+1) (1+\eta)^{-2\gamma+2} (1+\eta+x\eta_x)^{-\gamma-2} \\
           & \quad \cdot [x^2\partial_x^{N+1}\eta + (2N+1) x \partial_x^N \eta + N^2 \partial_x^{N-1}\eta]_x (x\eta_{x\tau}+3\eta_\tau) \\
         & + 2\gamma^2 (1+\eta)^{-2\gamma+1} (1+\eta+x\eta_x)^{-\gamma-1} (x\partial_x^N \eta + N \partial_x^{N-1}\eta)_x (x\eta_{x\tau}+3\eta_\tau) \\
         & + 2\gamma (\gamma+1) (1+\eta)^{-2\gamma+1} (1+\eta+x\eta_x)^{-\gamma-2}\\
         & \quad \cdot [x^2\partial_x^{N+1}\eta + (2N+1) x \partial_x^N \eta + N^2 \partial_x^{N-1}\eta]_x x\eta_{x}\eta_\tau \\
         & + 2\gamma(2\gamma+1) (1+\eta)^{-2\gamma} (1+\eta+x\eta_x)^{-\gamma-1} (x\partial_x^N \eta + N \partial_x^{N-1}\eta)_x x\eta_{x}\eta_\tau \\
         &-2\gamma (1+\eta)^{-2\gamma+1} (1+\eta+x\eta_x)^{-\gamma-1} \\
         &\quad \cdot [x^2\partial_x^{N+1}\eta + 2N x \partial_x^N \eta + N(N-1) \partial_x^{N-1}\eta]_x \eta_\tau \big\}\\
	    &+ \mathcal{O} (1)\epsilon \sum_{j=3}^{N+1} \bar\rho^{(\gamma-1)(j-N-\frac{3}{2} )} (  x^2\abs{ \partial_x^{j} \eta_\tau } +   x\abs{ \partial_x^{j-1} \eta_\tau }+ x^2 \abs{ \partial_x^{j} \eta} +  x \abs{ \partial_x^{j-1} \eta}) \\
	    &+ \mathcal{O} (1) \epsilon \sum_{j=3}^{N} \bar\rho^{(\gamma-1)(j-N-\frac{1}{2} )} ( x\abs{ \partial_x^{j} \eta_\tau } +   \abs{ \partial_x^{j-1} \eta_\tau } +  x \abs{ \partial_x^{j} \eta} +  \abs{ \partial_x^{j-1} \eta}) \\
        &+ \mathcal{O} (1) \epsilon  \bar\rho^{(\gamma-1)\frac{-N}{2} } ( x\abs{ \partial_x^2 \eta_\tau } +   \abs{ \partial_x \eta_\tau } +  x \abs{ \partial_x^2 \eta} +  \abs{ \partial_x \eta}) \big\}.
	\end{align*}
    
	This, together with $\eqref{eq-alpha1}_1$ yields,
	\begin{align*}
&I_1+I_2+I_3\\
        =  &  \partial_x\Bigl\{  \rhobar^{(\gamma-1)(N+1)+1} \big[ -\gamma (1+\eta)^{-2\gamma+2} (1+\eta+x\eta_x)^{-\gamma-1} G_{N,\tau} \\
        & + \gamma (\gamma+1) (1+\eta)^{-2\gamma+2} (1+\eta+x\eta_x)^{-\gamma-2} \\
        & \quad \cdot [x^2\partial_x^{N+1}\eta + (2N+1) x \partial_x^N \eta + N^2 \partial_x^{N-1}\eta] (x\eta_{x\tau}+3\eta_\tau) \\
         & + 2\gamma^2 (1+\eta)^{-2\gamma+1} (1+\eta+x\eta_x)^{-\gamma-1} (x\partial_x^N \eta + N \partial_x^{N-1}\eta) (x\eta_{x\tau}+3\eta_\tau) \\
         & + 2\gamma (\gamma+1) (1+\eta)^{-2\gamma+1} (1+\eta+x\eta_x)^{-\gamma-2}\\
        & \quad \cdot  [x^2\partial_x^{N+1}\eta + (2N+1) x \partial_x^N \eta + N^2 \partial_x^{N-1}\eta]  x\eta_{x}\eta_\tau \\
         & + 2\gamma(2\gamma+1) (1+\eta)^{-2\gamma} (1+\eta+x\eta_x)^{-\gamma-1} (x\partial_x^N \eta + N \partial_x^{N-1}\eta) x\eta_{x}\eta_\tau \\
         &-2\gamma (1+\eta)^{-2\gamma+1} (1+\eta+x\eta_x)^{-\gamma-1}\\
        & \quad \cdot [x^2\partial_x^{N+1}\eta + 2N x \partial_x^N \eta + N(N-1) \partial_x^{N-1}\eta] \eta_\tau \big] \Big\}\\
	    &+ \mathcal{O} (1) \epsilon \sum_{j=3}^{N+1} \bar\rho^{(\gamma-1)(j-\frac{1}{2} )+1} (  x^2\abs{ \partial_x^{j} \eta_\tau } +   x\abs{ \partial_x^{j-1} \eta_\tau }  +  x^2 \abs{ \partial_x^{j} \eta} +  x \abs{ \partial_x^{j-1} \eta}) \\
	    &+ \mathcal{O} (1) \epsilon \sum_{j=3}^{N} \bar\rho^{(\gamma-1)(j+\frac{1}{2} )+1} ( x\abs{ \partial_x^{j} \eta_\tau } +  \abs{ \partial_x^{j-1} \eta_\tau }  +  x \abs{ \partial_x^{j} \eta} +  \abs{ \partial_x^{j-1} \eta})  \\
        &+\mathcal{O}(1) \epsilon \bar\rho^{(\gamma-1)\frac{N+1}{2}+1}(x\abs{\partial_x^2\eta_\tau}+\abs{\partial_x\eta_\tau}+ x\abs{\partial_x^2\eta}+\abs{\partial_x\eta}),
	\end{align*}
    which implies \eqref{ineq-Pxt2}.
\end{proof}

Now we derive the positive definiteness of $G_N^2$ defined in \eqref{def-GN}, which is precisely the purpose and the achievement to construct the non-trivial differential operator.

\begin{lemma}\label{lem-GN}
For $1 \leq N \le N_0$, $G_N$ and $G_{N,\tau}$ admit the following lower bounds, respectively: 
    \begin{equation}\label{ineq-GN1}
    \begin{aligned}
    &\int \bar\rho^{(\gamma-1)(N+1) + 1} G_N^2 \, dx \\
        \ge & C \int \bar\rho^{(\gamma-1)(N+1) +1} \Big\{x^4 (\partial_x^{N+1} \eta)^2 +  x^2 (\partial_x^{N}\eta)^2 \Big\}\,dx\\ &
        - C\int \bar\rho^{(\gamma-1)N +1} \Big\{x^4 (\partial_x^{N} \eta)^2 + x^2 (\partial_x^{N-1}\eta)^2 \Big\} \,dx     
    \end{aligned}
    \end{equation}
and
    \begin{equation}\label{ineq-GN2}
    \begin{aligned}
    &\int \bar\rho^{(\gamma-1)(N+1) + 1} G_{N,\tau}^2 \, dx \\
        \ge & C \int \bar\rho^{(\gamma-1)(N+1) +1} \Big\{x^4 (\partial_x^{N+1} \eta_\tau)^2 +  x^2 (\partial_x^{N}\eta_\tau)^2 \Big\}\,dx\\ &
        - C \int \bar\rho^{(\gamma-1)N +1} \Big\{x^4 (\partial_x^{N} \eta_\tau)^2 + x^2 (\partial_x^{N-1}\eta_\tau)^2 \Big\} \,dx.
    \end{aligned}
    \end{equation}
\end{lemma}

\begin{proof}
It suffices to prove \eqref{ineq-GN1}, and the other one is treated similarly. The proof is finished by direct calculations with integration by parts. Noting that $\eqref{eq-alpha1}_1$, we have
\begin{align*}
& \int \bar\rho^{(\gamma-1)(N+1) + 1} G_N^2 \, dx \\
= & \int \bar\rho^{(\gamma-1)(N+1) + 1} [x^4 (\partial_x^{N+1} \eta)^2 + (2N+3)^2 x^2 (\partial_x^N \eta)^2 + N^2(N+2)^2 (\partial_x^{N-1} \eta)^2  ] \, dx\\
    & + (2N+3) \int \bar\rho^{(\gamma-1)(N+1) + 1} x^3 \, d (\partial_x^N \eta)^2 \\
    & + 2N(N+2) \int \bar\rho^{(\gamma-1)(N+1) + 1} x^2 \partial_x^{N-1} \eta \, d\partial_x^{N} \eta \\ 
    &+ N(N+2)(2N+3)  \int \bar\rho^{(\gamma-1)(N+1) + 1} x \, d (\partial_x^{N-1} \eta)^2 \\
= & \int \bar\rho^{(\gamma-1)(N+1) + 1} [x^4 (\partial_x^{N+1} \eta)^2 + 2N(N+1) x^2 (\partial_x^N \eta)^2 + N^2(N+2)^2 (\partial_x^{N-1} \eta)^2  ] \, dx\\
    & +N(N+2)(2N+1) \int \bar\rho^{(\gamma-1)(N+1) + 1} x  \,d (\partial_x^{N-1} \eta)^2 \\
    & -(2N+3) \int (\bar\rho^{(\gamma-1)(N+1) + 1})_x x^3 (\partial_x^N \eta)^2 \, dx \\
    &- 2N(N+2) \int (\bar\rho^{(\gamma-1)(N+1) + 1})_x x^2 \partial_x^{N-1} \eta \, \partial_x^{N} \eta \, dx \\ 
= & \int \bar\rho^{(\gamma-1)(N+1) + 1} [x^4 (\partial_x^{N+1} \eta)^2 + 2N(N+1) x^2 (\partial_x^N \eta)^2 \\
& \qquad+ (N-1)N(N+1)(N+2) (\partial_x^{N-1} \eta)^2  ] \, dx\\    
    & -(2N+3) \int (\bar\rho^{(\gamma-1)(N+1) + 1})_x x^3 (\partial_x^N \eta)^2 \, dx \\
    & - 2N(N+2) \int (\bar\rho^{(\gamma-1)(N+1) + 1})_x x^2 \partial_x^{N-1} \eta \, \partial_x^{N} \eta \, dx \\ 
    & -N(N+2)(2N+1) \int (\bar\rho^{(\gamma-1)(N+1) + 1})_x x (\partial_x^{N-1} \eta)^2 \,dx  \\
\geq & C\int \bar\rho^{(\gamma-1)(N+1) +1} \Big\{x^4 (\partial_x^{N+1} \eta)^2 +  x^2 (\partial_x^{N}\eta)^2 \Big\}\,dx\\ &
        - C\int \bar\rho^{(\gamma-1)N +1} \Big\{x^4 (\partial_x^{N} \eta)^2 + x^2 (\partial_x^{N-1}\eta)^2 \Big\} \,dx  .
\end{align*}
Consequently, \eqref{ineq-GN1} is proved.
\end{proof}

\subsection{Proof of higher-order energy estimates}
With the above estimates of the lower-order terms at hand, we can now proceed to prove \propref{prop-high}. \\
\textbf{ Proof of Proposition \ref*{prop-high}. }
Rewrite equation \eqref{eq-Lag-high} via the notation $\mathfrak{Q}$ (see \eqref{def-Qfrak}):
\begin{equation} \label{eq-Lag-0}
\begin{aligned}
&  \alpha x \eta_{\tau\tau} +  \alpha_\tau x\eta_\tau-c_1\alpha^{3-3\gamma}\tilde{\alpha}(1+\eta) x\eta+  A\alpha^{3-3\gamma}\tilde{\alpha}\frac{(1+\eta)^2}{\bar\rho}(\bar\rho^\gamma\mathfrak{Q})_x \\
 = & -\frac{2\bar\mu+\bar\lambda}{\gamma} \alpha^{3-3\gamma}\frac{(1+\eta)^2}{\bar\rho}(\bar\rho^\gamma\mathfrak{Q_\tau})_x  
 - 4\bar\mu \alpha^{3-3\gamma} \dfrac{(1+\eta)\eta_\tau}{\bar\rho}[\bar\rho^\gamma ( \mathfrak{Q} + 1 ) ]_x . 
\end{aligned}
\end{equation}
For $1 \le N \le N_0,$ taking the operator $x\partial_x^{N} + (N+1) \partial_x^{N-1}$ to \eqref{eq-Lag-0}, one gets
\begin{equation} \label{eq-Lag-1}
\begin{aligned}
&\alpha [x\partial_x^{N} + (N+1) \partial_x^{N-1}] (x\eta_{\tau\tau})+\alpha_\tau [x\partial_x^{N} + (N+1) \partial_x^{N-1}] (x\eta_{\tau})\\ &
-c_1\alpha^{3-3\gamma}\tilde{\alpha} [x\partial_x^{N} + (N+1) \partial_x^{N-1}] \big(x\eta(1+\eta)\big)\\ &
+ A\alpha^{3-3\gamma}\tilde{\alpha} [x\partial_x^{N} + (N+1) \partial_x^{N-1}] \left\{ (\bar\rho^\gamma\mathfrak{Q})_x \frac{(1+\eta)^2}{\bar\rho}\right\}\\
 =& -\frac{2\bar\mu+\bar\lambda}{\gamma}\alpha^{3-3\gamma} [x\partial_x^{N} + (N+1) \partial_x^{N-1}]  \left\{ ( \bar\rho^\gamma \mathfrak{Q}_\tau )_x \frac{(1+\eta)^2}{\bar\rho} \right\}\\&
 - 4\bar\mu \alpha^{3-3\gamma} [x\partial_x^{N} + (N+1) \partial_x^{N-1}] \left\{ [\bar\rho^\gamma ( \mathfrak{Q} + 1 ) ]_x  \frac{(1+\eta)\eta_\tau}{\bar\rho }\right\}. 
\end{aligned}
\end{equation}

Multiply \eqref{eq-Lag-1} by $\alpha^{r_1}\bar\rho^{(\gamma-1)N+1} [x\partial_x^{N} + (N+1) \partial_x^{N-1}](x\eta_\tau)$ and integrate to obtain
\begin{align}
& \frac{d}{d \tau} \frac{1}{2}  \alpha^{1 + r_1} \int \bar\rho^{ (\gamma - 1)N + 1} G_{N-1,\tau} ^2 \, dx\notag\\ 
&+ \frac{1 - r_1}{2} \alpha^{r_1} \alpha_\tau \int \bar\rho^{ (\gamma - 1)N + 1 } G_{N-1,\tau} ^2 \,dx\notag\\
= & c_1 \alpha^{3 - 3\gamma + r_1} \tilde\alpha \int \bar\rho^{ (\gamma - 1) N + 1 } G_{N-1,\tau} \cdot [x\partial_x^{N} + (N+1) \partial_x^{N-1}] ( x \eta (1 + \eta) )\, dx\notag\\
&-  A \alpha^{3 - 3 \gamma + r_1} \tilde\alpha \int \bar\rho^{(\gamma - 1 ) N + 1} G_{N-1,\tau}  \cdot [x\partial_x^{N} + (N+1) \partial_x^{N-1}] \left\{ \left( \bar\rho^\gamma \Qfrak \right)_x \frac{(1 + \eta)^2}{\bar\rho} \right\} \, dx \notag\\
&-  \frac{2\bar\mu + \bar\lambda }{\gamma} \alpha^{3 - 3 \gamma + r_1} \int \bar\rho^{(\gamma - 1) N + 1} G_{N-1,\tau} \notag\\
&\qquad\qquad\qquad\qquad\qquad \cdot [x\partial_x^{N} + (N+1) \partial_x^{N-1}] \left\{  (\rhobar^\gamma \Qfrak_\tau)_x  \frac{(1 + \eta)^2}{\bar\rho} \right\} \, dx \notag\\ &
- 4 \bar\mu \alpha^{3 - 3 \gamma + r_1} \int \bar\rho^{(\gamma - 1) N + 1}G_{N-1,\tau} \notag\\
& \qquad\qquad\qquad\qquad\cdot [x\partial_x^{N} + (N+1) \partial_x^{N-1}] \left\{ [ \rhobar^\gamma (\Qfrak+1)]_x \frac{( 1 + \eta) \eta_\tau }{\bar\rho } \right \} \, dx  \notag\\
=: & L_1 + L_2 + L_3 + L_4. \label{eq-Lag-2}
\end{align}
Here $G_N$ is given in \eqref{def-GN}.

\textit{Step 1. Estimate of $L_2$.}
By virtue of \eqref{eq-alpha1}, \eqref{eq-Qfrak-1} and \eqref{ineq-Px2}, it holds that
\begin{align*}
	&(x\partial_x^{N-1} + N\partial_x^{N-2})\left\{\left(\bar\rho^\gamma \Qfrak \right)_x\frac{(1+\eta)^2}{\bar\rho} \right\} \\
	=&
	(x\partial_x^{N-1} + N\partial_x^{N-2})\left\{ \bigl[ \bar\rho^{\gamma-1} \Qfrak_x + \frac{\gamma}{\gamma-1} ( \rhobar^{\gamma-1} )_x \Qfrak] (1+\eta)^2 \right\} \\
	=& \Bigl[ \rhobar^{\gamma-1}  (x\partial_x^{N}+N\partial_x^{N-1}) \Qfrak
	+ ( \rhobar^{\gamma-1} )_x x (N+\frac{1}{\gamma-1}) \partial_x^{N-1} \Qfrak \Bigr] (1+\eta)^2
	\\
	&\quad + \Bigl [ C ( \rhobar^{\gamma-1} )_x  \partial_x^{N-2} \Qfrak 
	+ C ( \rhobar^{\gamma-1} )_{xx} \partial_x^{N-3} \Qfrak +  C ( \rhobar^{\gamma-1} )_{xx} x \partial_x^{N-2} \Qfrak \Bigr] (1+\eta)^2 \\
	&\quad +  \bigl[ x\partial_x^{N-1} + N \partial_x^{N-2} , (1+\eta)^2 \bigr] \bigl\{\bar\rho^{\gamma-1} \Qfrak_x + \frac{\gamma}{\gamma-1} ( \rhobar^{\gamma-1} )_x \Qfrak \bigr\}.
\end{align*}	
for $2 \le N \le N_0$, where we use the convention that for $i<0$,
$$\partial_x^{-i}f:=0$$ for simplicity.
This yields
\begin{equation} \label{tran-L2}
\begin{aligned} 
   &(x\partial_x^{N}+(N+1)\partial_x^{N-1})\left\{ \left(\bar\rho^\gamma \Qfrak \right)_x\frac{(1+\eta)^2}{\bar\rho} \right\} 
	\\
   =& \frac{1}{\rhobar^{(\gamma-1)N+1}} \partial_x \Big\{ \bigl[\rhobar^{(\gamma-1)(N+1)+1}   (x\partial_x^{N}+N\partial_x^{N-1}) \Qfrak \bigr] (1 + \eta)^2 \Big\}  \\
   &\quad
   - C (1+\eta)^2 ( \rhobar^{\gamma-1} )_x  \partial_x^{N-1} \Qfrak - 2 \rhobar^{\gamma-1} \eta_x (1+\eta) (x\partial_x^{N}+N\partial_x^{N-1}) \Qfrak  
   \\
   &\quad 
   + \partial_x \Big\{ C (1+\eta)^2 \Bigl[  ( \rhobar^{\gamma-1} )_x \partial_x^{N-2} \Qfrak + ( \rhobar^{\gamma-1} )_{xx}  \partial_x^{N-3} \Qfrak + ( \rhobar^{\gamma-1} )_{xx} x \partial_x^{N-2} \Qfrak \Bigr] 
   \\
   &\qquad \qquad + \bigl[ x\partial_x^{N-1} + N \partial_x^{N-2} , (1+\eta)^2 \bigr] \bigl\{ \bar\rho^{\gamma-1} \Qfrak_x + \frac{\gamma}{\gamma-1} ( \rhobar^{\gamma-1} )_x \Qfrak \bigr\} \Big\}  \\
	= &  \frac{1}{\rhobar^{(\gamma-1)N+1}} \partial_x \big\{ -\gamma \rhobar^{(\gamma-1)(N+1)+1} \big[ (1+\eta)^{-2\gamma+2} (1+\eta+x\eta_x)^{-\gamma-1}  G_N  \big] \big\} + H_{N,2},  
\end{aligned}
\end{equation}
for $2 \le N \le N_0$, where 
\begin{align*}
    H_{N,2}= & \mathcal{O}(1) \epsilon  \sum_{j=3}^{N+1} \bar\rho^{(\gamma-1)(j-N-\frac{1}{2})} ( x^2 \abs{ \partial_x^j \eta } + x \abs{ \partial_x^{j - 1} \eta } ) \\
	    & 
	    +\mathcal{O}(1) \epsilon  \sum_{j=3}^{N} \bar\rho^{(\gamma-1)(j-N+\frac{1}{2})} ( x \abs{ \partial_x^j \eta } +  \abs{ \partial_x^{j - 1} \eta } ) \\
        &+\mathcal{O}(1) \epsilon \bar\rho^{(\gamma-1)\frac{-N+1}{2}}(x\abs{\partial_x^2\eta}+\abs{\partial_x\eta})\\    
   &
   +  C (1+\eta)^2 ( \rhobar^{\gamma-1} )_x  \partial_x^{N-1} \Qfrak - 2 \rhobar^{\gamma-1} \eta_x (1+\eta) (x\partial_x^{N}+N\partial_x^{N-1}) \Qfrak
   \\
   &\
   + \partial_x \Big\{ C (1+\eta)^2 \Bigl[  ( \rhobar^{\gamma-1} )_x \partial_x^{N-2} \Qfrak + ( \rhobar^{\gamma-1} )_{xx}  \partial_x^{N-3} \Qfrak + ( \rhobar^{\gamma-1} )_{xx} x \partial_x^{N-2} \Qfrak \Bigr] 
   \\
   &\qquad  + \bigl[ x\partial_x^{N-1} + N \partial_x^{N-2} , (1+\eta)^2 \bigr] \bigl\{ \bar\rho^{\gamma-1} \Qfrak_x + \frac{\gamma}{\gamma-1} ( \rhobar^{\gamma-1} )_x \Qfrak \bigr\} \Big\} 
\end{align*}
is the lower order term. A direct computation shows that \eqref{tran-L2} also holds for $N=1$. As a result, integration by parts leads to
\begin{align*}
L_2 = & - A \gamma \alpha^{3 - 3 \gamma + r_1} \tilde\alpha \int   \rhobar^{(\gamma-1)(N+1)+1} (1+\eta)^{-2\gamma+2} (1+\eta+x\eta_x)^{-\gamma-1}  G_N G_{N,\tau} \, dx \\
    &- A \alpha^{3 - 3 \gamma + r_1} \tilde\alpha \int \bar\rho^{(\gamma - 1 ) N + 1} G_{N-1,\tau} \cdot  H_{N,2} \,dx \\
    & + C \alpha^{3 - 3 \gamma + r_1} \tilde\alpha \rhobar^{(\gamma-1)(N+1)+1} G_N G_{N-1,\tau}|^{x=1}_{x=0} ~~(\text{vanishes})\\
=:& L_{21}+L_{22}.  
\end{align*}
In the a priori estimates, $\eta$ is assumed to be smooth and even (see \rmkref{rmk-trace}) and thus the boundary term $\rhobar^{(\gamma-1)(N+1)+1} G_N G_{N-1,\tau}$ vanishes. In the subsequent proof, the boundary terms resulting from integration by parts vanish likewise.

Then under the a priori assumption \eqref{def-apriori-assum1}, it holds that
\begin{equation*} 
\begin{aligned}  
L_{21} \le  & - \frac{A \gamma}{2} \frac{d}{d \tau} \alpha^{3 - 3 \gamma + r_1} \tilde\alpha \int \rhobar^{(\gamma-1)(N+1)+1} (1+\eta)^{-2\gamma+2} (1+\eta+x\eta_x)^{-\gamma-1}  G_N^2 \, dx  \\
    & -\frac{A \gamma}{2} ( -\alpha^{3 - 3 \gamma + r_1} \tilde\alpha )_\tau \int \rhobar^{(\gamma-1)(N+1)+1} (1+\eta)^{-2\gamma+2} (1+\eta+x\eta_x)^{-\gamma-1}  G_N^2 \, dx  \\
    & + C \epsilon  \alpha^{3 - 3 \gamma + r_1-\min\{(3\gamma-3)/2, 1\}} \tilde\alpha \int \rhobar^{(\gamma-1)(N+1)+1} G_N^2 \, dx. 
\end{aligned}
\end{equation*}
Therefore,
\begin{equation}
\begin{aligned} \label{ineq-L21}
\int_0^\tau L_{21} \,d\tau' \le 
& - \frac{A \gamma}{4} \alpha^{3 - 3 \gamma + r_1} \tilde\alpha \int \rhobar^{(\gamma-1)(N+1)+1} G_N^2 \, dx  \\
    & -\frac{A \gamma}{4} \int_0^\tau ( -\alpha^{3 - 3 \gamma + r_1} \tilde\alpha )_\tau \int \rhobar^{(\gamma-1)(N+1)+1}   G_N^2 \, dx d\tau'\\
    &+ C \epsilon  \alpha^{3 - 3 \gamma + r_1-\min\{(3\gamma-3)/2, 1\}} \tilde\alpha \int \rhobar^{(\gamma-1)(N+1)+1} G_N^2 \, dx.   \\
    \le & - \frac{A \gamma}{4} \alpha^{3 - 3 \gamma + r_1} \tilde\alpha \int \rhobar^{(\gamma-1)(N+1)+1} G_N^2 \, dx  \\
    &  -\frac{A \gamma}{4}  ( 3\gamma- 4 - r_1 ) \int_0^\tau \alpha^{2-3\gamma+r_1}\alpha_\tau\tilde{\alpha} \int \rhobar^{(\gamma-1)(N+1)+1}   G_N^2 \, dx d\tau'\\
    &  + C(2\bar\mu+3\bar\lambda) \int_0^\tau \alpha^{-1} \mathcal{E}_N\,d\tau'+C \epsilon \int_0^\tau \alpha^{-\min\{(3\gamma-3)/2, 1\}} \mathcal{E}_N\,d\tau'.
\end{aligned}
\end{equation}

By the upper bound estimates on $G_{N-1,\tau}$, similar to (and easier than) the lower bound estimates \eqref{ineq-GN2}, we have 
\begin{equation*} 
\int_0^\tau L_{22} \,d\tau' \le C \int_0^\tau \alpha^{\frac{3-3\gamma}{2}} \mathcal{E}_N \,d\tau' + C\int_0^\tau \alpha^{\frac{9-9\gamma}{2}+r_1} \tilde\alpha \int \rhobar^{(\gamma-1)N+1} H_{N,2}^2\,dx d\tau' \, .
\end{equation*}

Next, we control the lower-order term $H_{N,2}$. Following a similar procedure as in \lemref{lem-Px}, we can obtain the estimate 
\begin{align*}
\abs{H_{N,2}}  
\le & C \big( x^2 \abs{\partial_x^{N}\eta} + x \abs{\partial_x^{N-1}\eta}  + \abs{\partial_x^{N-2}\eta} + \abs{\partial_x^{N-3}\eta} \big)\\
& +C \epsilon  \sum_{j=3}^{N+1} \bar\rho^{(\gamma-1)(j-N-\frac{1}{2})} ( x^2 \abs{ \partial_x^j \eta } + x \abs{ \partial_x^{j - 1} \eta } ) \\
	    & 
	    +C \epsilon  \sum_{j=3}^{N} \bar\rho^{(\gamma-1)(j-N+\frac{1}{2})} ( x \abs{ \partial_x^j \eta } +  \abs{ \partial_x^{j - 1} \eta } ) \\
        &+C \epsilon \bar\rho^{(\gamma-1)\frac{-N+1}{2}}(x\abs{\partial_x^2\eta}+\abs{\partial_x\eta}) .
\end{align*}
Then it follows from the Hardy inequality that 
\begin{equation} \label{ineq-L22}
\int_0^\tau L_{22}\,d\tau' \le  C \int_0^\tau \alpha^{\frac{3-3\gamma}{2}} \mathcal{E}_N \,d\tau'.
\end{equation}

\textit{Step 2. Estimate of $L_3$.}
Rewrite the higher-order derivatives of $\mathfrak{Q}_\tau$ as: 
\begin{align}
   & (x\partial_x^{N}+(N+1)\partial_x^{N-1})\left\{ \left(\bar\rho^\gamma \Qfrak_\tau \right)_x\frac{(1+\eta)^2}{\bar\rho} \right\} \nonumber
	\\
   =&  \frac{1}{\rhobar^{(\gamma-1)N+1}} \partial_x \Big\{ \bigl[\rhobar^{(\gamma-1)(N+1)+1}   (x\partial_x^{N}+N\partial_x^{N-1}) \Qfrak_\tau \bigr] (1 + \eta)^2 \Big\} \nonumber \\
   &
   - C (1+\eta)^2 ( \rhobar^{\gamma-1} )_x  \partial_x^{N-1} \Qfrak_\tau -2\rhobar^{\gamma-1} \eta_x (1+\eta) (x\partial_x^{N}+N\partial_x^{N-1}) \Qfrak_\tau \nonumber
   \\
   & 
   + \partial_x \Big\{ C (1+\eta)^2 \Bigl[  ( \rhobar^{\gamma-1} )_x \partial_x^{N-2} \Qfrak_\tau + ( \rhobar^{\gamma-1} )_{xx}  \partial_x^{N-3} \Qfrak_\tau + ( \rhobar^{\gamma-1} )_{xx} x \partial_x^{N-2} \Qfrak_\tau \Bigr]  \label{tran-L3}
   \\
   &\qquad + \bigl[ x\partial_x^{N-1} + N \partial_x^{N-2} , (1+\eta)^2 \bigr] \bigl\{ \bar\rho^{\gamma-1} \Qfrak_{x\tau} + \frac{\gamma}{\gamma-1} ( \rhobar^{\gamma-1} )_x \Qfrak_\tau \bigr\} \Big\} \nonumber \\
	= & \frac{1}{\rhobar^{(\gamma-1)N+1}}  \partial_x\Bigl\{ \rhobar^{(\gamma-1)(N+1)+1} \big[ -\gamma (1+\eta)^{-2\gamma+2} (1+\eta+x\eta_x)^{-\gamma-1} G_{N,\tau} + H_{N,1} \big] \Big\} \nonumber \\&
    + H_{N,3}, \nonumber
\end{align}
for $1 \le N \le N_0$, where $H_{N,1}$ is defined in \lemref{lem-Pxt} and 
\begin{align*}
&H_{N,3}\\
=& \mathcal{O} (1) \epsilon \sum_{j=3}^{N+1} \bar\rho^{(\gamma-1)(j-N-\frac{1}{2} )} (  x^2\abs{ \partial_x^{j} \eta_\tau } +   x\abs{ \partial_x^{j-1} \eta_\tau }  +  x^2 \abs{ \partial_x^{j} \eta} +  x \abs{ \partial_x^{j-1} \eta}) \\
	    &+ \mathcal{O} (1) \epsilon \sum_{j=3}^{N} \bar\rho^{(\gamma-1)(j-N+\frac{1}{2} )} ( x\abs{ \partial_x^{j} \eta_\tau } +  \abs{ \partial_x^{j-1} \eta_\tau }  +  x \abs{ \partial_x^{j} \eta} +  \abs{ \partial_x^{j-1} \eta})  \\
        &+\mathcal{O}(1) \epsilon \bar\rho^{(\gamma-1)\frac{-N+1}{2}}(x\abs{\partial_x^2\eta_\tau}+\abs{\partial_x\eta_\tau}+ x\abs{\partial_x^2\eta}+\abs{\partial_x\eta})\\
        & + C (1+\eta)^2 ( \rhobar^{\gamma-1} )_x  \partial_x^{N-1} \Qfrak_\tau -2\rhobar^{\gamma-1} \eta_x (1+\eta) (x\partial_x^{N}+N\partial_x^{N-1}) \Qfrak_\tau
   \\
   & 
   + \partial_x \Big\{ C(1+\eta)^2 \Bigl[  ( \rhobar^{\gamma-1} )_x \partial_x^{N-2} \Qfrak_\tau + ( \rhobar^{\gamma-1} )_{xx}  \partial_x^{N-3} \Qfrak_\tau + ( \rhobar^{\gamma-1} )_{xx} x \partial_x^{N-2} \Qfrak_\tau \Bigr] 
   \\
   &\qquad + \bigl[ x\partial_x^{N-1} + N \partial_x^{N-2} , (1+\eta)^2 \bigr] \bigl\{ \bar\rho^{\gamma-1} \Qfrak_{x\tau} + \frac{\gamma}{\gamma-1} ( \rhobar^{\gamma-1} )_x \Qfrak_\tau \bigr\} \Big\} 
\end{align*}
are lower-order terms. Consequently, it follows from integration by parts that
\begin{align*}
L_3  = & -(2\bar\mu+\bar\lambda) \alpha^{3 - 3 \gamma + r_1} \int  \rhobar^{(\gamma-1)  (N+1)+1} (1+\eta)^{-2\gamma+2} (1+\eta+x\eta_x)^{-\gamma-1}  G_{N,\tau}^2\,dx\\ 
        & + C(2\bar\mu+\bar\lambda) \alpha^{3 - 3 \gamma + r_1}  \int   \rhobar^{(\gamma-1)(N+1)+1} G_{N,\tau} H_{N,1} \,dx\\
        &- C(2\bar\mu+\bar\lambda)\alpha^{3 - 3 \gamma + r_1} \int \bar\rho^{(\gamma - 1 ) N + 1} G_{N-1,\tau} H_{N,3} \,dx \\
        & + C(2\bar\mu+\bar\lambda)\alpha^{3 - 3 \gamma + r_1} \rhobar^{(\gamma-1)  (N+1)+1} G_{N,\tau} G_{N-1,\tau}|^{x=1}_{x=0} ~~(\text{vanishes})\\
        &- C(2\bar\mu+\bar\lambda) \alpha^{3 - 3 \gamma + r_1} \rhobar^{(\gamma-1)  (N+1)+1} H_{N,1} G_{N-1,\tau}|^{x=1}_{x=0} ~~(\text{vanishes}) \\
    =: & L_{31} + L_{32} + L_{33},
\end{align*}
where
\begin{align} 
&\begin{aligned} \label{ineq-L31}
\int_0^\tau L_{31}  \,d\tau' \le -\frac{2\bar\mu+\bar\lambda}{2} \int_0^\tau \alpha^{3 - 3 \gamma + r_1} \int  \rhobar^{(\gamma-1)(N+1)+1} G_{N,\tau}^2\,dx \, d\tau',
\end{aligned}
\\
&\begin{aligned} \label{ineq-L32}
\int_0^\tau L_{32} \,d\tau' 
\le & \omega(2\bar\mu+\bar\lambda) \int_0^\tau \alpha^{3 - 3 \gamma + r_1} \int  \rhobar^{(\gamma-1)(N+1)+1} G_{N,\tau}^2\,dx \, d\tau'\\
& + C_{\omega}\epsilon(2\bar\mu+\bar\lambda) \int_0^\tau  \alpha^{3 - 3 \gamma + r_1}  \int   \rhobar^{(\gamma-1)(N+1)+1} \\
& \qquad\cdot \big[ x^4 (\partial_x^{N+1}\eta)^2 + x^2 (\partial_x^N\eta)^2 +(\partial_x^{N-1}\eta)^2\big]\,dx\,d\tau'\\
\le & \omega(2\bar\mu+\bar\lambda) \mathcal{D}_N+C_{\omega} \epsilon (2\bar\mu+\bar\lambda) \int_0^\tau \alpha^{-1} \mathcal{E}_N \,d\tau'
\end{aligned}
\intertext{where $\omega$ is a small constant to be determined later and}
&\begin{aligned} \nonumber
\int_0^\tau L_{33} \,d\tau' \le 
& C(2\bar\mu+\bar\lambda) \int_0^\tau \alpha^{3 - 3 \gamma + r_1} \int \bar\rho^{(\gamma - 1 ) N + 1} G_{N-1,\tau}^2\,dx \, d\tau' \\
    & + C(2\bar\mu+\bar\lambda)  \int_0^\tau \alpha^{3 - 3 \gamma + r_1}\int \bar\rho^{(\gamma - 1 ) N + 1} H_{N,3}^2 \,dx \, d\tau' \\
\le & C(2\bar\mu+\bar\lambda) \Big(\int_0^\tau \alpha^{-1} \mathcal{E}_{N}\,d\tau' +  \int_0^\tau \alpha^{3 - 3 \gamma + r_1}  \int \bar\rho^{(\gamma - 1 ) N + 1} H_{N,3}^2 \,dx \, d\tau' \Big).
\end{aligned}
\end{align}
Through a similar computation, we obtain the following control of $H_{N,3}$
\begin{align*}
\abs{H_{N,3}} \le & C \big( x^2 \abs{\partial_x^{N}\eta_\tau} + x \abs{\partial_x^{N-1}\eta_\tau}  + \abs{\partial_x^{N-2}\eta_\tau} + \abs{\partial_x^{N-3}\eta_\tau} \big)\\
&+C \epsilon \sum_{j=3}^{N+1} \bar\rho^{(\gamma-1)(j-N-\frac{1}{2} )} (  x^2\abs{ \partial_x^{j} \eta_\tau } +   x\abs{ \partial_x^{j-1} \eta_\tau }  +  x^2 \abs{ \partial_x^{j} \eta} +  x \abs{ \partial_x^{j-1} \eta}) \\
	    &+C\epsilon \sum_{j=3}^{N} \bar\rho^{(\gamma-1)(j-N+\frac{1}{2} )} ( x\abs{ \partial_x^{j} \eta_\tau } +  \abs{ \partial_x^{j-1} \eta_\tau }  +  x \abs{ \partial_x^{j} \eta} +  \abs{ \partial_x^{j-1} \eta})  \\
        &+C \epsilon \bar\rho^{(\gamma-1)\frac{-N+1}{2}}(x\abs{\partial_x^2\eta_\tau}+\abs{\partial_x\eta_\tau}+ x\abs{\partial_x^2\eta}+\abs{\partial_x\eta}).
\end{align*}
Then the Hardy inequality yields
\begin{equation} \label{ineq-L33}
\int_0^\tau L_{33} \,d\tau' \le  C(2\bar\mu+\bar\lambda)\left( \int_0^\tau \alpha^{-1} \mathcal{E}_{N}\,d\tau' +   \epsilon \mathcal{D}_{N}  \right).
\end{equation} 

\textit{Step 3. Estimate of $L_4$.}  
A procedure analogous to the estimate of $L_2$ yields that for $2 \le N \le N_0$
\begin{equation} \label{tran-L4}
\begin{aligned}
	&(x\partial_x^{N-1} + N\partial_x^{N-2}) \left\{ [ \rhobar^\gamma (\Qfrak+1)]_x \frac{( 1 + \eta) \eta_\tau }{\bar\rho } \right\} \\
	= &
	(x\partial_x^{N-1} + N\partial_x^{N-2})\left\{ \bigl[ \bar\rho^{\gamma-1} \Qfrak_x + \frac{\gamma}{\gamma-1} ( \rhobar^{\gamma-1} )_x (\Qfrak+1) ] (1+\eta) \eta_\tau \right\} \\
	= & \Bigl[ \rhobar^{\gamma-1}  (x\partial_x^{N}+N\partial_x^{N-1}) \Qfrak
	+ ( \rhobar^{\gamma-1} )_x x (N+\frac{1}{\gamma-1}) \partial_x^{N-1} \Qfrak \Bigr] (1+\eta)\eta_\tau
	\\
	& + \Bigl [ C ( \rhobar^{\gamma-1} )_x  \partial_x^{N-2} \Qfrak 
	+ C ( \rhobar^{\gamma-1} )_{xx} \partial_x^{N-3} \Qfrak +  C ( \rhobar^{\gamma-1} )_{xx} x \partial_x^{N-2} \Qfrak \Bigr] (1+\eta)\eta_\tau \\
	& +  \bigl[ x\partial_x^{N-1} + N \partial_x^{N-2} , (1+\eta)\eta_\tau \bigr] \bigl\{\bar\rho^{\gamma-1} \Qfrak_x + \frac{\gamma}{\gamma-1} ( \rhobar^{\gamma-1} )_x \Qfrak \bigr\}\\
    &\quad + \frac{\gamma}{\gamma-1} (x\partial_x^{N-1} + N\partial_x^{N-2}) \big[( \rhobar^{\gamma-1} )_x (1+\eta) \eta_\tau \big],
\end{aligned} 
\end{equation}
then for $1 \le N \le N_0$
\begin{align*}
   &(x\partial_x^{N}+(N+1)\partial_x^{N-1}) \left\{ [ \rhobar^\gamma (\Qfrak+1)]_x \frac{( 1 + \eta) \eta_\tau }{\bar\rho } \right\}
	\\
   =& \frac{1}{\rhobar^{(\gamma-1)N+1}} \partial_x \Big\{ \bigl[\rhobar^{(\gamma-1)(N+1)+1}   (x\partial_x^{N}+N\partial_x^{N-1}) \Qfrak \bigr] (1 + \eta)\eta_\tau \Big\}\\
   &
   -  C (1+\eta)\eta_\tau ( \rhobar^{\gamma-1} )_x  \partial_x^{N-1} \Qfrak -  \rhobar^{\gamma-1} (\eta_x \eta_\tau + ( 1 + \eta ) \eta_{x\tau} )(x\partial_x^{N}+N\partial_x^{N-1}) \Qfrak
   \\
   &
   + \partial_x \Big\{ C (1+\eta)\eta_\tau \Bigl[  ( \rhobar^{\gamma-1} )_x \partial_x^{N-2} \Qfrak + ( \rhobar^{\gamma-1} )_{xx}  \partial_x^{N-3} \Qfrak + ( \rhobar^{\gamma-1} )_{xx} x \partial_x^{N-2} \Qfrak \Bigr] 
   \\
   &\qquad  + \bigl[ x\partial_x^{N-1} + N \partial_x^{N-2} , (1+\eta)\eta_\tau \bigr] \bigl\{ \bar\rho^{\gamma-1} \Qfrak_x + \frac{\gamma}{\gamma-1} ( \rhobar^{\gamma-1} )_x \Qfrak \bigr\}  \\
   &\qquad    + \frac{\gamma}{\gamma-1} (x\partial_x^{N-1} + N\partial_x^{N-2}) \big[( \rhobar^{\gamma-1} )_x (1+\eta) \eta_\tau \big] \Big\} \\
	= &  \frac{1}{\rhobar^{(\gamma-1)N+1}} \partial_x \big\{ -\gamma \rhobar^{(\gamma-1)(N+1)+1} \big[ (1+\eta)^{-2\gamma+1} (1+\eta+x\eta_x)^{-\gamma-1} \eta_\tau G_N  \big] \big\} + H_{N,4} .  
\end{align*}
A similar procedure to \lemref{lem-Px} implies that the lower-order term
\begin{align*}
H_{N,4} = & \mathcal{O}(1) \epsilon  \sum_{j=3}^{N+1} \bar\rho^{(\gamma-1)(j-N-\frac{1}{2})} ( x^2 \abs{ \partial_x^j \eta } + x \abs{ \partial_x^{j - 1} \eta } ) \\
	    & 
	    +\mathcal{O}(1) \epsilon  \sum_{j=3}^{N} \bar\rho^{(\gamma-1)(j-N+\frac{1}{2})} ( x \abs{ \partial_x^j \eta } +  \abs{ \partial_x^{j - 1} \eta } ) \\
        &+\mathcal{O}(1) \epsilon \bar\rho^{(\gamma-1)\frac{-N+1}{2}}(x\abs{\partial_x^2\eta}+\abs{\partial_x\eta})\\ 
        &   + \partial_x \Big\{ C (1+\eta)\eta_\tau \Bigl[  ( \rhobar^{\gamma-1} )_x \partial_x^{N-2} \Qfrak + ( \rhobar^{\gamma-1} )_{xx}  \partial_x^{N-3} \Qfrak + ( \rhobar^{\gamma-1} )_{xx} x \partial_x^{N-2} \Qfrak \Bigr] 
   \\
   &\qquad  + \bigl[ x\partial_x^{N-1} + N \partial_x^{N-2} , (1+\eta)\eta_\tau \bigr] \bigl\{ \bar\rho^{\gamma-1} \Qfrak_x + \frac{\gamma}{\gamma-1} ( \rhobar^{\gamma-1} )_x \Qfrak \bigr\}  \\
   &\qquad    + \frac{\gamma}{\gamma-1} (x\partial_x^{N-1} + N\partial_x^{N-2}) \big[( \rhobar^{\gamma-1} )_x (1+\eta) \eta_\tau \big] \Big\} .
\end{align*}
Therefore, integration by parts shows that
\begin{align*}
L_4 = & -4\bar\mu \gamma  \alpha^{3 - 3 \gamma + r_1} \int \rhobar^{(\gamma-1)(N+1)+1} (1+\eta)^{-2\gamma+1} (1+\eta+x\eta_x)^{-\gamma-1} \eta_\tau G_N G_{N,\tau} \,dx \\
 & - 4\bar\mu  \alpha^{3 - 3 \gamma + r_1} \int \rhobar^{(\gamma-1)N+1} G_{N-1,\tau} H_{N,4} \,dx\\
 & + C \bar\mu \alpha^{3 - 3 \gamma + r_1}  \rhobar^{(\gamma-1)(N+1)+1} \eta_\tau G_N G_{N-1,\tau} |^{x=1}_{x=0} ~~(\text{vanishes})\\
  \le & \omega(2\bar\mu+\bar\lambda) \alpha^{3 - 3 \gamma + r_1}  \int \rhobar^{(\gamma-1)(N+1)+1} G_{N,\tau}^2 \,dx \\
  & + C_{\omega}\epsilon \frac{\bar\mu^2}{2\bar\mu+\bar\lambda} \alpha^{3 - 3 \gamma + r_1}  \int \rhobar^{(\gamma-1)(N+1)+1} G_{N}^2 \,dx\\
  & + C\bar\mu \alpha^{3 - 3 \gamma + r_1} \int \rhobar^{(\gamma-1)N+1} (G_{N-1,\tau}^2 + H_{N,4}^2) \,dx.
\end{align*}

Similarly, we have the following control of $H_{N,4}$
\begin{align*}
\abs{H_{N,4}} \le 
        &C \big( x^2 \abs{\partial_x^{N}\eta_\tau} + x \abs{\partial_x^{N-1}\eta_\tau}  + \abs{\partial_x^{N-2}\eta_\tau} \big)\\
&+C \epsilon \sum_{j=3}^{N+1} \bar\rho^{(\gamma-1)(j-N-\frac{1}{2} )} (  x^2\abs{ \partial_x^{j} \eta_\tau } +   x\abs{ \partial_x^{j-1} \eta_\tau }  +  x^2 \abs{ \partial_x^{j} \eta} +  x \abs{ \partial_x^{j-1} \eta}) \\
	    &+C\epsilon \sum_{j=3}^{N} \bar\rho^{(\gamma-1)(j-N+\frac{1}{2} )} ( x\abs{ \partial_x^{j} \eta_\tau } +  \abs{ \partial_x^{j-1} \eta_\tau }  +  x \abs{ \partial_x^{j} \eta} +  \abs{ \partial_x^{j-1} \eta})  \\
        &+C \epsilon \bar\rho^{(\gamma-1)\frac{-N+1}{2}}(x\abs{\partial_x^2\eta_\tau}+\abs{\partial_x\eta_\tau}+ x\abs{\partial_x^2\eta}+\abs{\partial_x\eta}).
\end{align*}

It follows from the Hardy inequality that
\begin{equation}
\label{ineq-L4}
\int_0^\tau L_4 \,d\tau' \le  \big[\omega(2\bar\mu + \bar\lambda) +C\epsilon\bar\mu\big] \mathcal{D}_N + C_{\omega} \Big (\frac{\epsilon \bar\mu^2}{2\bar\mu + \bar\lambda} + \bar\mu \Big) \int_0^\tau \alpha^{-1} \mathcal{E}_{N}\,d\tau'. 
\end{equation} 

\textit{Step 4. Estimate of $L_1$.}   
\begin{equation} \label{ineq-L1-imp}
\begin{aligned}
&\int_0^\tau L_1 \,d\tau'  \\
\le & C \int_0^\tau \alpha^{\frac{9-9\gamma}{2}+r_1} \tilde\alpha \int\bar\rho^{(\gamma-1)N+1}\left\{[x\partial_x^{N}+(N+1)\partial_x^{N-1}] \big(x\eta(1+\eta)\big)\right\}^2\,dx\,d\tau' \\
&  + C \int_0^\tau \alpha^{\frac{3-3\gamma}{2}} \mathcal{E}_N \,d\tau',
\end{aligned}
\end{equation}
where
\begin{equation} \label{tran-L1}
\begin{aligned}
&\abs{[x\partial_x^{N}+(N+1)\partial_x^{N-1}] (x\eta(1+\eta))}\\
=&\mathcal{O}(1) \sum_{k=0}^2  x^{2-k} [ \abs{\partial_x^{N-k}\eta} + \sum_{j=1}^{N-k} \abs{\partial_x^{N-k-j}\eta}\abs{\partial_x^j\eta} ]\\
=&\mathcal{O}(1) \sum_{k=0}^2 x^{2-k} \Big[ \abs{\partial_x^{N-k}\eta} + \epsilon \abs{\partial_x^{N-k-1}\eta}   + \epsilon \bar\rho^{-1/2(\gamma-1)}\abs{\partial_x^{N-k-2}\eta}  \\
&\qquad\qquad\qquad\quad+ \epsilon  \sum_{j=3}^{N-k} \bar\rho^{(\gamma-1) ( -N+k+j+\frac{3}{2} )} \abs{\partial_x^{j}\eta} \Big].
\end{aligned}
\end{equation}

Then the Hardy inequality shows that
\begin{equation}\label{ineq-L1}
\begin{aligned}
\int_0^\tau L_1\,d\tau' \le &  C \int_0^\tau \alpha^{\frac{3-3\gamma}{2}} \mathcal{E}_N \,d\tau' +C_{\omega} \int_0^\tau \alpha^{\frac{9-9\gamma}{2}+r_1} \tilde\alpha \sum_{j=0}^{N} \int \bar\rho^{(\gamma-1)N+1} x^4 (\partial_x^{N-j}\eta)^2\,dxd\tau'
\\
\le & 
C \int_0^\tau \alpha^{\frac{3-3\gamma}{2}} \mathcal{E}_N \,d\tau'.
\end{aligned}
\end{equation}

Combining \eqref{eq-Lag-2}, \eqref{ineq-L21}, \eqref{ineq-L22}, \eqref{ineq-L31}, \eqref{ineq-L32}, \eqref{ineq-L33}, \eqref{ineq-L4} and \eqref{ineq-L1},  for $r_1 \le \min \{ 3\gamma-4, 1 \}$, and taking $\omega$ sufficiently small, we can derive \eqref{ineq-energy1}. 
\hfill$\Box$

\section{Proof of main theorems} 
This section is devoted to proving \autoref{thm-stab} and \autoref{thm-vanish}. 

\subsection{The existence} 
Observing the rapid growth of $\alpha$ \eqref{ineq-alpha1}, a combination of \eqref{ineq-energy1} with an iterative process and Gr\"onwall's inequality yields that under the a priori assumption \eqref{def-apriori-assumx}, \eqref{def-apriori-assumt} and \eqref{def-apriori-assum1} with sufficiently small $\epsilon$,  
\begin{equation} \label{ineq-total}
     \mathcal{E}_{N_0}+ (1-r_1) \mathcal{E}_{N_0}^{(1)} + (4-3\gamma+r_1) \mathcal{E}_{N_0}^{(2)} +(2\bar\mu+\bar\lambda)\mathcal{D}_{N_0} \le C \mathcal{E}_{in},
\end{equation}
holds for any $r_1 \le \min \{ 3\gamma-4,1\}$.
Based on the aforementioned energy estimates \eqref{ineq-total}, we can derive more estimates of $\eta_{\tau}$ and $\eta_{\tau\tau}$.

\begin{lemma} We further have the following estimates of $\eta$ and $\eta_{\tau\tau}$:
\begin{enumerate} [label = \rmfamily(\roman*),ref = \rmfamily(\roman*)]
    \item For any fixed  $r_1< 1, r_1 \le 3\gamma-4$, any $1 \le N \le N_0-1$ and any time $T>0$,
\begin{equation} \label{ineq-etatt1}
\int_0^\tau \alpha^{1+r_1} \int \bar\rho^{(\gamma-1)N+1}x^4(\partial_x^{N}\eta_{\tau\tau})^2\,dx d\tau' \le C [1+ (2\bar\mu+\bar\lambda)^2]\mathcal{E}_{in}.
\end{equation}

\item For any fixed  $r_1 \le  \min\{3\gamma-4,1\}$  any $1 \le N \le N_0-2$ and any time $T>0$,
\begin{equation} \label{ineq-etatt2}
\sup_{0 \le \tau \le T} \alpha^{1+r_1} \int \bar\rho^{(\gamma-1)N+1}x^4(\partial_x^{N}\eta_{\tau\tau})^2\,dx \le C [1+ (2\bar\mu+\bar\lambda)^2]\mathcal{E}_{in}.
\end{equation}

\item For any time $T>0$, $0 \le N \le N_0-1$ and $r_2<0$,
\begin{equation} \label{ineq-eta1}
\begin{aligned}
\sup_{0 \le \tau \le T}\int\bar\rho^{(\gamma-1)(N+1)+1} x^4 (\partial_x^{N+1}\eta)^2\,dx  \le &C\mathcal{E}_{in},\\
\int _0^T \alpha^{2r_2}\int\bar\rho^{(\gamma-1)(N+1)+1} x^4 (\partial_x^{N+1}\eta)^2\,dx d\tau
    \le & C\mathcal{E}_{in}.
\end{aligned}
\end{equation}

\item For any time $T>0$,
\begin{equation} \label{ineq-H2}
\begin{alignedat}{3}
&\sup_{0 \le \tau \le T} \norm{\eta}{H^2}^2 &&\le C\mathcal{E}_{in};\\
&\sup_{0 \le \tau  < T } \alpha^{1+r_1} \norm{\eta_\tau}{H^2}^2  &&\le  C\mathcal{E}_{in}
&&\quad \text{for any } ~r_1 \le \min \{3\gamma-4,1 \};\\
&\norm{\alpha^{r_2}\eta}{L^2([0,T],H^2)}^2  && \le C\mathcal{E}_{in}
&& \quad\text{for any } ~r_2 < 0;\\
&\norm{\alpha^{\frac{1+r_1}{2}}\eta_\tau}{L^2([0,T],H^2)}^2  &&\le  C\mathcal{E}_{in}
&&\quad\text{for any } ~r_1 <1,~r_1 \le 3\gamma-4 ;\\
&\norm{ \alpha^{\frac{4-3\gamma+r_1}{2}}\eta}{L^2([0,T],H^2)}^2 &&\le  C\mathcal{E}_{in}   
&&\quad \text{for any } ~r_1 \le 1, ~r_1 < 3\gamma-4;\\
&\norm{\alpha^{\frac{1+r_1}{2}}\eta_{\tau\tau}}{L^2([0,T];H^1)}^2 &&\le   C_{r_1} [1+ (2\bar\mu+\bar\lambda)^2]\mathcal{E}_{in}
&& \quad\text{for any } ~r_1 < 1, ~r_1\le 3\gamma-4;\\
&\sup_{0 \le \tau \le T} \norm{\alpha^{\frac{1+r_1}{2}}\eta_{\tau\tau}}{H^1}^2 &&\le   C [1+ (2\bar\mu+\bar\lambda)^2]\mathcal{E}_{in}
&& \quad\text{for any } ~r_1 \le \min \{3\gamma-4,1 \}.
\end{alignedat}
\end{equation}
\end{enumerate}
\end{lemma}

\begin{proof}
It follows from  \eqref{eq-Lag-1} and \eqref{ineq-GN1} that $\eta_{\tau\tau}$ can be estimated by 
\begin{align*}
&\int_0^\tau \alpha^{1+r_1} \int \bar\rho^{(\gamma-1)N+1}x^4(\partial_x^{N}\eta_{\tau\tau})^2\,dx d\tau' \\
\le & C \int_0^\tau \alpha^{1+r_1} \int \bar\rho^{(\gamma-1)N+1} \Big[(x\partial_x^{N} + (N+1)\partial_x^{N-1})(x\eta_{\tau\tau}) \Big]^2\,dx d\tau' \\
\le &C \Big\{ \int_0^\tau \alpha^{-1+r_1} \alpha_\tau^2\int \bar\rho^{(\gamma-1)N+1} \Big[(x\partial_x^{N} + (N+1)\partial_x^{N-1})(x\eta_{\tau}) \Big]^2\,dx d\tau'  \\
&+ \int_0^\tau \alpha^{5-6\gamma+r_1} \tilde\alpha^2 \int \bar\rho^{(\gamma-1)N+1} \Big[(x\partial_x^{N} + (N+1)\partial_x^{N-1})\big(x\eta(1+\eta)\big) \Big]^2\,dx  d\tau' \\
&+\int_0^\tau \alpha^{5-6\gamma+r_1} \tilde\alpha^2 \int\bar\rho^{(\gamma-1)N+1} \Big\{(x\partial_x^{N} + (N+1)\partial_x^{N-1}) \Big[ (\bar\rho^\gamma\mathfrak{Q})_x \frac{(1+\eta)^2}{\bar\rho}\Big] \Big\}^2\,dx d\tau' \\
&+(2\bar\mu+\bar\lambda)^2 \int_0^\tau \alpha^{5-6\gamma+r_1} \int\bar\rho^{(\gamma-1)N+1} \\
&\qquad\qquad\qquad\quad \cdot \Big\{(x\partial_x^{N} + (N+1)\partial_x^{N-1}) \Big[ (\bar\rho^\gamma\mathfrak{Q}_\tau)_x \frac{(1+\eta)^2}{\bar\rho}\Big] \Big\}^2\,dx d\tau' \\
&+\bar\mu^2 \int_0^\tau \alpha^{5-6\gamma+r_1} \int\bar\rho^{(\gamma-1)N+1} \\
&\qquad\quad\quad\quad \cdot \Big\{(x\partial_x^{N} + (N+1)\partial_x^{N-1}) \Big[ [\bar\rho^\gamma(\mathfrak{Q}+1)]_x \frac{(1+\eta)\eta_\tau}{\bar\rho}\Big] \Big\}^2\,dx  d\tau' \Big\}.
\end{align*}
Then \eqref{tran-L2}, \eqref{tran-L3}, \eqref{tran-L4}, \eqref{tran-L1} and the Hardy inequality yield that
\begin{align*}
&\int_0^\tau \alpha^{1+r_1} \int \bar\rho^{(\gamma-1)N+1}x^4(\partial_x^{N}\eta_{\tau\tau})^2\,dx d\tau'\\
\le & C  \Big\{ \int_0^\tau \alpha^{1+r_1} \int \bar\rho^{(\gamma-1)N +1} G_{N-1,\tau}^2\,dx d\tau' \\
&+ \int_0^\tau \alpha^{5-6\gamma+r_1} \tilde\alpha^2 \int \bar\rho^{(\gamma-1)(N+2)+1} G_{N+1}^2 + \bar\rho^{(\gamma-1)N +1} (G_{N}^2 + H_{N,2}^2)\,dx d\tau'  \\
&+ (2\bar\mu+\bar\lambda)^2 \int_0^\tau \alpha^{5-6\gamma+r_1} \int \bar\rho^{(\gamma-1)(N+2)+1} \big( G_{N+1,\tau}^2 +(\partial_x H_{N,1})^2 \big)\,dx d\tau'\\
&+ (2\bar\mu+\bar\lambda)^2 \int_0^\tau \alpha^{5-6\gamma+r_1} \int \bar\rho^{(\gamma-1)N +1} (G_{N,\tau}^2 +H_{N,1}^2 + H_{N,3}^2) \,dx d\tau'\\
&+ \bar\mu^2 \epsilon\int_0^\tau \alpha^{5-6\gamma+r_1} \int \bar\rho^{(\gamma-1)(N+2)+1} G_{N+1}^2 + \bar\rho^{(\gamma-1)N +1} (G_{N}^2 + H_{N,4}^2) \,dx d\tau'\\
&+ \int_0^\tau \alpha^{5-6\gamma+r_1} \tilde\alpha^2 \int \bar\rho^{(\gamma-1)N+1} \Big[(x\partial_x^{N} + (N+1)\partial_x^{N-1})\big(x\eta(1+\eta)\big) \Big]^2\,dx  d\tau' \Big\}  \\
\le &  C \Big\{  \mathcal{E}_{N}^{(1)} + [1+ (2\bar\mu+\bar\lambda)^2] \int_0^\tau\alpha^{3-3\gamma} \mathcal{E}_{N+1} \,d\tau'   + (2\bar\mu+\bar\lambda)^2  \mathcal{D}_{N+1} \Big\} \\
\le & C_{r_1}  [1+ (2\bar\mu+\bar\lambda)^2] \mathcal{E}_{in}
\end{align*}
for any fixed $r_1< 1, r_1 \le 3\gamma-4$ and any $0 \le N \le N_0-1$. This proves \eqref{ineq-etatt1}.

Similarly, it holds that
\begin{align*}
& \alpha^{1+r_1} \int \bar\rho^{(\gamma-1)N+1}x^4(\partial_x^{N}\eta_{\tau\tau})^2\,dx\\
\le & C  \Big\{ \alpha^{1+r_1} \int \bar\rho^{(\gamma-1)N +1} G_{N-1,\tau}^2\,dx  \\
&+\alpha^{5-6\gamma+r_1} \tilde\alpha^2 \int \bar\rho^{(\gamma-1)(N+2)+1} G_{N+1}^2 + \bar\rho^{(\gamma-1)N +1} (G_{N}^2 + H_{N,2}^2)\,dx   \\
&+ (2\bar\mu+\bar\lambda)^2 \alpha^{5-6\gamma+r_1} \int \bar\rho^{(\gamma-1)(N+2)+1} \big( G_{N+1,\tau}^2 +(\partial_x H_{N,1})^2 \big)\,dx \\
&+ (2\bar\mu+\bar\lambda)^2  \alpha^{5-6\gamma+r_1} \int \bar\rho^{(\gamma-1)N +1} (G_{N,\tau}^2 + H_{N,1}^2 + H_{N,3}^2) \,dx \\
&+ \bar\mu^2 \epsilon\alpha^{5-6\gamma+r_1} \int \bar\rho^{(\gamma-1)(N+2)+1} G_{N+1}^2 + \bar\rho^{(\gamma-1)N +1} (G_{N}^2 + H_{N,4}^2) \,dx '\\
&+  \alpha^{5-6\gamma+r_1} \tilde\alpha^2 \int \bar\rho^{(\gamma-1)N+1} \Big[(x\partial_x^{N} + (N+1)\partial_x^{N-1})\big(x\eta(1+\eta)\big) \Big]^2\,dx  \Big\}  \\
\le &  C [1+(2\bar\mu+\bar\lambda)^2 ] \mathcal{E}_{N+2} \le C [1+ (2\bar\mu+\bar\lambda)^2] \mathcal{E}_{in}
\end{align*}
for any fixed $r_1 \le \min\{3\gamma-4,1\}$ and any $1 \le N \le N_0-2$. This proves \eqref{ineq-etatt2}. 

Moreover, Bochner's inequality yields that for any  $1 \le N \le N_0-1$,
\begin{align*}
&\int\bar\rho^{(\gamma-1)(N+1)+1} x^4 (\partial_x^{N+1}\eta)^2\,dx\\
=&\int \bar\rho^{(\gamma-1)(N+1)+1} x^4 \Big(\partial_x^{N+1}\eta_0+\int_0^\tau \partial_x^{N+1}\eta_\tau\,d\tau'\Big)^2\,dx \\
\le & C\int \bar\rho^{(\gamma-1)(N+1)+1} x^4 (\partial_x^{N+1}\eta_0 )^2\,dx + C\Big(\int_0^\tau \norm{\bar\rho^{(\gamma-1)\frac{(N+1)}{2}+\frac{1}{2}}x^2 \partial_x^{N+1} \eta_\tau}{L_2} \,d\tau'\Big)^2,
\end{align*}
where $\eqref{ineq-total}_1$ together with the Hardy inequality yields that 
\begin{align*}
&\int\bar\rho^{(\gamma-1)(N+1)+1} x^4 (\partial_x^{N+1}\eta_\tau)^2\,dx\\
\le& C\alpha^{-\min\{3\gamma-3,1\}} \mathcal{E}_{N_0} ~(\text{with $r_1=\min\{3\gamma-3,1\}-1$ in $\mathcal{E}_{N_0}$})\\
\le & C\alpha^{-\min\{3\gamma-3,1\}} \mathcal{E}_{in}.
\end{align*}
This along with the rapid growth of $\alpha$ in \eqref{ineq-alpha1} leads to
\begin{equation*}
 \int\bar\rho^{(\gamma-1)(N+1)+1} x^4 (\partial_x^{N+1}\eta)^2\,dx \le  C  \mathcal{E}_{in} + C \Big( \int_0^\tau \alpha^{-\min\{3\gamma-3,1\}/2}\,d\tau' \Big)^2 \cdot\mathcal{E}_{in} \le C  \mathcal{E}_{in}.
\end{equation*}
And thereby for any $r_2<0$, we have
\begin{align*}
    &\int _0^T \alpha^{2r_2}\int\bar\rho^{(\gamma-1)(N+1)+1} x^4 (\partial_x^{N+1}\eta)^2\,dx d\tau\\
    \le & C\int _0^T \alpha^{2r_2}\,d\tau  \mathcal{E}_{in} \le C\mathcal{E}_{in},
\end{align*}
This implies \eqref{ineq-eta1}. 
\eqref{ineq-H2} just follows from the Hardy inequality for $N_0 \geq 4 + \lceil \frac{1}{\gamma-1} \rceil$. 
\end{proof}

Now we are ready to verify the a priori assumptions \eqref{def-apriori-assumx}, \eqref{def-apriori-assumt} and \eqref{def-apriori-assum1}. It then follows from \lemref{lem-Hardy} and \lemref{lem-embedding} that for $\gamma>1$ and $2 \le k \le N_0-3$, 
\begin{align} \label{ineq-assumex}
&\begin{aligned}
    &\sup_{0 \le \tau \le T} \norm{ \bar\rho^{( \gamma - 1 )( k - \frac{3}{2})} \partial_x^{k} \eta }{\Lnorm{\infty}}^2 \\
   \le & C \sup_{0 \le \tau \le T} \int \bar\rho^{ ( \gamma - 1 )(2 k - 2)} (\partial_x^{k+1} \eta)^2 + \bar\rho^{ ( \gamma - 1 )(2 k - 4)} (\partial_x^{k} \eta)^2 \,dx\\ 
\le & C \sup_{0 \le \tau \le T}  \int \bar\rho^{ ( \gamma - 1 )(2 k - 4+2(N_0-k))} x^4 \sum_{i=k}^{N_0}  (\partial_x^{i} \eta)^2 \,dx \\
\le & C \sup_{0 \le \tau \le T} \int \bar\rho^{( \gamma - 1 )N_0+1} x^4 \sum_{i=0}^{N_0} (\partial_x^{i} \eta)^2 \,dx \\
\le & C \sum_{i=0}^{N_0} \Big[ \sup_{0 \le \tau \le T} \int \bar\rho^{( \gamma - 1 )N_0+1} x^4  (\partial_x^{i} \eta_0)^2 \,dx + \Big(\int_0^T \norm{\bar\rho^{(\gamma-1)\frac{N_0}{2}+\frac{1}{2}}x^2 \partial_x^{i} \eta_\tau}{L_2} \,d\tau'\Big)^2 \Big]\\
\le & C\mathcal{E}_{in} + C\Big(\int_0^T \alpha^{-\min\{3\gamma-3,2\}/2}\mathcal{E}_{N_0}^{1/2}\,d\tau'\Big)^2~(\text{with $r_1=\min\{3\gamma-4,1\}$ in $\mathcal{E}_{N_0}$ } )\\
\le & C\mathcal{E}_{in},
\end{aligned}\\
\intertext{where
$N_0 \geq 4 + \lceil \frac{1}{\gamma-1} \rceil$ are applied in 3rd inequality. And in the same manner,  }
&\begin{aligned} \label{ineq-assumet}
    &\sup_{0 \le \tau \le T} \norm{ \bar\rho^{( \gamma - 1 )( k - \frac{3}{2})} \partial_x^{k} \eta_\tau }{\Lnorm{\infty}}^2 \\
\le & C \sup_{0 \le \tau \le T} \int \bar\rho^{( \gamma - 1 )N_0+1} x^4 \sum_{i=1}^{N_0} (\partial_x^{i} \eta_\tau)^2 \,dx \le   C\mathcal{E}_{N_0}(T)~(\text{with $ r_1=-1$ in $\mathcal{E}_{N_0}$ } )\\
\le &C\mathcal{E}_{in},
\end{aligned}\\
\intertext{and}
&\begin{aligned} \label{ineq-assume1}
    &\sup_{0 \leq \tau \leq T} \alpha^{\min\{\frac{3\gamma-3}{2},1 \}} \norm{ (\eta_\tau,\eta_{x\tau}) }{\Lnorm{\infty}}^2 \\
\le & C \sup_{0 \le \tau \le T} \alpha^{\min\{\frac{3\gamma-3}{2},1 \}} \int \bar\rho^{( \gamma - 1 )N_0+1} x^4 \sum_{i=1}^{N_0} (\partial_x^{i} \eta_\tau)^2 \,dx \\
\le &   C\mathcal{E}_{N_0}(T)~(\text{with $ r_1=\min\{3\gamma-4,1\}$ in $\mathcal{E}_{N_0}$ } ) \le C\mathcal{E}_{in}.
\end{aligned}
\end{align}   
A similar argument together with the standard Sobolev embedding $H^1(I) \hookrightarrow L^\infty(I) $ yields 
\begin{equation*}
\begin{aligned}
   &\sup_{0 \le \tau \le T} \norm{ (\eta, \eta_x) }{\Lnorm{\infty}}^2 + \sup_{0 \le \tau \le T} \norm{ \bar\rho^{( \gamma - 1 )(N_0-\frac{7}{2})} x \partial_x^{N_0-2} \eta }{\Lnorm{\infty}}^2  \le C\mathcal{E}_{in},\\
   &\sup_{0 \le \tau \le T} \norm{ (\eta_\tau, \eta_{x\tau}) }{\Lnorm{\infty}}^2 + \sup_{0 \le \tau \le T} \norm{ \bar\rho^{( \gamma - 1 )(N_0-\frac{7}{2})} x \partial_x^{N_0-2} \eta_\tau }{\Lnorm{\infty}}^2  \le C\mathcal{E}_{in}.
\end{aligned}
\end{equation*}  
This verifies \eqref{def-apriori-assumx}, \eqref{def-apriori-assumt}, and  \eqref{def-apriori-assum1} for small enough $\mathcal{E}_{in}$. 

Based on the energy estimates established above, the local well-posedness of strong solutions to equation \eqref{eq-Lag} can be established by adapting the framework of Li-Wang-Xin \cite{Li.W.X2025}, at least for sufficiently small initial energy $\mathcal{E}_{in}$. More precisely, one can construct an iterative sequence ${\eta^{(n)}}$ by successively solving a linearized degenerate parabolic problem using a refined Galerkin method analogous to that introduced in \cite{Li.W.X2025}. To this end, one considers a weighted $H^1$ space with a weight function capturing the degeneracy with an associated degenerate-singular elliptic operator. Such an operator yields a set of eigenfunctions that form an orthogonal basis of the weighted space.
Using this basis, each $\eta^{(n)}$ can be solved and shown to satisfy uniform energy estimates analogous to those in \propref{prop-basic} and \propref{prop-high}. It should be noted that, thanks to the reflection invariance of the equation and the even symmetry of the initial data (see \rmkref{rmk-trace}), every $\eta^{(n)}$ preserves the even symmetry, thereby all boundary terms arising from integration by parts at the origin vanish. These uniform estimates imply that the sequence $\{\eta^{(n)}\}$ is compact in an appropriate weighted energy space. Extracting a convergent subsequence yields a weak solution of the original equation \eqref{eq-Lag}. By further applying weighted interpolation inequalities, the convergence can be improved, finally yielding a local strong solution with the desired regularity.

Then the global existence of strong solutions can be concluded through a standard continuation argument from \eqref{ineq-total}, \eqref{ineq-etatt1}, and \eqref{ineq-etatt2}. 
Therefore, the proof \autoref{thm-stab} is completed.

\subsection{The inviscid limit} 

Noting that the energy estimates \eqref{ineq-energy1} are uniform with respect to bounded $\bar\mu$ and $2\bar\mu+3\bar\lambda$ (see \rmkref{rmk-visc2}), the inviscid limit, \eqref{conv}, then can be easily verified.
Now to prove \autoref{thm-vanish}, it remains to derive the convergence rate \eqref{ineq-conv1}.

Let $\eta^{\bar\mu,\bar\lambda}$ be the strong solution to the problem \eqref{eq-Lag} with initial data \eqref{data-vanish} and $\eta^0$ the solution to the same problem with $\bar\mu=2\bar\mu + 3\bar\lambda=0$ and with the same initial data. 
Based on \lemref{lem-check-alpha}, by setting 
\begin{equation}
    \label{def-xi}
    \xi:=\eta^{\bar\mu,\bar\lambda}-\eta^0,
\end{equation} 
we conclude the following estimates of $\xi$ which shows the convergence rates with respect to the viscosity coefficients.
\begin{lemma} \label{lem-conv-rate}
For $0 \le N \le N_0-2$, any time $T$ and any $r_1 \le \min\{3\gamma-4,1 \}$, the following estimates hold:
     \begin{equation} \label{ineq-conv-rate1}
     \begin{aligned}
        \sup_{0\le \tau\le T} (\alpha^0)^{4-3\gamma+r_1}\int \bar\rho^{(\gamma-1)N+1} x^4(\partial_x^{N+1}\xi)^2\,dx \le& C_T (2\bar\mu+\bar\lambda)^2 \mathcal{E}_{in},\\
         \sup_{0\le \tau\le T} (\alpha^0)^{1+r_1}\int \bar\rho^{(\gamma-1)N+1} x^4(\partial_x^N\xi_{\tau})^2\,dx \le & C_T(2\bar\mu+\bar\lambda)^2 \mathcal{E}_{in},\\
         \sup_{0\le \tau\le T} (\alpha^0)^{4-3\gamma+r_1}\norm{\xi}{H^1}^2 \le& C_T (2\bar\mu+\bar\lambda)^2 \mathcal{E}_{in},\\
         \sup_{0\le \tau\le T} (\alpha^0)^{1+r_1}\norm{\xi_\tau}{H^1}^2 \le &C_T (2\bar\mu+\bar\lambda)^2 \mathcal{E}_{in},\\
        \end{aligned}
     \end{equation}
\end{lemma}
\begin{proof}
Denote $\mathcal{E}_N^{\bar\mu,\bar\lambda},\mathcal{E}_N^{(1),\bar\mu,\bar\lambda},\mathcal{E}_N^{(2),\bar\mu,\bar\lambda},\mathcal{D}_N^{\bar\mu,\bar\lambda}$ (resp. $\mathcal{E}_N^0,\mathcal{E}_N^{(1),0},\mathcal{E}_N^{(2),0},\mathcal{D}_N^0$) the energy and dissipation functionals defined in \eqref{def-energy} with $\eta, \alpha$ replaced by $\eta^{\bar\mu,\bar\lambda},\alpha^{\bar\mu,\bar\lambda}$(resp. $\eta^0, \alpha^0$). Denote $G_N^{\bar\mu,\bar\lambda}$ (resp. $G_N^0; \check G_N$) the function $G_N$ with $\eta$ replaced by $\eta^{\bar\mu,\bar\lambda}$ (resp. $\eta^0; \xi$).
Define
\begin{equation*} 
\begin{aligned}
		\mathcal{\check E}_N(\tau):=\sum_{j=0}^N & \Bigl( (\alpha^0)^{1+r_1} \int \bar\rho^{(\gamma-1)j+1} x^4(\partial_x^j\xi_{\tau})^2\,dx\\
		&+(\alpha^0)^{4-3\gamma+r_1}\int\bar\rho^{(\gamma-1)(j+1)+1} \bigl( x^4 (\partial_x^{j+1}\xi)^2 + x^2(\partial_x^j\xi)^2 \bigr)\,dx \Bigr).
	\end{aligned}
\end{equation*}
Then with the definition, \eqref{def-xi}, $\xi$ satisfies
\begin{equation} \label{eq-xi}
\begin{aligned}
&x \xi_{\tau\tau} +  \frac{\alpha^0_\tau}{\alpha^0} x\xi_\tau - c_1(\alpha^0)^{3-3\gamma}  (1+\eta^0+\eta^{\bar\mu,\bar\lambda}) x\xi \\
&+  A(\alpha^0)^{3-3\gamma} \Big[\frac{(1+\eta^{\bar\mu,\bar\lambda})^2}{\bar\rho}(\bar\rho^\gamma\mathfrak{Q}^{\bar\mu,\bar\lambda} )_x  - \frac{(1+\eta^0)^2}{\bar\rho} (\bar\rho^\gamma\mathfrak{Q}^0)_x  \Big]\\
 = & -\frac{2\bar\mu+\bar\lambda}{\gamma} (\alpha^{\bar\mu,\bar\lambda})^{2-3\gamma} \frac{(1+\eta^{\bar\mu,\bar\lambda})^2}{\bar\rho} (\bar\rho^\gamma\mathfrak{Q^{\bar\mu,\bar\lambda}_\tau})_x \\
 &- 4\bar\mu (\alpha^{\bar\mu,\bar\lambda})^{2-3\gamma} \dfrac{(1+\eta^{\bar\mu,\bar\lambda})\eta_\tau^{\bar\mu,\bar\lambda}}{\bar\rho} [\bar\rho^\gamma ( \mathfrak{Q}^{\bar\mu,\bar\lambda} + 1 ) ]_x \\
 & - \Big(\frac{\alpha_\tau^{\bar\mu,\bar\lambda}}{\alpha^{\bar\mu,\bar\lambda}}-\frac{\alpha_\tau^0}{\alpha^0} \Big) x\eta^{\bar\mu,\bar\lambda}_\tau + c_1\big[(\alpha^{\bar\mu,\bar\lambda})^{2-3\gamma} \tilde\alpha^{\bar\mu,\bar\lambda} -(\alpha^0)^{3-3\gamma}\big](1+\eta^{\bar\mu,\bar\lambda}) x\eta^{\bar\mu,\bar\lambda} \\
 &-A \big[(\alpha^{\bar\mu,\bar\lambda})^{2-3\gamma} \tilde\alpha^{\bar\mu,\bar\lambda} -(\alpha^0)^{3-3\gamma}\big] \frac{(1+\eta^{\bar\mu,\bar\lambda})^2}{\bar\rho} (\bar\rho^\gamma\mathfrak{Q}^{\bar\mu,\bar\lambda})_x . 
\end{aligned}
\end{equation}    
with the initial data 
\begin{equation} \label{data-xi}
\begin{aligned}
&\bar\rho^{(\gamma-1)\frac{N+1}{2}+\frac{1}{2}}x^2\partial_x^{N+1}\xi(x,0)=\bar\rho^{(\gamma-1)\frac{N+1}{2}+\frac{1}{2}}x\partial_x^{N}\xi(x,0)=0,\\
&\bar\rho^{(\gamma-1)\frac{N}{2}+\frac{1}{2}}x^2\partial_x^{N}\xi_\tau(x,0)=0,    
\end{aligned}
\end{equation}
for $0 \le N \le N_0-2$. Since \eqref{def-apriori-assumx}, \eqref{def-apriori-assumt} and \eqref{def-apriori-assum1} hold for both $\eta^{\bar\mu,\bar\lambda}$ and $\eta^0$, then for any $ \tau \in (0,T) $, 
\begin{align}
	\label{ineq-xi-infx}
    &\begin{aligned}
    \sup_{0 \leq \tau \leq T}  \norm{ (\xi,\xi_x) }{\Lnorm{\infty}}^2 &+ \sup_{0 \leq \tau \leq T}\sum_{k=2}^{N_0-3} \norm{  \bar\rho^{( \gamma - 1 )( k - \frac{3}{2})} \partial_x^{k} \xi }{\Lnorm{\infty}}^2 \\
    &+ \sup_{0 \leq \tau \leq T} \norm{  \bar\rho^{( \gamma - 1 )(N_0 - \frac{7}{2})} x\partial_x^{N_0-2} \xi }{\Lnorm{\infty}}^2  < C\epsilon,
    \end{aligned}\\
&\begin{aligned} \label{ineq-xi-inft}
    \sup_{0 \leq \tau \leq T}  \norm{ (\xi_\tau,\xi_{x\tau}) }{\Lnorm{\infty}}^2 &+ \sup_{0 \leq \tau \leq T}\sum_{k=2}^{N_0-3} \norm{  \bar\rho^{( \gamma - 1 )( k - \frac{3}{2})} \partial_x^{k} \xi_\tau }{\Lnorm{\infty}}^2 \\
    &+ \sup_{0 \leq \tau \leq T} \norm{  \bar\rho^{( \gamma - 1 )(N_0-\frac{7}{2})} x\partial_x^{N_0-2} \xi_\tau }{\Lnorm{\infty}}^2  < C\epsilon,
\end{aligned}\\
\intertext{and}
&\begin{aligned} \label{ineq-xi-inf1}
    \sup_{0 \leq \tau \leq T} \alpha^{\min\{ (3\gamma-3)/2, 1\}} \norm{ (\xi_\tau,\xi_{x\tau}) }{\Lnorm{\infty}}   < \epsilon.
\end{aligned}
\end{align}
Similarly, we apply the operator $x\partial_x^N+(N+1)\partial_x^{N-1}$ for $0 \le N \le N_0-2$ to \eqref{eq-xi} and get
\begin{equation} \label{eq-xi-high}
\begin{aligned}
& [x\partial_x^{N} + (N+1) \partial_x^{N-1}] (x\xi_{\tau\tau})+ \frac{\alpha^0_\tau}{\alpha^0}[x\partial_x^{N} + (N+1) \partial_x^{N-1}] (x\xi_{\tau})\\ &
- c_1(\alpha^0)^{3-3\gamma}    [x\partial_x^{N} + (N+1) \partial_x^{N-1}] \Big( (1+\eta^0+\eta^{\bar\mu,\bar\lambda}) x \xi \Big)\\ &
+ A(\alpha^0)^{3-3\gamma}  [x\partial_x^{N} + (N+1) \partial_x^{N-1}] \Bigg\{ \frac{(1+\eta^{\bar\mu,\bar\lambda})^2}{\bar\rho}(\bar\rho^\gamma\mathfrak{Q}^{\bar\mu,\bar\lambda})_x \\
& \qquad- \frac{(1+\eta^0)^2}{\bar\rho}(\bar\rho^\gamma\mathfrak{Q}^0)_x  \Bigg\}\\ 
 = & -\frac{2\bar\mu+\bar\lambda}{\gamma} (\alpha^{\bar\mu,\bar\lambda})^{2-3\gamma} [x\partial_x^{N} + (N+1) \partial_x^{N-1}]  \left\{ \frac{(1+\eta^{\bar\mu,\bar\lambda})^2}{\bar\rho} ( \bar\rho^\gamma \mathfrak{Q}_\tau^{\bar\mu,\bar\lambda} )_x  \right\}\\&
 - 4\bar\mu  (\alpha^{\bar\mu,\bar\lambda})^{2-3\gamma} [x\partial_x^{N} + (N+1) \partial_x^{N-1}] \left\{ \frac{(1+\eta^{\bar\mu,\bar\lambda})\eta_\tau^{\bar\mu,\bar\lambda}}{\bar\rho }[\bar\rho^\gamma ( \mathfrak{Q}^{\bar\mu,\bar\lambda} + 1 ) ]_x  \right\}\\
  & - \Big(\frac{\alpha_\tau^{\bar\mu,\bar\lambda}}{\alpha^{\bar\mu,\bar\lambda}}-\frac{\alpha_\tau^0}{\alpha^0} \Big) [x\partial_x^{N} + (N+1) \partial_x^{N-1}] (x\eta^{\bar\mu,\bar\lambda}_\tau)\\
  &+ c_1\big[(\alpha^{\bar\mu,\bar\lambda})^{2-3\gamma} \tilde\alpha^{\bar\mu,\bar\lambda} -(\alpha^0)^{3-3\gamma}\big][x\partial_x^{N} + (N+1) \partial_x^{N-1}] \Big((1+\eta^{\bar\mu,\bar\lambda}) x\eta^{\bar\mu,\bar\lambda}\Big)\\
 &-A \big[(\alpha^{\bar\mu,\bar\lambda})^{2-3\gamma} \tilde\alpha^{\bar\mu,\bar\lambda} -(\alpha^0)^{3-3\gamma}\big] [x\partial_x^{N} + (N+1) \partial_x^{N-1}] \left\{\frac{(1+\eta^{\bar\mu,\bar\lambda})^2}{\bar\rho} (\bar\rho^\gamma\mathfrak{Q}^{\bar\mu,\bar\lambda})_x \right\}. 
\end{aligned} 
\end{equation}
Multiply \eqref{eq-xi-high} by $(\alpha^0)^{1+r_1}\bar\rho^{(\gamma-1)N+1} [x\partial_x^{N} + (N+1) \partial_x^{N-1}](x\xi_\tau)$ and integrate, then we arrive at 
\begin{align*}
& \frac{d}{d \tau} \frac{1}{2}  (\alpha^0)^{1 + r_1} \int \bar\rho^{ (\gamma - 1)N + 1}  \check G_{N-1,\tau} ^2 \, dx\\ 
+ & \frac{1 - r_1}{2} (\alpha^0)^{r_1} \alpha^0_\tau \int \bar\rho^{ (\gamma - 1)N + 1 }  \check G_{N-1,\tau} ^2 \,dx\\
= & c_1 (\alpha^0)^{4 - 3\gamma + r_1}  \int \bar\rho^{ (\gamma - 1) N + 1 } \check G_{N-1,\tau} \cdot [x\partial_x^{N} + (N+1) \partial_x^{N-1}] \Big( (1+\eta^0+\eta^{\bar\mu,\bar\lambda}) x\xi \Big)\, dx\\
&-  A (\alpha^0)^{4 - 3 \gamma + r_1}  \int \bar\rho^{(\gamma - 1 ) N + 1} \check G_{N-1,\tau}  \\
&\qquad\cdot [x\partial_x^{N} + (N+1) \partial_x^{N-1}] \left\{ \frac{(1+\eta^{\bar\mu,\bar\lambda})^2}{\bar\rho}(\bar\rho^\gamma\mathfrak{Q}^{\bar\mu,\bar\lambda})_x  - \frac{(1+\eta^0)^2}{\bar\rho}(\bar\rho^\gamma\mathfrak{Q}^0)_x \right\} \, dx \\
& -  \frac{2\bar\mu + \bar\lambda }{\gamma} (\alpha^{\bar\mu,\bar\lambda})^{2 - 3 \gamma} (\alpha^0)^{1+r_1} \int \bar\rho^{(\gamma - 1) N + 1} \check G_{N-1,\tau} \\
&\qquad \cdot [x\partial_x^{N} + (N+1) \partial_x^{N-1}] \left\{ \frac{(1 + \eta^{\bar\mu,\bar\lambda})^2}{\bar\rho} (\rhobar^\gamma \Qfrak_\tau^{\bar\mu,\bar\lambda})_x   \right\} \, dx \\ &
- 4 \bar\mu (\alpha^{\bar\mu,\bar\lambda})^{2 - 3 \gamma} (\alpha^0)^{1+r_1}  \int \bar\rho^{(\gamma - 1) N + 1}\check G_{N-1,\tau} \\
& \qquad\cdot [x\partial_x^{N} + (N+1) \partial_x^{N-1}] \left\{\frac{( 1 + \eta^{\bar\mu,\bar\lambda}) \eta_\tau^{\bar\mu,\bar\lambda} }{\bar\rho } [ \rhobar^\gamma (\Qfrak^{\bar\mu,\bar\lambda}+1)]_x  \right \} \, dx \\
& - \Big(\frac{\alpha_\tau^{\bar\mu,\bar\lambda}}{\alpha^{\bar\mu,\bar\lambda}}-\frac{\alpha_\tau^0}{\alpha^0} \Big) (\alpha^0)^{1 + r_1} \int \bar\rho^{ (\gamma - 1)N + 1}  \check G_{N-1,\tau} [x\partial_x^{N} + (N+1) \partial_x^{N-1}] (x\eta^{\bar\mu,\bar\lambda}_\tau)\, dx \\
  &+ c_1\big[(\alpha^{\bar\mu,\bar\lambda})^{2-3\gamma} \tilde\alpha^{\bar\mu,\bar\lambda} -(\alpha^0)^{3-3\gamma}\big] (\alpha^0)^{1 + r_1} \int \bar\rho^{ (\gamma - 1)N + 1}  \check G_{N-1,\tau} \\
& \qquad\cdot [x\partial_x^{N} + (N+1) \partial_x^{N-1}] \Big((1+\eta^{\bar\mu,\bar\lambda}) x\eta^{\bar\mu,\bar\lambda}\Big)\, dx \\
 &-A \big[(\alpha^{\bar\mu,\bar\lambda})^{2-3\gamma} \tilde\alpha^{\bar\mu,\bar\lambda} -(\alpha^0)^{3-3\gamma}\big] (\alpha^0)^{1 + r_1} \int \bar\rho^{ (\gamma - 1)N + 1}  \check G_{N-1,\tau} \\
& \qquad\cdot [x\partial_x^{N} + (N+1) \partial_x^{N-1}] \left\{\frac{(1+\eta^{\bar\mu,\bar\lambda})^2}{\bar\rho} (\bar\rho^\gamma\mathfrak{Q}^{\bar\mu,\bar\lambda})_x \right\}\, dx \\
=:  & \check L_1 + \check L_2 + \check L_3 + \check L_4 + \check L_5 + \check L_6 + \check L_7.
\end{align*}

Likewise,
\begin{align*}
   &(x\partial_x^{N}+(N+1)\partial_x^{N-1})\left\{ (\bar\rho^\gamma\mathfrak{Q}^{\bar\mu,\bar\lambda})_x \frac{(1+\eta^{\bar\mu,\bar\lambda})^2}{\bar\rho} - (\bar\rho^\gamma\mathfrak{Q}^0)_x \frac{(1+\eta^0)^2}{\bar\rho}\right\}
	\\
	= &  \frac{1}{\rhobar^{(\gamma-1)N+1}} \partial_x \big\{ -\gamma \rhobar^{(\gamma-1)(N+1)+1} \big[ (1+\eta^{\bar\mu,\bar\lambda})^{-2\gamma+2} (1+\eta^{\bar\mu,\bar\lambda}+x\eta_x^{\bar\mu,\bar\lambda})^{-\gamma-1}  G_N^{\bar\mu,\bar\lambda}   \\
    &\qquad - (1+\eta^0)^{-2\gamma+2} (1+\eta^0+x\eta_x^0)^{-\gamma-1}  G_N^0  \big] \big\} + \check H_{N,2}\\
    =&  \frac{1}{\rhobar^{(\gamma-1)N+1}} \partial_x \big\{  -\gamma\rhobar^{(\gamma-1)(N+1)+1} (1+\eta^{\bar\mu,\bar\lambda})^{-2\gamma+2} (1+\eta^{\bar\mu,\bar\lambda}+x\eta_x^{\bar\mu,\bar\lambda})^{-\gamma-1}  \check G_N    \big\} \\
    &+ \check H_{N,5}+\check H_{N,2},
\end{align*}
where 
\begin{align*}
\check H_{N,5}= & \frac{1}{\rhobar^{(\gamma-1)N+1}} \partial_x \big\{  -\gamma\rhobar^{(\gamma-1)(N+1)+1}  \big[(1+\eta^{\bar\mu,\bar\lambda})^{-2\gamma+2} (1+\eta^{\bar\mu,\bar\lambda}+x\eta_x^{\bar\mu,\bar\lambda})^{-\gamma-1} \\
&- (1+\eta^0)^{-2\gamma+2} (1+\eta^0+x\eta_x^0)^{-\gamma-1} \big]  G_N^0 \big\},
\end{align*}
and
\begin{align*}
    \check H_{N,2}= &  \frac{1}{\rhobar^{(\gamma-1)N+1}}  \partial_x \Big\{ \bigl[\rhobar^{(\gamma-1)(N+1)+1}   (x \partial_x^{N}+N \partial_x^{N-1}) \Qfrak^{\bar\mu,\bar\lambda} \bigr] (1 + \eta^{\bar\mu,\bar\lambda})^2 \Big\} \\
			& + \frac{\gamma}{\rhobar^{(\gamma-1)N+1}}  \partial_x \Big\{  \rhobar^{(\gamma-1)(N+1)+1}  (1+\eta^{\bar\mu,\bar\lambda})^{-2\gamma+2} (1+\eta^{\bar\mu,\bar\lambda}+x\eta_x^{\bar\mu,\bar\lambda})^{-\gamma-1} G_N^{\bar\mu,\bar\lambda}   \Big\}   \\ 
             & -\frac{1}{\rhobar^{(\gamma-1)N+1}}  \partial_x \Big\{ \bigl[\rhobar^{(\gamma-1)(N+1)+1}   (x \partial_x^{N}+N \partial_x^{N-1}) \Qfrak^0 \bigr] (1 + \eta^0)^2 \Big\} \\
			& -  \frac{\gamma}{\rhobar^{(\gamma-1)N+1}}  \partial_x \Big\{  \rhobar^{(\gamma-1)(N+1)+1}  (1+\eta^0)^{-2\gamma+2} (1+\eta^0+x\eta_x^0)^{-\gamma-1} G_N^0   \Big\}   \\ 
   &
   -  C (1+\eta^{\bar\mu,\bar\lambda})^2 ( \rhobar^{\gamma-1} )_x  \partial_x^{N-1} \Qfrak^{\bar\mu,\bar\lambda} - 2 \rhobar^{\gamma-1} \eta_x^{\bar\mu,\bar\lambda} (1+\eta^{\bar\mu,\bar\lambda}) (x\partial_x^{N}+N\partial_x^{N-1}) \Qfrak^{\bar\mu,\bar\lambda}
   \\
   &
   + \partial_x \Big\{ C (1+\eta^{\bar\mu,\bar\lambda})^2 \Bigl[  ( \rhobar^{\gamma-1} )_x \partial_x^{N-2} \Qfrak^{\bar\mu,\bar\lambda} + ( \rhobar^{\gamma-1} )_{xx}  \partial_x^{N-3} \Qfrak^{\bar\mu,\bar\lambda} \\
   &\qquad+ ( \rhobar^{\gamma-1} )_{xx} x \partial_x^{N-2} \Qfrak^{\bar\mu,\bar\lambda} \Bigr] 
   \\
   &\qquad  + \bigl[ x\partial_x^{N-1} + N \partial_x^{N-2} , (1+\eta^{\bar\mu,\bar\lambda})^2 \bigr] \bigl\{ \bar\rho^{\gamma-1} \Qfrak_x^{\bar\mu,\bar\lambda} + \frac{\gamma}{\gamma-1} ( \rhobar^{\gamma-1} )_x \Qfrak^{\bar\mu,\bar\lambda} \bigr\} \Big\} \\
   &+  C (1+\eta^0)^2 ( \rhobar^{\gamma-1} )_x  \partial_x^{N-1} \Qfrak^0 + 2 \rhobar^{\gamma-1} \eta_x^0 (1+\eta^0) (x\partial_x^{N}+N\partial_x^{N-1}) \Qfrak^0
   \\
   &\
   - \partial_x \Big\{ C (1+\eta^0)^2 \Bigl[  ( \rhobar^{\gamma-1} )_x \partial_x^{N-2} \Qfrak^0 + ( \rhobar^{\gamma-1} )_{xx}  \partial_x^{N-3} \Qfrak^0 + ( \rhobar^{\gamma-1} )_{xx} x \partial_x^{N-2} \Qfrak^0 \Bigr] 
   \\
   &\qquad  + \bigl[ x\partial_x^{N-1} + N \partial_x^{N-2} , (1+\eta^0)^2 \bigr] \bigl\{ \bar\rho^{\gamma-1} \Qfrak_x^0 + \frac{\gamma}{\gamma-1} ( \rhobar^{\gamma-1} )_x \Qfrak^0 \bigr\} \Big\} 
\end{align*}
is the lower order term. Taylor's expansion yields that 
\begin{align*}
\abs{\check H_{N,2}}  
\le & C \big( x^2 \abs{\partial_x^{N}\xi} + x \abs{\partial_x^{N-1}\xi}  + \abs{\partial_x^{N-2}\xi} + \abs{\partial_x^{N-3}\xi} \big)\\
& +C \epsilon  \sum_{j=2}^{N+1} \bar\rho^{(\gamma-1)(j-N-\frac{1}{2})} ( x^2 \abs{ \partial_x^j \xi } + x \abs{ \partial_x^{j - 1} \xi } ) \\
	    & 
	    +C \epsilon  \sum_{j=2}^{N} \bar\rho^{(\gamma-1)(j-N+\frac{1}{2})} ( x \abs{ \partial_x^j \xi } +  \abs{ \partial_x^{j - 1} \xi } ) \\
        & + C (x\abs{\partial_x\xi} + \abs{\xi}) \big(\abs{x\eta_{xx}^{\bar\mu,\bar\lambda}}+\abs{x\eta_{xx}^0}+\abs{\eta_{x}^{\bar\mu,\bar\lambda}}+\abs{\eta_{x}^0} \big) \\
        & \qquad\cdot \big(x^2\abs{\partial_x^{N+1}\eta^{\bar\mu,\bar\lambda}} + x^2\abs{\partial_x^{N+1}\eta^{0}} \big) ,\\
\abs{\check H_{N,5}} \le & C\bar\rho^{\gamma-1}(\abs{\xi}+\abs{x\xi_x})\abs{G_{N+1}^0}+ C(\abs{\xi}+\abs{x\xi_x})\abs{G_{N}^0} \\
&+ C\bar\rho^{\gamma-1}(\abs{\xi_x}+\abs{x\xi_{xx}})\abs{G_{N}^0}.
\end{align*}
Following an analogous argument as in \eqref{ineq-assumex}, we can derive
\begin{align*}
&\sup_{0\le\tau\le T} (\alpha^0)^{\min\{5-3\gamma, 0\}}\norm{\bar\rho^{(\gamma-1)N}G_{N+1}^0}{L^\infty}^2 \\
\le & C \sup_{0\le\tau\le T} (\alpha^0)^{\min\{5-3\gamma, 0\}} \int \bar\rho^{(\gamma-1)(2N+1)} (G_{N+2}^0)^2+ \bar\rho^{(\gamma-1)(2N-1)}  (G_{N+1}^0)^2\,dx\\
\le & C \sum_{i=N}^{N_0-2}\sup_{0\le\tau\le T} (\alpha^0)^{\min\{5-3\gamma, 0\}} \int \bar\rho^{(\gamma-1)(2N_0-3)} (G_{i+2}^0)^2\,dx\\
\le & C \mathcal{E}_{N_0}^0 ~(\text{with $r_1=\min\{3\gamma-4,1\}$ in $\mathcal{E}_{N_0}^0$ } )\le C\mathcal{E}_{in}  < \epsilon
\end{align*}
for $0 \le N \le N_0-2$. Similarly, we have
\begin{align*}
\sup_{0\le\tau\le T}(\alpha^0)^{\min\{5-3\gamma, 0\}}\norm{\bar\rho^{(\gamma-1)(N-1)}G_{N}^0}{L^\infty}^2 < \epsilon,        \\
\sup_{0\le\tau\le T} \norm{\bar\rho^{(\gamma-1)(N_0-\frac{5}{2})} x^2 (\partial_x^{N_0-1}\eta^{\bar\mu,\bar\lambda},\partial_x^{N_0-1}\eta^0)}{L^\infty}^2 < \epsilon,
\end{align*}  
for small $\mathcal{E}_{in}$ and $0 \le N \le N_0-2$. It then follows from integration by parts and the Hardy inequality that 
\begin{align*}
   \int_0^\tau \check L_2 \,d\tau' 
   \le & - \frac{A \gamma}{4}  (\alpha^0)^{4 - 3 \gamma + r_1}  \int \rhobar^{(\gamma-1)(N+1)+1} \check G_N^2 \, dx  \\
    & - \frac{A \gamma}{4}  \int_0^\tau ( -(\alpha^0)^{4 - 3 \gamma + r_1}  )_\tau \int \rhobar^{(\gamma-1)(N+1)+1}  \check G_N^2 \, dx d\tau' \\
    &+C \epsilon \int_0^\tau (\alpha^0)^{ 4 - 3 \gamma + r_1-\min\{(3\gamma-3)/2,1\}}  \int \rhobar^{(\gamma-1)(N+1)+1}  \check G_N^2 \, dx d\tau' \\
    &+C\int_0^\tau (\alpha^0)^{4-3\gamma+r_1} \Big\{\int \bar\rho^{(\gamma - 1 ) N + 1} (\check G_{N-1,\tau})^2\,dx \Big\}^{\frac{1}{2}} \\
    &\qquad\qquad\cdot\Big\{\int \bar\rho^{(\gamma - 1 ) N + 1} \big[(\check H_{N,2})^2+ (\check H_{N,5})^2 \big]\,dx \Big\}^{\frac{1}{2}}\,d\tau'\\
    & +C \int_0^\tau (\alpha^0)^{4 - 3 \gamma + r_1}  \rhobar^{(\gamma-1)(N+1)+1} \check G_N \check G_{N-1,\tau}|^{x=1}_{x=0} \,d\tau' ~~(\text{vanishes})\\
    \le & - \frac{A \gamma}{4}  (\alpha^0)^{4 - 3 \gamma + r_1} \int \rhobar^{(\gamma-1)(N+1)+1} \check G_N^2 \, dx  \\
    & - \frac{A \gamma}{4}  \int_0^\tau ( -(\alpha^0)^{4 - 3 \gamma + r_1} )_\tau \int \rhobar^{(\gamma-1)(N+1)+1}  \check G_N^2 \, dx d\tau' \\
    &+C \epsilon \int_0^\tau (\alpha^0)^{4 - 3 \gamma + r_1-\min\{(3\gamma-3)/2,1\}}  \int \rhobar^{(\gamma-1)(N+1)+1}  \check G_N^2 \, dx d\tau' \\
    & + C\int_0^\tau (\alpha^0)^{\max\{ (3-3\gamma)/2,~-1}\}\mathcal{\check E}_{N}\,d\tau'.
\end{align*}
Also,
\begin{align*}
&\int_0^\tau \check L_1 \,d\tau' \\
\le & C \int_0^\tau (\alpha^0)^{\frac{11-9\gamma}{2}+r_1} \int\bar\rho^{(\gamma-1)N+1}\big\{[x\partial_x^{N}+(N+1)\partial_x^{N-1}] \big(x(1+\eta^0+\eta^{\bar\mu,\bar\lambda})\xi\big)\big\}^2\,dx d\tau'\\
& + C\int_0^\tau (\alpha^0)^{\frac{3-3\gamma}{2}} \mathcal{\check E}_N \,d\tau',\\
\intertext{where}
&\abs{[x\partial_x^{N}+(N+1)\partial_x^{N-1}] \Big(x(1+\eta^0+\eta^{\bar\mu,\bar\lambda})\xi\Big)}\\
=&\mathcal{O}(1) \sum_{k=0}^2  x^{2-k} \Big[ \abs{\partial_x^{N-k}\xi} + \sum_{j=1}^{N-k} \Big(\abs{\partial_x^j\eta^0}+ \abs{\partial_x^j\eta^{\bar\mu,\bar\lambda}}\Big) \abs{\partial_x^{N-k-j}\xi} \Big]\\
=&\begin{cases}
  \mathcal{O}(1) \sum_{k=0}^2 x^{2-k} \Big[ \abs{\partial_x^{N-k}\xi} + \epsilon \abs{\partial_x^{N-k-1}\xi}  \\
  \qquad\qquad\qquad\quad+ \epsilon \sum_{j=2}^{N-k}\bar\rho^{(\gamma-1) ( -j+\frac{3}{2} )} \abs{\partial_x^{N-k-j}\xi}  \Big], & \text{if } 0 \le N \le N_0-3,\\
  \mathcal{O}(1) \sum_{k=0}^2 x^{2-k} \Big[ \abs{\partial_x^{N_0-k-2}\xi} + \epsilon \abs{\partial_x^{N_0-k-3}\xi}    \\
  \qquad\qquad\qquad\quad + \epsilon \sum_{j=2}^{N_0-k-3}\bar\rho^{(\gamma-1) ( -j+\frac{3}{2} )} \abs{\partial_x^{N_0-k-j-2}\xi}  \Big]\\
  \quad + \mathcal{O}(1)\epsilon\sum_{k=1}^2 x^{2-k}  \bar\rho^{(\gamma-1) ( -N_0+k+\frac{7}{2} )} \abs{\xi} \\
  \quad+ \mathcal{O}(1) x\bar\rho^{(\gamma-1) ( -N_0+\frac{7}{2} )}\abs{\xi}, & \text{if } N =N_0-2.
\end{cases}
\end{align*} 
Then the Hardy inequality shows that
\begin{align*}
\int_0^\tau \check L_1\,d\tau' \le & C \int_0^\tau (\alpha^0)^{\frac{11-9\gamma}{2}+r_1} \sum_{j=0}^{N} \int \bar\rho^{(\gamma-1)N+1} x^4 (\partial_x^{N-j}\xi)^2\,dxd\tau' \\
&+ C\int_0^\tau (\alpha^0)^{\frac{3-3\gamma}{2}} \mathcal{\check E}_N \,d\tau'\\
\le &  C\int_0^\tau (\alpha^0)^{\frac{3-3\gamma}{2}} \mathcal{\check E}_N \,d\tau'.\\
\intertext{A similar argument yields that}
\int_0^\tau \check L_3\,d\tau' \le& C \int_0^\tau (\alpha^0)^{2-3\gamma}\mathcal{\check E}_{N}d\tau'\\
&+C_T (2\bar\mu+\bar\lambda)^2 \int_0^\tau   (\alpha^{\bar\mu,\bar\lambda})^{4-6\gamma} (\alpha^0)^{-1+3\gamma+r_1}   \int \bar\rho^{(\gamma - 1) N + 1}\\
&\qquad\cdot \Big\{[x\partial_x^{N} + (N+1) \partial_x^{N-1}] \big\{  (\rhobar^\gamma \Qfrak_\tau^{\bar\mu,\bar\lambda})_x  \frac{(1 + \eta^{\bar\mu,\bar\lambda})^2}{\bar\rho} \Big\}\Big\}^2 \, dx \\
\le &C \int_0^\tau (\alpha^0)^{2-3\gamma}\mathcal{\check E}_{N}d\tau' +C_T (2\bar\mu+\bar\lambda)^2 \Big(\mathcal{D}_{N+1}^{\bar\mu,\bar\lambda}+\int_0^\tau \mathcal{E}_{N+1}^{\bar\mu,\bar\lambda}\,d\tau' \Big),\\
\int_0^\tau \check L_4\,d\tau' \le& C \int_0^\tau(\alpha^0)^{2-3\gamma}\mathcal{\check E}_{N}d\tau'\\
&+ C_T \bar\mu^2 \int_0^\tau  (\alpha^{\bar\mu,\bar\lambda})^{4-6\gamma} (\alpha^0)^{-1+3\gamma+r_1} \int \bar\rho^{(\gamma - 1) N + 1}\\
&\qquad\cdot \Big\{[x\partial_x^{N} + (N+1) \partial_x^{N-1}] \big\{  [ \rhobar^\gamma (\Qfrak^{\bar\mu,\bar\lambda}+1)]_x \frac{( 1 + \eta^{\bar\mu,\bar\lambda}) \eta_\tau^{\bar\mu,\bar\lambda} }{\bar\rho }\Big\}\Big\}^2 \, dx \\
\le &C \int_0^\tau (\alpha^0)^{2-3\gamma}\mathcal{\check E}_{N}d\tau' +C_T\bar\mu^2 \Big(\mathcal{D}_{N+1}^{\bar\mu,\bar\lambda}+\int_0^\tau \mathcal{E}_{N+1}^{\bar\mu,\bar\lambda}\,d\tau' \Big),\\
\int_0^\tau \check L_5\,d\tau' 
\le &  C_T (2\bar\mu+3\bar\lambda) \int_0^\tau (\alpha^0)^{1+r_1} \int \bar\rho^{ (\gamma - 1)N + 1}  \abs{\check G_{N-1,\tau}} \abs{ G_{N-1,\tau}^{\bar\mu,\bar\lambda}} \, dx d\tau' \\
\le &  C \int_0^\tau (\alpha^0)^{\frac{3-3\gamma}{2}} \mathcal{\check E}_N  \,d\tau' +  C_T (2\bar\mu+3\bar\lambda)^2  \int_0^\tau \mathcal{E}_N^{\bar\mu,\bar\lambda} \, d\tau', \\
\int_0^\tau \check L_6\,d\tau' 
\le &  C_T(2\bar\mu+3\bar\lambda) \int_0^\tau  [(\alpha^0)^{4-3\gamma+r_1} + (\alpha^{\bar\mu,\bar\lambda})^{2-3\gamma} (\alpha^0)^{1+r_1} ]   \int \bar\rho^{ (\gamma - 1)N + 1}  \abs{\check G_{N-1,\tau} }\\
& \qquad\cdot \abs{[x\partial_x^{N} + (N+1) \partial_x^{N-1}] \Big((1+\eta^{\bar\mu,\bar\lambda}) x\eta^{\bar\mu,\bar\lambda}\Big)} \, dx d\tau'\\
\le &  C \int_0^\tau (\alpha^0)^{\frac{3-3\gamma}{2}} \mathcal{\check E}_N  \,d\tau' + C_T (2\bar\mu+3\bar\lambda)^2 \int_0^\tau \mathcal{ E}_{N-1}^{\bar\mu,\bar\lambda}\,d\tau',\\
\int_0^\tau \check L_7\,d\tau'
\le &  C_T(2\bar\mu+3\bar\lambda) \int_0^\tau [(\alpha^0)^{4-3\gamma+r_1} + (\alpha^{\bar\mu,\bar\lambda})^{2-3\gamma} (\alpha^0)^{1+r_1} ]   \int \bar\rho^{ (\gamma - 1)N + 1}  \abs{\check G_{N-1,\tau} }\\
& \qquad\cdot \abs{[x\partial_x^{N} + (N+1) \partial_x^{N-1}] \left\{\frac{(1+\eta^{\bar\mu,\bar\lambda})^2}{\bar\rho} (\bar\rho^\gamma\mathfrak{Q}^{\bar\mu,\bar\lambda})_x \right\}} \, dx d\tau'\\
\le &  C \int_0^\tau (\alpha^0)^{\frac{3-3\gamma}{2}} \mathcal{\check E}_N  \,d\tau' + C_T (2\bar\mu+3\bar\lambda)^2 \int_0^\tau \mathcal{ E}_{N+1}^{\bar\mu,\bar\lambda}\,d\tau'.
\end{align*} 
Combining the above estimates, we conclude that
\begin{align*}
\mathcal{\check E}_{N}  \le C\int_0^\tau \alpha^{\max\{(3-3\gamma)/2,~-1\}}\mathcal{\check E}_{N}\,d\tau' + C_T(2\bar\mu+\bar\lambda)^2  \Big(\mathcal{D}_{N+1}^{\bar\mu,\bar\lambda}+  \int_0^\tau \mathcal{E}_{N+1}^{\bar\mu,\bar\lambda} \,d\tau' \Big).
\end{align*}
Noting \eqref{data-xi} and \eqref{ineq-total}, it then follows from Gr\"onwall's inequality that for any $0 \le N \le N_0-2$ and any $r_1\le \min\{3\gamma-4,1\}$
\begin{align*}
\mathcal{\check E}_{N} \le  C_T   (2\bar\mu+\bar\lambda)^2 \mathcal{E}_{in}.
\end{align*}
This finishes the proof.
\end{proof}

\vspace{0.5cm}

\noindent\textbf{Acknowledgements}.
This work is based on the Ph.D. thesis of the first author, conducted under the supervision of the second author at the Institute of Mathematical Sciences, the Chinese University of Hong Kong. Part of the work was done when the third author visited the Institute of Mathematical Sciences at The Chinese University of Hong Kong and the first author visited South China Normal University. They thank the both institutions for their support and hospitality.
The first author is supported by the Hong Kong PhD Fellowship Scheme from RGC of Hong Kong (No. PF20-48286). The second author is supported by the the General Research Fund from RGC of Hong Kong (Project Nos. 14304325 and 14301023). The third author is supported by the National Natural Science Foundation of China (No. 12471210).
\bibliographystyle{abbrv}
\bibliography{ns-affine}

\end{document}